\documentclass{article}

\pdfoutput=1
\usepackage[english]{babel}

\usepackage[letterpaper,top=2cm,bottom=2cm,left=3cm,right=3cm,marginparwidth=1.75cm]{geometry}

\usepackage{amssymb}
\usepackage{amsmath}
\usepackage{amsthm}
\usepackage{graphicx}
\usepackage[colorlinks=true, allcolors=blue]{hyperref}
\usepackage[ruled,linesnumbered]{algorithm2e}
\usepackage{algorithmic}
\usepackage{subfigure}
\usepackage{cite}
\usepackage{appendix}
\usepackage{booktabs}
\usepackage{multirow}
\usepackage{array}
\usepackage{bm}
\usepackage{newtxtext,newtxmath}
\usepackage{makecell}
\usepackage{indentfirst}
\usepackage{hyperref}
\hypersetup{colorlinks = true, urlcolor = blue}

\title{\bf Deep Reinforcement Learning in Finite-Horizon to Explore the Most Probable Transition Pathway}

\author{\bf\normalsize{
Jin Guo$^{1,}$\footnotemark[2],
Ting Gao$^{1,}$\footnotemark[1],
Peng Zhang$^{1,}$\footnotemark[3],
Jiequn Han$^{2,}$\footnotemark[4],
Jinqiao Duan$^{3,}$\footnotemark[5]
}\\[10pt]
\footnotesize{$^1$ Center for Mathematical Sciences, Huazhong University of Science and Technology, Wuhan 430074, China} \\
\footnotesize{$^2$ Center for Computational Mathematics, Flatiron Institute, New York, 10010, NY, USA}\\
\footnotesize{$^3$ Department of Mathematics and Department of Physics, Great Bay University, Dongguan, Guangdong 523000, China}
}

\date{}

\providecommand{\keywords}[1]{\textbf{\textit{Keywords:}} #1}

\begin{document}
\maketitle

\begin{abstract}
In many scientific and engineering problems, noise and nonlinearity are unavoidable, which could induce interesting mathematical problem such as transition phenomena. This paper focuses on efficiently discovering the most probable transition pathway for stochastic dynamical systems employing reinforcement learning. With the Onsager–Machlup action functional theory to quantify rare events in stochastic dynamical systems, finding the most probable pathway is equivalent to solving a variational problem on the action functional. When the action function cannot be explicitly expressed by paths near the reference orbit, the variational problem needs to be converted into an optimal control problem. First, by integrating terminal prediction into the reinforcement learning framework, we develop a Terminal Prediction Deep Deterministic Policy Gradient (TP-DDPG) algorithm to deal with the finite-horizon optimal control issue in a forward way. Next, we present the convergence analysis of our algorithm for the value function in terms of the neural network's approximation error and estimation error. Finally, we conduct various experiments in different dimensions for the transition problems in applications to illustrate the effectiveness of our algorithm.
\end{abstract}


\section{Introduction}
\setlength{\parindent}{2em} Stochastic dynamical systems are models for complex phenomena \cite{duan2015introduction} in many scientific and engineering fields, such as chemistry, biology, physics, and mathematical finance \cite{lucarini2019transitions, chung2012experimental, yang2023neural}. Under long-term noise disturbances, a dynamical system may experience enormous fluctuations, which could lead to transitions between metastable states. However, such a noise-induced transition is impossible in deterministic cases. Thus a transition phenomenon between metastable states is an important issue in stochastic dynamical systems, especially in nonlinear cases. 

Both the Onsager-Machlup theory and the Freidlin-Wentzell theory can be utilized to investigate transition phenomena in stochastic dynamical systems. The difference lies in that the Freidlin-Wentzell theory of large deviations \cite{kifer1988random,dembo2009large} is primarily a measure of the probability of rare occurrences in the presence of small noise\cite{heymann2008geometric}. In contrast, the Onsager-Machlup theory can be utilized to characterize the transition phenomena with no constraints on noise intensity. Recall that Onsager and Machlup \cite{onsager1953fluctuations} develop an action functional for diffusion processes with linear drift components and constant diffusion coefficients. Then, Tisza and Manning \cite{tisza1957fluctuations} expand the results of \cite{onsager1953fluctuations} to nonlinear diffusion equations. Next, Hara and Takahashi \cite{hara2016stochastic} extend the Onsager-Machlup functional for Brownian motion on Riemann manifolds by employing probabilistic arguments. The expression of the Onsager-Machlup functional is more general than that of the Freidlin-Wentzell functional and, as recently demonstrated by Br{\"o}cker \cite{brocker2019correct}, the Onsager-Machlup functional is more typical than the other energy functional, thus its minimizer can be rigorously interpreted as the most probable transition path of the stochastic dynamical systems. In addition, Chao and Duan \cite{chao2019onsager} study Onsager-Machlup functional under l\'evy noise and determine the most probable transition path via computing the Euler Lagrangian equation.

\footnotetext[2]{Email: \texttt{jinguo0805@hust.edu.cn}}
\footnotetext[1]{Email: 
 \texttt{tgao0716@hust.edu.cn}}
\footnotetext[3]{Email: \texttt{kazusa\_zp@hust.edu.cn}}
\footnotetext[4]{Email: \texttt{jiequnhan@gmail.com}}
\footnotetext[5]{Email: \texttt{duan@gbu.edu.cn}}

There are some numerical methods for solving the most probable transition pathway for stochastic dynamical systems. In the Onsager-Machlup framework, usually we consider the action functional as the integral of a Lagrangian to discover the most probable transition pathway via the variational principle \cite{durr1978onsager}. Hu et al. \cite{hu2022onsager} use machine learning techniques to determine the most probable transition pathway with the Euler-Lagrange equation. For small noise, Heymann and Vanden-Eijnden \cite{heymann2008geometric} construct the geometric minimum action method (gMAM) to solve for the most probable transition path. The core of the method is to pre-define a spatial curve and then lead the curve to converge to the optimal path according to certain convergence rules (difference equations). However, in some scenarios, the rate functional cannot be explicitly represented by pathways in the neighborhood of the reference orbit. To address this issue, Wei et al. \cite{wei2022optimal} discover the transition path by converting it into an optimal control problem and developing a neural network to solve it under small noise. However, this method has a high computational cost to train the neural networks. Alternatively, Chen et al. \cite{chen2023data} utilize physics-informed neural networks (PINNs) to compute the most probable transition pathway by solving the Euler-Lagrange equation. Li et al. \cite{li2021machine} propose a machine learning framework to determine the most probable pathways and address the shortcoming of the shooting method by a reformulation of the boundary value problem associated with a Hamiltonian system. Then, Chen et al. \cite{chen2023detecting} tackle the optimal control problem for detecting the most probable transition pathway with Pontryagin's Maximum Principle, and employ neural networks to approximate the optimal solution. 

Reinforcement learning is a natural choice for solving optimal control problems. For example, Zhou et al. \cite{zhou2021actor} present a numerical approach for solving high-dimensional Hamilton-Jacobi-Bellman (HJB) type elliptic partial differential equations (PDE). They employ the $It\hat{o}$ formula to the value function and combine it with the HJB equation to update the critic network and optimize the actor network by minimizing the cost function. Inspired by these, we can take the Onsager-Machlup action functional as the running cost, and treat the terminal prediction loss as the terminal cost.
 
 Reinforcement learning algorithms can be classified into two types: model-free algorithms and model-based algorithms. On the one hand, the model-free algorithm does not require the model of the environment and interacts directly with the real environment to obtain feedback. Mnih et al. \cite{mnih2013playing} propose the Deep Q-network (DQN), which is based on Q-learning and CNN network structure and utilizes Q-network rather than Q-table to choose the action that maximizes Q value in a given state. Since the DQN algorithm tends to overestimate the Q-value, Hasselt et al. \cite{van2016deep} propose the Double DQN algorithm, in which another Q-network is introduced to make decisions, so that the original Q-network only estimates the Q-value without making decisions. Silver et al. \cite{silver2014deterministic} propose the Deterministic Policy Gradient (DPG) algorithm to choose a deterministic action for the current state. It can be shown theoretically that the gradient of the deterministic policy equals the expectation of the Q-function's gradient and the deterministic policy is more efficient than the stochastic one. Lillicrap et al. \cite{lillicrap2015continuous} propose Deep Deterministic Policy Gradient (DDPG) algorithm based on DPG \cite{silver2014deterministic} by altering the Q function to Q-network. Moreover, they add a target Q network and a target policy network to improve the stability of their algorithm. Schulman et al. \cite{schulman2015trust} present the Trust Region Policy Optimization (TRPO) method, which provides a monotonic approach to policy improvement in order to assure that the new policy is not worse than the original policy. On the other hand, model-based algorithms require the simulation of the environment and obtaining feedback via interaction with the simulated environment \cite{huang2023model}. Prior to executing an action, the agent has the capability to generate predictions regarding the subsequent state and corresponding rewards. There are some typical methods such as Dyna \cite{sutton1991dyna}, Model-based policy optimization (MBPO) \cite{janner2019trust}, Model-based value expansion (MVE) \cite{feinberg2018model}, etc

Some efforts have been made to seek transition pathways via reinforcement learning in real-world applications. For example, Zhang et al. \cite{zhang2021deep} formulate the transition of a chemical reaction as a game and improve the success rate of the game to be higher than 10\% with their algorithm, which is similar to imitation learning. They update the critic network and policy network by determining the gradient of contrastive loss and minimizing the KL divergence of the target distribution and the expert-guided sample distribution respectively. Inspired by Boltzmann Generator \cite{noe2019boltzmann}, Liu et al. \cite{liu2022pathflow} utilize the normalizing flow method to establish the relationship between the hidden space and the sampling space. They evaluate the learned invertible function via the Minimum Energy Path and the Minimum Free Energy Path in order to map the linear interpolation paths in the hidden space to physically meaningful transition paths. Nagami and Schwager \cite{nagami2021hjb} estimate the pathway and minimum time for a quadrotor to pass through 12 doors employing a combination of supervised learning and reinforcement learning. They initialize the reinforcement learning neural network by employing the network parameters trained by supervised learning with the help of the HJB equation. Rose et al. \cite{rose2021reinforcement} attempt to discover an alternative sampling dynamic to obtain rare trajectories in an optimal way, and show that the formulation of regularized reinforcement learning issue is equivalent to finding the dynamics of optimal sampling rare trajectories. In our case, the transition pathway needs to reach the target state in a finite time. This motivates us to deal with transition pathway challenges by extending a model-free reinforcement learning method to a finite horizon.

There are some research works on finite-horizon issues with reinforcement learning \cite{lei2020dynamic}. Vivek and Bhatnagar \cite{vp2021finite} propose a Q-learning algorithm to address finite-horizon problems, and analyze its stability and convergence. While we add terminal prediction into the DDPG algorithm and design terminal loss as a regularization condition. Since the Hamilton-Jacobi-Bellman(HJB) equation for the finite horizon optimal control problem is time-varying, Zhao and Gan \cite{zhao2020finite} come up with a circular fixed-finite-horizon based on reinforcement learning method to solve the time-varying HJB equation. Furthermore, they experimentally verify the effectiveness of their algorithm for the finite-horizon optimal control problem of continuous-time uncertain nonlinear systems. Instead, we employ actor network to choose actions instead of assuming the optimal action has an explicit expression. Moreover, Hur\'e et al. \cite{hure2021deep} propose several algorithms based on deep learning for solving discrete finite-horizon control problems and provide convergence analysis and numerical experiments \cite{bachouch2022deep}. On the contrary, we focus on solving reinforcement learning in a forward way, which does not need to have a pre-assumed distribution on the previous transition state.

To summarize, we turn the challenge of exploring the transition path between two metastable states into a finite-horizon optimal control problem with the Onsager-Machlup functional and terminal cost minimization. The following are our primary contributions in this paper:

\begin{itemize}

\item[$\bullet$] We design a new DDPG-based algorithm, called Terminal Prediction Deep Deterministic Policy Gradient (TP-DDPG) through integrating terminal prediction with reinforcement learning for solving the finite-horizon optimal control problem in a forward way.

\item[$\bullet$] We present the convergence analysis of TP-DDPG algorithm in terms of estimation error and approximation error, as well as a convergence study of the Q value critic network and the terminal loss predicted by actor network.

\item[$\bullet$] We conduct various experiments with stochastic dynamical systems in three different dimensions to obtain the most probable transition pathways, and observe some consistency between our numerical solution and the convergence theory. Besides, our algorithm is efficient and requires lower computational cost.

\end{itemize}
The paper is arranged as follows. In section 2, we briefly introduce the Onsager-Machlup theorem and formulate the problem of minimizing the Onsager-Machlup action functional to an optimal control problem. In section 3, we set up our algorithm and present a flowchart to explain the framework of our algorithm. In section 4, we perform some convergence analysis of the algorithm TP-DDPG. In section 5, we present some numerical experiments and compare our results with the solution of the Euler-Lagrange equation. Finally, we conclude the paper with discussions in section 6.

\section{Mathematical Setup}
\setlength{\parindent}{2em} Consider the following stochastic differential equation
\begin{equation}
\begin{cases}
     dX(t)=b(X(t))dt+ \sigma(X(t)) dB(t)\,,\\
    X(0)=X_0\,, \label{sde}
\end{cases}
   \end{equation}
for $t\in [0,T]$, where $b:\mathbb{R}^d \rightarrow \mathbb{R}^d$ is a drift function, $\sigma:\mathbb{R}^d \rightarrow \mathbb{R}^{d\times u}$ is a $d\times u$ matrix-valued function and $B$ is a Brownian motion in $\mathbb{R}^u$. The well-posedness of the stochastic differential equation (\ref{sde}) has been extensively studied \cite{karatzas1991brownian}. Here we are interested in the transition phenomena between two metastable states.

In the following, we first introduce some theories on Onsager–Machlup action functional and optimal control.

\subsection{Onsager–Machlup Action Functional}

\setlength{\parindent}{2em} To investigate transition phenomena, the probability of the solution trajectories in a small tube should be estimated. When the noise intensity is a constant, the Onsager-Machlup action functional can be employed to quantify the tubular domain surrounding a reference orbit $z(t)$ on $[0,T]$. The Onsager-Machlup theory of stochastic dynamical systems (\ref{sde}) gives an estimation of the probability
\begin{equation}
  \begin{aligned}
    \mathbb{P}\{\|X-z\|_T\leq\delta\}\propto C(\delta, T)exp\{-S^{OM}_T(z,\dot{z})\}\\
=\mathbb{P}^{x_0}\{\|B^{\varepsilon}_t-x_0\|_T\leq \delta\}exp\{-S^{OM}_T(z,\dot{z})\},\label{transiton probability}
\end{aligned}
\end{equation}
where $\delta$ is positive and sufficiently small, $\mathbb{P}^{x_0}$ denotes the probability under the condition of initial position $X(0)=x_0 $, $\|\cdot\|_T $ is defined as the uniform norm depending on $T$:
$$\|\varphi(t)\|_T:=\sup_{0\leq t\leq T}|\varphi(t)|,$$
and $B^{\varepsilon}_t=\varepsilon (B_t-B_0)+x_0$ (a shifted Brownian motion with magnitude $\varepsilon$). The Onsager-Machlup action functional is
\begin{equation}
    S^{OM}_T(z,\dot{z})=\frac{1}{2}\int_0^T \big\{[\dot{z}-b(z)][\sigma\sigma^T]^{-1}[\dot{z}-b(z)]+\nabla\cdot b(z)\big\}dt. \label{S^OM}
\end{equation}

Equation (\ref{transiton probability}) indicates that for a given transition time $T$, the most probable transition path is the smooth path that maximizes the probability near the tube of the reference orbit. Therefore, determining the most probable transition path could employ the calculus of variations to find the path $X_t$ that minimizes the Onsager-Machlup action functional, which is also called the principle of least action. We can write the corresponding Lagrangian to be
\begin{equation}    
L(z,\dot{z})=\frac{1}{2} \big\{[\dot{z}-b(z)][\sigma\sigma^T]^{-1}[\dot{z}-b(z)]+\nabla\cdot b(z)\big\}.\label{lagrangian}
\end{equation}
In the classical calculus of variations framework, the corresponding Euler-Lagrange equation is
$$\frac{d}{dt}\frac{\partial}{\partial \dot{z}}L(z,\dot{z})=\frac{\partial}{\partial z}L(z,\dot{z}).$$

\subsection{Formulation of Optimal Control Problem}

\setlength{\parindent}{2em} Consider the following ordinary differential equation (ODE)
\begin{equation}
    \begin{cases}
             \dot{X}(t)=b(t, X(t))+ \sigma(X(t)) u(t),\,\,\,\,\,t\in [0,T]\,,\\ 
X(0)=X_0\in \mathcal{X}\,,\label{ode}   
    \end{cases}
\end{equation}
where $X(0)$ is the initial state, and the map $u(\cdot):[0,T] \rightarrow R^{u}$ is called a control. Here state space $\mathcal{X}$ is an open, connected and compact set in $\mathbb{R}^d$, and the control space $\mathbb{A}$ is a convex closed domain in $\mathbb{R}^{u}$. 

We then consider an optimal control problem to minimize the following cost functional
$$
\mathcal{J}[X,u]=\int_0^T f(t, X(t),u(t))dt+g(X(T))\,,
$$
where $f:[0,T]\times \mathcal{X}\times \mathbb{A}\rightarrow \mathbb{R}$ is called the running cost, $g: \mathcal{X}\rightarrow  \mathbb{R}$ is called the terminal cost.

The minimization for the Onsager–Machlup action functional (\ref{S^OM}) can be reformulated as the following constrained optimal control problem:
\begin{equation}
    \begin{cases}
        \mathop{\mbox{inf}}\limits_{u\in \mathbb{A}}\,\,\,\, \mathcal{J}[X,u]=\int_{0}^T f(t, X(t),u(t))dt+g(X(T))\,,\\
\mbox{subject\,\,to}\,\,\dot{X}(t)=b(X(t))+\sigma(X(t)) u(t)\,,\\ 
\quad\quad\quad\quad X(0)=X_{0}\,,\label{optimal control problem}
 \end{cases}
\end{equation}
where $t=0$ is the initial time and $X_{0}$ is the initial state. Here $X(t):[0,T]\rightarrow \mathcal{X}$ is the state and $u(t):[0,T]\rightarrow \mathbb{A}$ is the control. From (\ref{S^OM}) and (\ref{ode}), the running cost $f(t, X(t),u(t))$ becomes
\begin{equation}
    f(t, X(t),u(t))=\frac{1}{2}[|u(t)|^2+\nabla\cdot b(X(t))]\,.\label{OM functional}
\end{equation}
The terminal cost $g(x)$ measures the discrepancy between the learned terminal state $x$ and the target terminal state $X_T$. We define it as the $L_2$-norm Euclidean distance:
\begin{equation} 
g(x)=\lambda\|x-X_T\|_2,
\end{equation}
where $\lambda$ is a scaling parameter.

The optimal orbit $X^*$ of the problem (\ref{optimal control problem}) is the most probable transition path from the metastable state $X_{0}$ to the metastable state $X_T$ within a given time $T$ for the system (\ref{sde}). Next, we introduce reinforcement learning to solve the control problem (\ref{optimal control problem}) directly.

\section{Reinforcement Learning Framework}

\subsection{Definition}

\setlength{\parindent}{2em} To properly formulate the preceding optimal control issue (\ref{optimal control problem}) as a reinforcement learning problem, we first give the basic definition of reinforcement learning. The process of reinforcement learning involves the agent's interaction with the environment. At each time step, the agent chooses an action $a_t$ based on the current state $s_t$ of the environment. The environment, in turn, provides feedback to the agent based on the selected action $a_t$. Following this interaction, the environment transitions to a new state $s_{t+1}$, and the agent receives a reward $r_t$. We divide time $T$ into $N$ time intervals, then we have $\Delta t=\frac{T}{N}$. For the optimal control problem (\ref{optimal control problem}) of obtaining the most probable transition path, the quadruplet of reinforcement learning is defined as follows,

\begin{itemize}

\item \textbf{State space}: The state $s_t$ is defined as the coordinates of $X(t)\in\mathbb{R}^d$ at timestep $t$.

\item \textbf{Action space}: The action space is defined as $U \subset \mathbb{R}^{u}$($a_t\in U$). At timestep $k$ ($k=0,1,\cdots,N$), $a_{k}$ corresponds to the control term $u(k\Delta t)$ in (\ref{optimal control problem}).

    \item \textbf{Cost (Reward)}: In the domain of reinforcement learning, the prevalent convention is to employ the reward rather than the cost. Consequently, the objective is typically formulated as the maximization of the cumulative reward, as opposed to the minimization of the cost. It is evident that these two perspectives are fundamentally equivalent, differing only in their sign conventions. In our study, we choose the notation cost instead of reward, because that not only aligns more closely with the control literature but also better conforms to our specific problem (\ref{optimal control problem}). For consistency with the optimal control problem (\ref{optimal control problem}), we will call the cost as running cost in the following. The running cost $r_t$ at timestep $t$ is defined as
    $$r_t=f(t, s(t),a(t))\Delta t.$$
    Our goal is to minimize the summation of running cost and terminal cost.
    \item \textbf{Transition dynamics}: The transition dynamics corresponding to the optimal control problem (\ref{optimal control problem}) can be expressed as $$s_{t+1}=s_t+b(s_t)\Delta t+ \sigma(s_t) a_t\Delta t\,.$$
\end{itemize}

\begin{figure}[ht]
    \centering
    \includegraphics[scale=0.7]{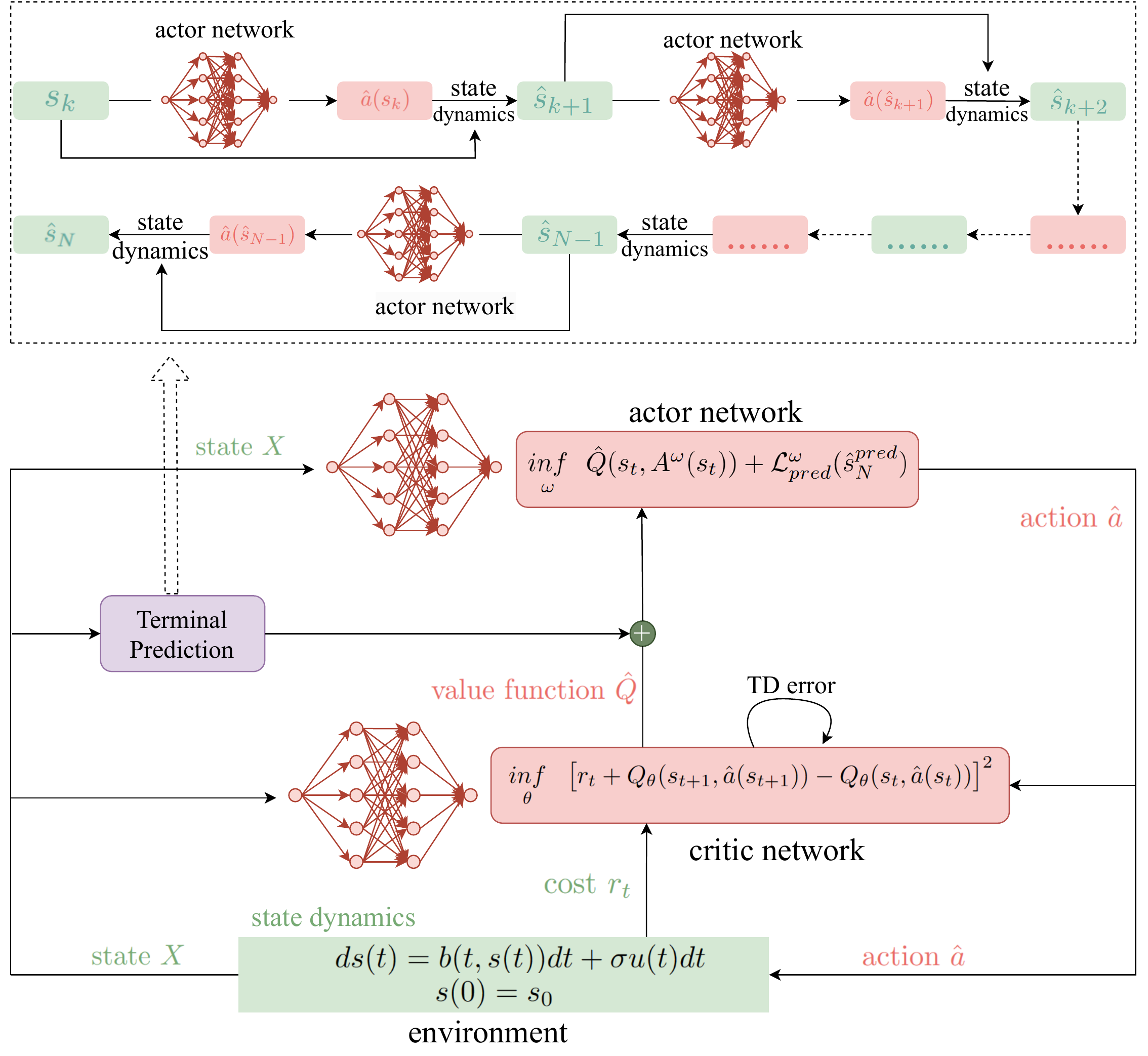}
    \caption{Flow chart of TP-DDPG.}
    \label{flow chart}
    \end{figure}

\subsection{Terminal Prediction}
In order to deal with the finite-horizon issue, we propose terminal prediction to tackle the difficulty in a forward way. As the name implies, the purpose of terminal prediction is to predict the state $\hat{s}_N^{pred}$ after $N-k$ timesteps with the help of the current actor network when the state $s_k$ of the $k$-th timestep is given. The loss of terminal predict is characterized as $\mathcal{L}_{pred}=\lambda\|\hat{s}_N^{pred}-X_T\|_2$, where $X_T$ is the target terminal point. Subsequently, we provide an elaboration on the concept behind the terminal prediction.

The $k$-th timestep state $s_k$ is input to the current actor network to yield $a_k=\hat{a}(s_k)$. Then $(s_k,\hat{a}(s_k))$ interact with the environment, resulting in the predicted next state $\hat{s}_{k+1}$ being determined by the dynamical equation (\ref{ode}). By Eulerian discretization of the state dynamics (\ref{ode}), we denote it as 
\begin{equation}
    X_{t+1}=F(X_t,u_t), \label{state dyn}
\end{equation}
then we have $\hat{s}_{k+1}=F(s_k,\hat{a}(s_k))$. Next, putting the state $\hat{s}_{k+1}$ into the current actor network to obtain $\tilde{a}_{k+1}=\hat{a}(\hat{s}_{k+1})$, then we get $\hat{s}_{k+2}=F(\hat{s}_{k+1},\tilde{a}_{k+1})=F\Big(F(s_k,\hat{a}(s_k)),\hat{a}\big(F(s_k,\hat{a}(s_k)\big)\Big)$ via the dynamic equation. Following $N-k$ timesteps, we get
$$\hat{s}_N^{pred}:=\hat{s}_N^{k,\hat{a}}=\underbrace{F\circ F \circ \cdots \circ F}_{N-k}(s_k,\hat{a}(s_k))\,,$$
where $\hat{s}_N^{k,\hat{a}}$denotes the predicted state $\hat{s}_N$ at time step $N$ derived from $s_k$ by the current actor network $\hat{a}$. In this manner, we obtain the predicted terminal state $\hat{s}_N^{k,\hat{a}}$ starting from timestep $k$ based on the current actor network. By predicting the terminal state at each step, the terminal loss compels the agent to approach the target terminal point. This prompts the actor network to update in a manner that minimizes the terminal cost, ultimately enabling the agent to successfully reach the target terminal point. We illustrate the process of terminal prediction in the upper segment of Figure \ref{flow chart}.

\subsection{Terminal Prediction Deep Deterministic Policy Gradient}

\setlength{\parindent}{2em} For the transition problem of stochastic dynamical systems, the state space and action space are continuous, Deep Deterministic Policy Gradient (DDPG) algorithm is a better candidate than other classical reinforcement learning algorithms to tackle the problem of discovering the most probable transition pathway.

Some existing reinforcement learning algorithms including DDPG are designed for the infinite problem. In our case, updating via running cost $r_t$, for $0\leq t \leq T-1$, $\hat{Q}(s_t,a_t)$ on finite horizon $[0,T]$ is the approximation of $$\sum_{k=t}^{T-1}r(s_k,a_k)\approx\int_t^Tf(X(\tau),u(\tau))d\tau\,.$$ 
In this situation, we define the $\hat{Q}_N=0$. If we solely rely on the Q-value to guide the update of the actor network, we would be unable to determine the optimal action that satisfies the optimal control problem (\ref{optimal control problem}). Because the cost  $\mathcal{J}[X_t,u_t]=\int_t^Tf(X(\tau),u(\tau))d\tau+g(X(T))$ is inconsistent with $\hat{Q}(s_t,a_t)$. In order to address the issue, we make some modifications to the DDPG algorithm with the help of terminal prediction and propose the new algorithm \ref{algorithm} called Terminal Prediction Deterministic Policy Gradient (TP-DDPG).

\begin{algorithm}
\caption{Terminal prediction Deep Deterministic Policy Gradient}\label{algorithm}
\KwIn{Randomly initialize critic network $Q_{\theta}(s,a)$ and actor network $A_{\omega}(s)$ with weights $\theta$ and $\omega$.}
\KwOut{Updated network $Q_\theta$ and network $A_{\omega}$ after $K$ episodes of the algorithm}
    \emph{Random run $M$ trajectories, collect samples $\{\mathcal{D}\}$}\;
    
    \For{$n\leftarrow 0$ \KwTo{episode}}
    { $s_0=s(0)$\;

    \emph{Set $Q_\theta(s_N) = 0$}\;
    
    \For{$t\leftarrow 0$ \KwTo{N}}{
     \emph{Select action with exploration noise $a_t\leftarrow A_\omega(s_t)+\epsilon, \,\,\epsilon\sim \mathcal{N}(0,\sigma^2_{act})$ according to the current policy}\;
     
     \emph{Execute action $a_t$ and observe running cost $r_t=\frac{1}{2}[a_t^2+\nabla  \cdot b(s_t)]\Delta t$, then get new state $s_{t+1}\leftarrow F(s_t,a_t)$}\;
     
     \emph{Store transition $(s_t,a_t,r_t,s_{t+1})$ in $\mathcal{D}_{[t]}$}\;

     \emph{Sample a minibatch of M transitions $\{(s_t^i,a_t^i,r_t^i,s_{t+1}^{i})\}_{i=1}^M$}\;

     \emph{Set $y_t^i\leftarrow r_t^i+Q_{\theta}(s_{t+1}^{i},A_{\omega}(s_{t+1}^{i}))$}\;

     \emph{Update critic by minimizing the TD loss:
     $$\mathcal{L}_Q\leftarrow\frac{1}{M}\sum\limits_{i=1}^M (y_t^i-(Q_{\theta}^i(s_t^i,a_t^i)))^2.$$}
     
     
     \emph{Terminal prediction $\mathcal{L}_{pred}\leftarrow g(\hat{s}_N^{i,pred})$}\;

     \emph{Update actor policy using the sampled policy gradient:}\;
     
     \emph{
     $\nabla_\omega \mathcal{L}_{act}\approx \frac{1}{M} \sum_{i=1}^M \Big[\nabla_a Q_\theta (s,a)|_{s=s_t^i,a=a_t^i}\cdot \nabla_{\omega}A(s)|_{s=s_t^i}+\nabla_{a}\mathcal{L}_{pred}\cdot \nabla_{\omega }A(s)|_{s=s_t^i}\Big] $.
     }
     }
    }
\end{algorithm}

In the algorithm TP-DDPG, the loss function of the actor network is defined as 
\begin{equation}
\mathcal{L}_{act}(s_k,a_k)=\hat{Q}(s_k,a_k)+\mathcal{L}_{pred},  \label{define actorl loss}
\end{equation}
where the first part is the state-action value $\hat{Q}(s_k,a_k)$ estimated by the critic network, and the second one is the loss of the terminal prediction. That means $\mathcal{L}_{act}$ is the estimation of total cost 
\begin{equation}
    \mathcal{\tilde{J}}_k=\sum_{i=k}^{N-1} f(X_i,u_i) \Delta t+g(X_N), \label{define total cost}
\end{equation}
where $\hat{Q}(s_k,a_k)$ estimates the running cost $\int_t^Tf(X(\tau),u(\tau))d\tau$, and $\mathcal{L}_{pred}$ calculates the approximated terminal cost. By the chain rule, we have
$$\frac{\partial\mathcal{L}_{pred}(\hat{s}_N^{k,\hat{a}})}{\partial \omega}=\frac{\partial g(\hat{s}_N^{k,\hat{a}})}{\partial \omega}=\frac{\partial g(\hat{s}_N^{k,\hat{a}})}{\partial \hat{a}}\cdot \frac{\partial \hat{a}}{\partial \omega}\,.$$ 
The actor network can be trained through the utilization of the back-propagation method, where the terminal loss is incorporated into the learning process at each interaction. This approach allows the agent tend to choose actions that are more likely to reach the target terminal state. Then the actor network is updates by
$$\frac{\partial\mathcal{L}_{act}}{\partial \omega}=\frac{\partial\mathcal{L}_{\mathcal{Q}}}{\partial \hat{a}}\cdot \frac{\partial \hat{a}}{\partial \omega}+\frac{\partial g(\hat{s}_N^{k,\hat{a}})}{\partial \hat{a}}\cdot \frac{\partial \hat{a}}{\partial \omega}\,.$$
The critic network is utilized to approximate the accumulative running cost $\sum_{k=t}^{N-1}  r_k(s_k,a_k)$, we update critic network by minimizing TD error:
$$\big(r_t^i+Q_{\theta}(s_{t+1}^{i},A_{\omega}(s_{t+1}^{i}))-(Q_{\theta}^i(s_t^i,a_t^i))\big)^2\,.$$
The flow chart is presented in Figure \ref{flow chart} to illustrate the concept of TP-DDPG.

\newtheorem{remark1}{Remark}[section]

\begin{remark1}
    Some studies investigate finite-horizon DDPG \cite{vp2021finite} in a backward manner, whereas our focus lies in addressing the finite-horizon optimal control problem through a forward approach. Besides, in our algorithm, the actor network updates its parameters to choose the action that leads to a lower total cost $\mathcal{J}$.
\end{remark1}
\begin{remark1}
In our experiments, we discover that at the beginning when the agent is unable to reach the target point, the terminal prediction loss dominates the update of the actor network, as $\mathcal{L}_{pred}$ is bigger at this stage. Following a few episodes of training, the agent is capable of achieving the target. In this situation, the Q value obtained by the critic network dominates the update of the actor network. In other words, our agent first finds a transition path that begins at $X_0$ and terminates at the target point $X_T$. Then, under the influence of the critic network, the agent updates the geometry of the pathway so that it approaches the smallest Onsager-Machlup action functional.
\end{remark1}

\section{Convergence Analysis}\label{Convergence analysis}

\setlength{\parindent}{2em} In our algorithm, the critic network is used to estimate the accumulative running cost $\int_t^T f(X(\tau),u(\tau))d\tau$, and the terminal prediction cost $\mathcal{L}_{pred}$ is utilized to estimate the terminal cost $g(X_N)$. Let $Q_n$ denote the optimal state-action value function by dynamic programming, for $n=N-1, \cdots, 0$, we have
\begin{equation*}
\begin{cases}
Q_N(x_N)=g(x_N)\,,\\
    Q_n(x_n,a)=f(x_n,a)\Delta t+\mathop{\mbox{inf}}\limits_{a\in \mathbb{A}}Q_{n+1}(x_{n+1}, a)\,,
\end{cases}
\end{equation*}
where $\Delta t=\frac{T}{N}$.

This section is devoted to the convergence of the estimator $\hat{Q}_n^M+g(X_N^{n,\hat{a}_n})$ of the state-action value function $Q_n$ obtained from a training sample of size $M$, where $X_N^{n,\hat{a}_n}$ stands for the terminal state predicted by current actor network $\hat{a}_n$
starting from $n$-th state.

This research investigates the estimation error caused by batch sampling and the approximation error of the neural network, leaving aside the optimization error. We first rewrite the update of the algorithm into a mathematical expression and then discuss the estimation error and approximation error of the actor and critic networks, respectively. In Lemma \ref{lemma 1} and Lemma \ref{lemma 2}, the boundaries of estimation error and approximation error caused by empirically estimating the optimal policy through the actor network are given. In Theorem \ref{theorem 1}, we investigate the two types of errors introduced by empirically estimating the state-action value function via the critic network and the relationship with the errors brought by the actor network. We refer to the paper \cite{hure2021deep}, \cite{kohler2006nonparametric}, and \cite{gyorfi2002distribution} for convergence analysis of the algorithm TP-DDPG. The research \cite{hure2021deep} presents the convergence of the proposed hybrid-now method, which is inspired by dynamic programming, and updates the network backward. While our method TP-DDPG updates neural networks in a forward way as it evolves according to state transition dynamics. That develops an essential difference between the two methods. The proof is mainly distinct from the literature in the following aspects:
\begin{itemize}
    \item \textbf{Diverse networks' update strategies}. The hybrid-now algorithm first updates the actor network and then the critic network, from the terminal point in a backward way. On the contrary, we update our networks first on critic then on actor and iterate the updating process in a forward manner with predicted terminal state. Besides, we let $\hat{Q}_N=0$ at the end of each episode which exerts a significant impact on the weight of the critic network to the next episode.
    \item \textbf{Assumption for the previous state}. The hybrid-now method needs to assume the prior knowledge on the distribution $\mu_n$ of
the previous state $X_n$ when the next state $X_{n+1}$ is known. Our TP-DDPG algorithm obtains a sequence of $\{X_n^{(m)}\}_{n=0,1\leq m\leq M}^N$ by generating $M$ trajectories according to the state transition dynamics without presumed prior distribution. 
    \item \textbf{Additional analysis for terminal prediction}. The TP-DDPG algorithm controls the agent to reach the target terminal point via terminal prediction, so it requires additional analysis of the error introduced by terminal prediction.
\end{itemize}

We provide certain notations in Table \ref{Notations} for reading convenience.

\begin{table*}[htb]
\centering
\caption{Notations}
\label{Notations}
\begin{tabular}{cc}
   \toprule
   Notation & Meaning   \\
   \midrule
$\mathcal{X}$ & State Space and $\mathcal{X}\in \mathbb{R}^d$ \\
\hline
$\mathbb{A}$ & Control Space and $\mathbb{A}\in \mathbb{R}^u$ \\
\hline

\multirow{2}{*}{$Q_n$} & \multirow{2}{*}{\makecell[c]{The theoretical optimal state-action function\\including both running and terminal cost}}\\
\specialrule{0em}{1pt}{1pt}\\
\hline
\multirow{2}{*}{$a^{opt}(X_n)$} & \multirow{2}{*}{\makecell[c]{The optimal feedback control for state-action\\function $Q_n$ under given state $X_n$}} \\
\specialrule{0em}{1pt}{1pt}\\
\hline
\multirow{2}{*}{$
\mathcal{Q}_M$} & \multirow{2}{*}{\makecell[c]{The set of neural networks to approximate\\ the state-action function excluding terminal cost}} \\
\specialrule{0em}{1pt}{1pt}\\
\hline
{$\mathcal{A}_M$} & {The class of neural networks for policy}\\
\hline
\multirow{2}{*}{$\mathbb{A}^\mathcal{X}$} & \multirow{2}{*}{\makecell[c]{The set of Borelian function from\\ state space $\mathcal{X}$ to the control space $\mathbb{A}$}} \\
\specialrule{0em}{1pt}{1pt}\\
\hline
\multirow{2}{*}{$\hat{a}_n$} & \multirow{2}{*}{\makecell[c]{The actor neural network after $n$ steps to\\ estimate the optimal policy}} \\
\specialrule{0em}{1pt}{1pt}\\
\hline
\multirow{2}{*}{$\hat{a}_n(X_n)$} & \multirow{2}{*}{\makecell[c]{The action output by actor network\\ after $n$ steps for a given state $X_n$}} \\
\specialrule{0em}{1pt}{1pt}\\
\hline
\multirow{2}{*}{$X_n^{(m)}$} & \multirow{2}{*}{The $n$-th state in $m$-th trajectories} \\
\specialrule{0em}{1pt}{1pt}\\
\hline
\multirow{2}{*}{$\hat{a}_n^\xi(X_n^{(m)})$} & \multirow{2}{*}{\makecell[c]{$\hat{a}_n^\xi(X_n^{(m)}):= \hat{a}_n(X_n^{(m)})+\xi_n^{(m)}$,\\ where $\xi_n^{(m)}\sim N(0,\sigma^2_{act})$}}\\
\specialrule{0em}{1pt}{1pt}\\
\hline
\multirow{2}{*}{$\tilde{Q}_n$} & \multirow{2}{*}{\makecell[c]{The untruncated estimation of\\the state-action function excluding terminal cost}} \\
\specialrule{0em}{1pt}{1pt}\\
\hline
\multirow{2}{*}{$\hat{Q}_n$} & \multirow{2}{*}{\makecell[c]{The truncated estimation of\\ the state-action function excluding terminal cost}} \\
\specialrule{0em}{1pt}{1pt}\\
\hline
\multirow{2}{*}{$X_N^{n,A}$} & \multirow{2}{*}{\makecell[c]{The terminal state predicted  by current actor network $A$\\ starting from $n$-th state $X_n$}}\\
\specialrule{0em}{1pt}{1pt}\\
\hline
\multirow{2}{*}{$\Bar{Q}_n$} & \multirow{2}{*}{The quantity estimated by $\hat{Q}_n$}\\
\specialrule{0em}{1pt}{1pt}\\
\hline
$(X_k^{'(m)})_{1\leq m \leq M,0\leq k \leq n} $ & The ghost sample of $(X_k^{(m)})_{1\leq m \leq M,0\leq k \leq n}$ \\
   \bottomrule
\end{tabular}
\end{table*}

In the following, we shall consider the following set of neural networks for the state-action value function approximation:
\begin{equation}
    ^{\eta}\mathcal{Q}^{\gamma}_{K}\triangleq\{x\in \mathcal{X}\times \mathbb{A}\mapsto \Phi(x;\theta)=\sum_{i=1}^K c_i\sigma_A(a_{i\cdot}x+b_i)+c_0,\theta=(a_i,b_i,c_i)_i\} \label{critic network}
\end{equation}
where $a_{i}\in \mathbb{R}^{d+u}$, $b_i,c_i\in \mathbb{R}$, and $\|a_i\|\leq\eta$, $|\sum_{i=1}^K c_i|\leq \gamma$, the $\|\cdot\|$ is the Euclidean norm in $\mathbb{R}^{d+u}$. Note that $\sigma_A$ represents the Arctan activation function, and $K$ denotes the number of neurons in one hidden layer. $\eta$ and $\gamma$ are often referred to in the literature as respectively kernel and total variation.

The actor network is designed to estimate feedback optimal control at time $n = 0,\cdots, N - 1$, the network architecture is depicted below:
\begin{equation*}
\begin{aligned}
    ^{\eta}\mathcal{A}^{\gamma}_{K}:=\big\{&x\in \mathcal{X}\triangleq A(x;\beta)=(A_1(x;\beta),\cdots,\, A_u(x;\beta))\in \mathbb{A},\\&A_i(x;\beta)=\sigma_{\mathbb{A}}\big(\sum_{j=1}^Kc_{ij}(a_{ij\cdot}x+b_{ij})_++c_{0j}\big), i=1,2,\cdots,\,u,\,\beta=(a_{ij},b_{ij},c_{ij})_{i,j}\big\},
    \end{aligned}
\end{equation*}
where $a_{ij}\in\mathbb{R}^u, b_{ij},c_{ij}\in \mathbb{R}$ and $\|a_{ij}\|\leq \eta, \sum_{i=1}^K|c_{ij}|\leq \gamma$, the $\|\cdot\|$ is the Euclidean norm in $\mathbb{R}^u$. The $u$ represents the dimension of $\mathbb{A}$. The activation function $\sigma_{\mathbb{A}}$ is determined by the form of the control space.
 The actor neural networks under consideration contain one hidden layer with $K$ neurons, Relu activation functions, and $\sigma_{\mathbb{A}}$ as the output layer.

Here are our assumptions on the boundedness and Lipschitz condition of the cost function:

\newtheorem{assumption 1}{Assumption}[section]
\begin{assumption 1}
There exist some positive constants $\| f\|_\infty$, $\| g\|_\infty$, $[f]_L$, and $[g]_L$ satisfying
\begin{equation}
    \begin{aligned}
        \|f(x,a)\|\leq\|f\|_{\infty}, \|g(x)\|\leq \|g\|_{\infty}  \,\,\,\forall x\in \mathcal{X},a\in \mathbb{A}\,,\\
|f(x,a)-f(x',a')|\leq [f]_L(|x-x'|+|a-a'|)\,,\\
|g(x)-g(x')|\leq [g]_L|x-x'|,\,\,\,\forall x,x'\in \mathcal{X},a,a'\in \mathbb{A}\,.
    \end{aligned}
\end{equation}
\end{assumption 1}

Under the boundedness condition, the state-action value function is also bounded as $\|Q_n\|_\infty \leq (N-n)\|f\|_{\infty}\Delta t+\|g\|_{\infty}$

We also assume a Lipschitz condition on the dynamics.

\newtheorem{assumption}[assumption 1]{Assumption}
\begin{assumption}
For couples $(x^{(m)}, a^{(m)})$ and $(x'^{(m)}, a'^{(m)})$ in $\mathcal{X} \times \mathbb{A}$, there exists a constant $\rho_M<1$ such that
\begin{equation}
    \begin{aligned}
        |F(x^{(m)},a^{(m)})-F(x'^{(m)},a'^{(m)})|\leq \rho_M(|x^{(m)}-x'^{(m)}|+|a^{(m)}-a'^{(m)}|)\,,
    \end{aligned}
\end{equation}
where $F$ is the discretion dynamics satisfies (\ref{state dyn}).
\end{assumption}

When feeding data into the neural network, we denote $\eta_M, \gamma_M$ and $K_M$ to be parameters of (\ref{critic network}) such that
\begin{equation}
    \eta_M, \gamma_M,K_M\xrightarrow{M\rightarrow \infty} \infty \,\,s.t. \frac{\gamma^4_MK_Mlog(M)}{M},\frac{\sqrt{\log{(M)}}\gamma_M}{\sqrt{M}}(\eta_M\gamma_M+[g]_L\frac{\rho_M-\rho_M^{N-n+1}}{1-\rho_M})\xrightarrow{M\rightarrow \infty} 0\,.
\end{equation}
where $M$ represents the batch size. For the simplicity of notations, we denote $\mathcal{Q}_M:= {^{\eta}\mathcal{Q}^{\gamma}_{K}}$

Denote the set of Borelian functions from the state space $\mathcal{X}$ into the control space $\mathbb{A}$ as $\mathbb{A}^\mathcal{X}$ and represent the set of neural networks as $\mathcal{A}_M:={^\eta\mathcal{A}_K^\gamma}$. For $n=0,\cdots, N-1$, $\hat{a}_n\in\mathcal{A}_M$ stands for the actor neural network after $n$ steps.

The actor network and critic network's modified formula should be written as the following mathematical expression:

\begin{itemize}
\item[$\bullet$] Sampling $M$ trajectories into buffer.
\item[$\bullet$] For $n=0,\cdots,N-2$, we take $M$ samples $\big\{\big(X_{n}^{(m)},a_n^{(m)},r(X_{n}^{(m)},a_n^{(m)}),X_{n+1}^{(m)}\big)_{m=1}^M\big\}$ from the buffer at each timestep along with $n$, where $a_n^{(m)}:=\hat{a}_n^{\xi}\big(X_n^{(m)}\big), r\big(X_{n}^{(m)},a_n^{(m)}\big):=f\big(X_{n}^{(m)},\hat{a}_n^{\xi}\big(X_n^{(m)}\big)\big)\Delta t. $

\begin{itemize}
    \item[(i)] Compute the untruncated estimation of the value function at time $n+1$.
    \begin{equation}
    \begin{aligned}
    \tilde{Q}_{n+1}\in \mathop{argmin}\limits_{\Phi\in\mathcal{Q}_M}
\frac{1}{M}\sum_{m=1}^M\bigg[-f\big(X_n^{(m)},\hat{a}_n^{\xi}\big(X_n^{(m)}\big)\big)\Delta t+\hat{Q}_n\big(X_n^{(m)},\hat{a}_n^{\xi}\big(X_n^{(m)}\big)\big)
-\Phi\Big(X_{n+1}^{(m)},\hat{a}_n(X_{n+1}^{(m)}))\Big)\bigg]^2\,,
\end{aligned}
    \end{equation}
where $\hat{a}_n^{\xi}\big(X_n^{(m)}\big)=\hat{a}_n(X_n^{(m)})+\xi_n^{(m)},\,\xi_n^{(m)}\sim N(0,\sigma^2_{act})$.

    Set the truncated estimated value function at time $n+1$
    \begin{equation}        \hat{Q}_{n+1}=max\big(min(\tilde{Q}_{n+1}+g,\|Q_{n+1}\|_{\infty}),-\|Q_{n+1}\|_{\infty}\big)-g\,.\label{truncated}
    \end{equation}

   \item[(ii)]
   Compute the approximated policy at time $n+1$
   \begin{equation}
       \begin{aligned}           \hat{a}_{n+1}\in\mathop{argmin}\limits_{A\in\mathcal{A}_M}\frac{1}{M}\sum_{m=1}^M\bigg[\hat{Q}_{n+1}\Big(X_{n+1}^{(m)},A\big(X_{n+1}^{(m)}\big)\Big)\Big)+g\Big(X_N^{(m),n+1,A}\Big)\bigg]\,, \label{actor update}
       \end{aligned}
   \end{equation}
     where $X_N^{(m),n+1,A}$ represents the terminal state predicted by network $A\in \mathcal{A}_M$ from $n+1$-th state $X_{n+1}^{(m)}$.
\end{itemize}

\end{itemize}
\newtheorem{remark}[assumption 1]{Remark}
\begin{remark}
In reinforcement learning, the DDPG applies TD error to update the critic network by
\begin{equation*}
    Q(s_t,a_t;\theta)\approx r(s_t,a_t)+Q(s_{t+1},a_{t+1};\theta)\,,
\end{equation*}
that is 
$$
Q(s_{t+1},a_{t+1};\theta)\approx - r(s_t,a_t) + Q(s_t,a_t;\theta) \,,
$$
where $a_k=\hat{a}(s_k)$. Our goal is to get the recurrent relation between $|\hat{Q}_{n+1}+g(X_N^{n+1,\hat{a}_{n+1}})-Q_{n+1}|$ and $|\hat{Q}_{n}+g(X_N^{n,\hat{a}_{n}})-Q_{n}|$.
\end{remark}
\begin{remark}
    Note that we truncate the estimated state-action value function $\hat{Q}_n$ with an a priori constraint (\ref{truncated}) on the true state-action value function $Q_n$. From a practical implementation point of view, this truncation step is natural and is also used to simplify the proof of convergence of the algorithm.
\end{remark}

First, we induce the estimation error at time $n+1$ associated with the TP-DDPG algorithm by
\begin{equation}
    \begin{aligned}
        \varepsilon_{n+1}^{esti}=\sup_{A\in\mathcal{A}_M}\Bigg|&\frac{1}{M}\sum_{m=1}^M\bigg[-f\big(X_n^{(m)},A(X_n^{(m)})\big)\Delta t+\hat{Q}_n\big(X_n^{(m)},A(X_n^{(m)})\big)+g\big(X_N^{(m),n,A}\big)\bigg]\\&-\mathbb{E}_M\bigg[-f\big(X_n,A(X_n)\big)\Delta t+\hat{Q}_n\big(X_n,A(X_n)\big)+g\big(X_N^{n,A}\big)\bigg]\Bigg|\,,
    \end{aligned}
\end{equation}
where the $\mathbb{E}_M$ means the expectation conditioned by the training set utilized in estimating optimal policies $(\hat{a}(X_k))_{1\leq k \leq n}$. The estimation error measures how closely the selected estimator (e.g., the mean square estimate) approximates a certain quantity (e.g., the conditional expectation). Obviously, we anticipate that the estimation to become more accurate when the size of the training set is sufficiently large.
\begin{remark}
    This estimation error $\varepsilon^{esti}_{n+1}$ is caused by the use of empirical cost functional in (\ref{actor update}) to approximate the optimal control function by the neural network in $\mathcal{A}_M$. For $\varepsilon^{esti}_{n+1}$, here $X_n$ represents previous state for $X_{n+1}$ obtained from the quadruplet $\big(X_{n}^{(m)},a_n^{(m)},r(X_{n}^{(m)},a_n^{(m)}),X_{n+1}^{(m)}\big)$.
\end{remark}
Then we have the following bound for the estimation error.
\newtheorem{lemma}[assumption 1]{Lemma}
\begin{lemma}\label{lemma 1}
    For $n=0,1,\cdots,N-1$, the following holds
    \begin{equation}
        \begin{aligned}
            \mathbb{E}\varepsilon^{esti}_{n+1}\leq (\sqrt{2}+16)\frac{(N-n+1)\|f\|_{\infty}+\|g\|_{\infty}}{\sqrt{M}}+16\bigg([f]_L+\eta_M\gamma_M+[g]_L\frac{\rho_M-\rho_M^{N-n+1}}{1-\rho_M}\bigg)\frac{\gamma_M}{\sqrt{M}}\,.
        \end{aligned}
    \end{equation}
    
\end{lemma}

\begin{proof}
See Appendix \ref{proof of lemma 1}.
\end{proof}

Since the class of neural networks $\mathcal{A}_M$ is not dense in the set $\mathbb{A}^\mathcal{X}$ of all Borelian functions, we consider the approximation error as a measure of how accurately the neural network function in set $\mathcal{A}_M$ approximates the regression function.

The approximation error at time $n+1$ is defined in the following
\begin{equation}
    \begin{aligned}
        \varepsilon_{n+1}^{approx}=&\inf_{A\in \mathcal{A}_M}\mathbb{E}_M\bigg[-f(X_n,A(X_n))\Delta t+\hat{Q}_n(X_n,A(X_n))+g(X_N^{n,A})\bigg]\\
        &-\inf_{A\in\mathbb{A}^\mathcal{X}}\mathbb{E}_M\bigg[-f(X_n,A(X_n))\Delta t+\hat{Q}_n(X_n,A(X_n))+g(X_N^{n,A})\bigg]\,.
    \end{aligned}
\end{equation}

\begin{lemma}\label{lemma 2}
    For $n=0,1,\cdots, N-1$, the following holds
    \begin{equation}
        \begin{aligned}
            \varepsilon_{n+1}^{approx}\leq &\inf_{A\in\mathcal{A}_M}\Big\{2[f]_L\Delta t\mathbb{E}_M\Big[\Big|A(X_n)-a^{opt}(X_n)\Big|\Big]+2\mathbb{E}_M\Big[\Big|\hat{Q}_n(X_n,A(X_n))+g(X_N^{n,A})-Q_n(X_n,a^{opt}(X_n))\Big|\Big]\Big\}\,.
        \end{aligned}
    \end{equation}
\end{lemma}

\begin{proof}
See Appendix \ref{proof of lemma 2}.
\end{proof}

\newtheorem{theorem}[assumption 1]{Theorom}

Now we present the main result for our algorithm TP-DDPG.
\begin{theorem}
(\textbf{\mbox{Convergence analysis}}) Assume there exists an optimal feedback control $(a^{opt}(X_k))_{k=1}^n$ for the control problem with the optimal state-action value $Q_k$ for $k=1,2,\cdots, n$. Then, as $M\rightarrow\infty$,
\begin{equation}
    \begin{aligned}
        \inf_{A\in\mathcal{A}_M} \mathbb{E}_M&\Big[\Big|\hat{Q}_n(X_n,A(X_n))+g(X_N^{n,A})-Q_n(X_n,a^{opt}(X_n)\Big|\Big]\\
&= \mathcal{O}_{\mathbb{P}}\Bigg(\Big(\gamma^4_M\frac{K_M log(M)}{M}\Big)^{\frac{1}{2n}}+\Big(\frac{\gamma_M\sqrt{\log{(M)}}}{\sqrt{M}}\Big(\eta_M\gamma_M+[g]_L\frac{\rho_M-\rho_M^{N-n+1}}{1-\rho_M}\Big)\Big)^{\frac{1}{2n}}\\
&+\sup_{1\leq k\leq n}\inf_{A\in\mathcal{A}_M}\inf_{\Phi\in\mathcal{Q}}\Big(\mathbb{E}_M\big[\big|\Phi(X_k,A(X_k))+g(X_N^{n,A})-Q_k(X_k,a^{opt}(X_k))\big|\big]\Big)^{\frac{1}{2n}}\\
&+\sup_{0\leq k\leq n-1}\inf_{A\in\mathcal{A}_M}\Big(\mathbb{E}_M\big[\big|A(X_k)-a^{opt}(X_k)\big|\big])^{\frac{1}{2n}}+\Big(\big|\hat{Q}_0(X_0,\hat{a}_0(X_0))-Q_0(X_0,a^{opt}(X_0)\big|\Big)^{\frac{1}{2n}}\,,
    \end{aligned}
\end{equation}
where $\mathbb{E}_M$ denotes the expectation conditioned by the training set used to estimate the optimal policies $(\hat{a}_k)_{k=1}^n$. The notation $z_M=\mathcal{O}_\mathbb{P}(y_M)$ as $M\rightarrow \infty$ stands for that there exists $c>0$ such that $\mathbb{P}(|z_M|>c|y_M|)\rightarrow 0$ as $M$ goes to infinity.
\label{theorem 1}
\end{theorem}

\begin{remark}
Note that $\Big(\gamma^4_M\frac{K_M log(M)}{M}\Big)^{\frac{1}{2n}}$ and $\Big(\frac{\gamma_M\sqrt{\log{(M)}}}{\sqrt{M}}(\eta_M\gamma_M+[g]_L\frac{\rho_M-\rho_M^{N-n+1}}{1-\rho_M})\Big)^{\frac{1}{2n}}$ represent the estimation error caused by empirically estimating the state-action value function and the control policy by neural networks, where $\Big(\frac{\gamma_M\sqrt{\log{(M)}}}{\sqrt{M}}[g]_L\frac{\rho_M-\rho_M^{N-n+1}}{1-\rho_M}\Big)$ corresponds to terminal prediction loss. It shows that the larger N is, the larger the error introduced by the terminal prediction. This is because accurately predicting the terminal state $s_N$ becomes harder as $N$ grows. The third term and fourth term denote the approximation error with respect to the value function and the optimal control by neural networks. The last term quantifies the difference between the estimated $\hat{Q}_0(X_0,\hat{a}_0(X_0))$ and $Q_0(X_0,a^{opt}(X_0)$. As the networks undergo updates over a series of episodes, the estimated $\hat{Q}_0$ gradually converges to $Q_0$.

    To clarify the meaning of the above items, we present the following Table \ref{error}, where
        \textcircled{1}$:=\Big(\gamma^4_M\frac{K_M log(M)}{M}\Big)^{\frac{1}{2n}}$,
        \textcircled{2}$:=\Big(\frac{\gamma_M\sqrt{\log{(M)}}}{\sqrt{M}}\big(\eta_M\gamma_M+[g]_L\frac{\rho_M-\rho_M^{N-n+1}}{1-\rho_M}\big)\Big)^{\frac{1}{2n}}$,
        \textcircled{3}$:=\mathop{sup}\limits_{0\leq k\leq n}\mathop{inf}\limits_{A\in\mathcal{A}_M}\mathop{inf}\limits_{\Phi\in\mathcal{Q}}\Big(\mathbb{E}\big[\big|\Phi(X_k,A(X_k))+g(X_N^{n,A})-Q_k(X_k,a^{opt}(X_k))\big|\big]\Big)^{\frac{1}{2n}}$, and 
        \textcircled{4}$:=\mathop{sup}\limits_{0\leq k\leq n-1}\mathop{inf}\limits_{A\in\mathcal{A}_M}\Big(\mathbb{E}\big[\big|A(X_k)-a^{opt}(X_k)\big|\big]\Big)^{\frac{1}{2n}}$.        

\begin{table*}[!h]
    \centering
    \caption{Explanation for the right of result (\ref{conclu})}
    \label{error}
    \setlength{\tabcolsep}{6mm}
    \begin{tabular}{ccc}
\toprule
Item & Error type & Causes\\
\hline
\specialrule{0em}{1pt}{1pt}\\
\textcircled{1} & Estimation error & Empirically estimating state-action function \\ 
\specialrule{0em}{1pt}{1pt}\\
\multirow{2}{*}{\textcircled{2}} & \multirow{2}{*}{Estimation error} &  \multirow{2}{*}{\makecell[c]{Empirically estimating policy function\\ and terminal prediction by actor network}} \\
\specialrule{0em}{1pt}{1pt}\\
\specialrule{0em}{1pt}{1pt}\\
\textcircled{3} & Approximation error & Approximating the state-action function by critic network\\
\specialrule{0em}{1pt}{1pt}\\
\textcircled{4} & Approximation error & Approximating the policy function by actor network\\
\specialrule{0em}{1pt}{1pt}\\
\bottomrule
\end{tabular}
\end{table*}
    
\end{remark}

\begin{remark}
    At each episode, the critical network (value of $\hat{Q}_0$) experiences significant influence from the condition $\hat{Q}_N=0$ in the last round, which ensures that the critic network is updated in a direction that progressively aligns with the actual cumulative running cost instead of other wrong orientations. Thus, the convergence of the critic network is intricately tied to the influence imparted from episodes. 
    
    Vivek and Bhatnagar \cite{vp2021finite} perform a convergence analysis related to the different episodes. Here we could have the similar way to analyze how the condition $\hat{Q}_N=0$ affects $\hat{Q}_0$ by updating the critic network in successive episodes. Thus, after a certain number of training episodes, when $n$ is large enough, the last term in the Theorem \ref{theorem 1} also converges to 0.
    
    Based on the experimental results shown in Section \ref{experiments}, we can see that $\hat{Q}_0$ achieves convergence after several episodes. Figure \ref{compare Q and J} demonstrates the consistency between the total cost and the summation of $\hat{Q}_0$ and $g(X_N)$.

\end{remark}

\begin{proof}
    First, we introduce the following forward deduction auxiliary procedure as
    \begin{equation}
        \begin{cases}
            \bar{Q}_0(x,a)=\hat{Q}_0(x,a),\,\,for \,\,x\in\mathcal{X},a\in \mathcal{A}_M\\
\bar{Q}_{n+1}(X_{n+1},\hat{a}_n(X_{n+1}))=\mathbb{E}\big[-f(X_n,\hat{a}_n(X_n)+\xi_n)\Delta t+\hat{Q}_n(X_n,\hat{a}_n(X_n)+\xi_n)\big|X_n,\hat{a}_n\big]\\
\qquad\qquad\qquad\qquad\quad\,\,=-f(X_n,\hat{a}_n(X_n)+)\Delta t+\hat{Q}_n(X_n,\hat{a}_n(X_n))\,,\,\,for \,\,X_{n},X_{n+1}\in\mathcal{X},\hat{a}_n\in \mathcal{A}_M\,.
        \end{cases}
    \end{equation}
    Notice that $\bar{Q}_{n+1}$ is the quantity estimated by $\hat{Q}_{n+1}$.

   \textbf{Step 1}. \quad 
   In this step, we aim to determine the bound of $$\mathbb{E}_M\Bigg[\Bigg|\bar{Q}_{n+1}(X_{n+1},\hat{a}_n(X_{n+1}))+g(X_N^{n+1,\hat{a}_n})-\inf_{a\in\mathbb{A}^\mathcal{X}}\Big\{\big[-f(X_n, a(X_n))\Delta t+\hat{Q}_n(X_n, a(X_n))+g(X_N^{n, a})\big]\Big\}\Bigg|^2\Bigg].$$
   
   Note the inequality
   \begin{equation}
       \begin{aligned}
           &\Big[\bar{Q}_{n+1}(X_{n+1},\hat{a}_n(X_{n+1}))+g(X_N^{n+1,\hat{a}_n})\Big]\\
           &\qquad\qquad\qquad\qquad-\inf_{a\in\mathbb{A}^\mathcal{X}}\Big\{\big[-f(x_n,a(x_n))\Delta t+\hat{Q}_n(x_n,a(x_n))+g(X_N^{n,a})\big]
\Big\}\geq 0\,,
       \end{aligned}
   \end{equation}
   that is because $\hat{a}_n$ cannot perform better than  the optimal strategy.

Moreover, we have
\begin{equation}
\begin{aligned}
        \mathbb{E}_M\Big[\bar{Q}_{n+1}(X_{n+1},\hat{a}_n(X_{n+1}))+g(X_N^{n+1,\hat{a}_n})\Big]&=\mathbb{E}_M\Big[-f(X_n,\hat{a}_n(X_n))\Delta t+\hat{Q}_n(X_n,\hat{a}_n(X_n))+g(X_N^{n,\hat{a}_n})\Big]\\
  &=\mathbb{E}_M\Big[-f(X_n,\hat{a}_n(X_n))\Delta t+\hat{Q}_n(X_n,\hat{a}_n(X_n))+g(X_N^{n,\hat{a}_n})\Big]\\
&-\frac{1}{M}\sum_{m=1}^M\Big[-f(X_n^{(m)},\hat{a}_n(X_n^{(m)}))\Delta t+\hat{Q}_n(X_n,\hat{a}_n(X_n^{(m)}))+g(X_N^{(m),n,\hat{a}_n})\Big]
\\&+\frac{1}{M}\sum_{m=1}^M\Big[-f(X_n^{(m)},\hat{a}_n(X_n^{(m)}))\Delta t+\hat{Q}_n(X_n,\hat{a}_n(X_n^{(m)}))+g(X_N^{(m),n,\hat{a}_n})\Big]\\
&\leq \varepsilon_{n+1}^{esti}+\frac{1}{M}\sum_{m=1}^M\Big[-f(X_n^{(m)},\hat{a}_n(X_n^{(m)}))\Delta t+\hat{Q}_n(X_n,\hat{a}_n(X_n^{(m)}))+g(X_N^{(m),n,\hat{a}_n})\Big]\,.\label{turn to estimate}
    \end{aligned}
\end{equation}
For any $A\in\mathcal{A}_M$,
\begin{equation}
    \begin{aligned}
        &\frac{1}{M}\sum_{m=1}^M\Big[-f(X_n^{(m)},A(X_n^{(m)})\Delta t+\hat{Q}_n(X_n^{(m)},A(X_n^{(m)}))+g(X_N^{(m),n,A})\Big]\\
&=\frac{1}{M}\sum_{m=1}^M\Big[-f(X_n^{(m)},A(X_n^{(m)})\Delta t+\hat{Q}_n(X_n^{(m)},A(X_n^{(m)}))+g(X_N^{(m),n,A})\Big]\\
&-\mathbb{E}_{M}\Big[-f(X_n,A(X_n))\Delta t+\hat{Q}_n(X_n,A(X_n))+g(X_N^{n,A})\Big]\\
&+\mathbb{E}_{M}\Big[-f(X_n,A(X_n))\Delta t+\hat{Q}_n(X_n,A(X_n))+g(X_N^{n,A})\Big]\\
&\leq \varepsilon_{n+1}^{esti}+\mathbb{E}_{M}\Big[-f(X_n,A(X_n))\Delta t+\hat{Q}_n(X_n,A(X_n))+g(X_N^{n,A})\Big]\,.\label{take infumum}
    \end{aligned}
\end{equation}
Recalling $\hat{a}_n=\mathop{\mbox{argmin}}\limits_{A\in\mathcal{A}_M}\frac{1}{M}\sum_{m=1}^M\big[\hat{Q}_n(X_n,A(X_n)))+g(X_N^{(m),n,A})\big]$, we have
\begin{equation}
    \begin{aligned}
          \frac{1}{M}\sum_{m=1}^M&\bigg[-f(X_n^{(m)},\hat{a}_n(X_n^{(m)}))\Delta t+\hat{Q}_n(X_n,\hat{a}_n(X_n^{(m)}))+g(X_N^{(m),n,\hat{a}_n})\bigg]\\
        &=\inf_{A\in\mathcal{A}_M}\frac{1}{M}\sum_{m=1}^M \bigg[-f(X_n^{(m)},\hat{a}_n(X_n^{(m)}))\Delta t+\hat{Q}_n(X_n,A(X_n^{(m)}))+g(X_N^{(m),n,A})\bigg]\\
&= \inf_{A\in \mathcal{A}_M}\bigg\{\frac{1}{M}\sum_{m=1}^M\Big[-f(X_n^{(m)},A(X_n^{(m)})\Delta t+\hat{Q}_n(X_n^{(m)},A(X_n^{(m)}))+g(X_N^{(m),n,A})\Big]\\&\qquad+\frac{1}{M}\sum_{m=1}^M\Big[-f(X_n^{(m)},\hat{a}_n(X_n^{(m)}))+f(X_n^{(m)},A(X_n^{(m)}))\Big]\Delta t\bigg\}\\
& \leq \inf_{A\in \mathcal{A}_M}\bigg\{\frac{1}{M}\sum_{m=1}^M\Big[-f(X_n^{(m)},A(X_n^{(m)})\Delta t+\hat{Q}_n(X_n^{(m)},A(X_n^{(m)}))+g(X_N^{(m),n,A})\Big]\}\\
& \qquad+\sup_{A\in\mathcal{A}_M}\frac{1}{M}\sum_{m=1}^M\Delta t\Big[-f(X_n^{(m)},\hat{a}_n(X_n^{(m)}))+f(X_n^{(m)},A(X_n^{(m)}))\Big]\,.
\label{need f-f} 
    \end{aligned}
\end{equation}
Taking the infimum over $A\in\mathcal{A}_M$ to the (\ref{take infumum}) on both sides, then with (\ref{need f-f}) and (\ref{turn to estimate}), we have
\begin{equation}
    \begin{aligned}
         & \mathbb{E}_M\Bigg[\Bigg|\bar{Q}_{n+1}(X_{n+1},\hat{a}_n(X_{n+1}))+g(X_N^{n+1,\hat{a}_n})-\inf_{a\in\mathbb{A}^\mathcal{X}}\Big\{\big[-f(X_n,a(X_n))\Delta t+\hat{Q}_n(X_n,a(X_n))+g(X_N^{n,a})\big]\Big\}\Bigg|^2\Bigg]\\
&\leq 2((N-n)\|f\|_{\infty}\Delta t+\|g\|_{\infty})(2\varepsilon_{n+1}^{esti}+\varepsilon_{n+1}^{approx}+\sup_{A\in\mathcal{A}_M}\frac{1}{M}\sum_{m=1}^M\Delta t\Big[-f(X_n^{(m)},\hat{a}_n(X_n^{(m)}))+f(X_n^{(m)},A(X_n^{(m)}))\Big]\\
&=2((N-n)\|f\|_{\infty}\Delta t+\|g\|_{\infty})(2\varepsilon_{n+1}^{esti}+\varepsilon_{n+1}^{approx}+Y(X_n^{(m)}))\,,
\label{need explain}
    \end{aligned}
\end{equation}
where we denote $Y(X_n^{(m)})=\sup_{A\in\mathcal{A}_M}\frac{1}{M}\sum_{m=1}^M\Delta t\Big[-f(X_n^{(m)},\hat{a}_n(X_n^{(m)}))+f(X_n^{(m)},A(X_n^{(m)}))\Big]$. Then we have
\begin{equation}
        \mathbb{E}[Y(X_n^{(m)})]\leq [f]_L\Delta t\frac{\gamma_M}{\sqrt{M}}\,.
\end{equation}
That inequality arises from the same arguments as the \textit{Step 3} and \textit{Step 4} in the proof of Lemma \ref{lemma 1}, which we show the details in Appendix \ref{appendix 3}.

\textbf{Step 2}. \quad 
In this step, we first show the following relation for $A\in\mathcal{A}_M$:
\begin{equation}
    \begin{aligned}
        \mathbb{E}_M&\bigg[\Big|\hat{Q}_{n+1}(X_{n+1},A(X_{n+1})-\bar{Q}_{n+1}(X_{n+1},A(X_{n+1}))\Big|^2\bigg]\\&=\mathcal{O}\bigg(\gamma^4_M\frac{K_M log(M)}{M}+\inf_{\Phi\in\mathcal{Q}_M}\mathbb{E}_M\Big[\Big|\Phi(X_{n+1},A(X_{n+1}))-\bar{Q}_{n+1}(X_{n+1},A(X_{n+1}))\Big|^2\Big]\bigg) \,.\label{need combine 1}
    \end{aligned}
\end{equation}
Fix $M\in \mathbb{N}^+$, let $x_1,x_2,\cdots, x_M\in\mathbb{R}^{d+u}$, and set $x^M=(x_1,\cdots,x_M)$. Define the distance $d(f,g)$ between $f:\mathbb{R}^{d+u}\rightarrow \mathbb{R}$ and $g:\mathbb{R}^{d+u}\rightarrow \mathbb{R}$ by
$$d(f,g)=\big(\frac{1}{M}\sum_{m=1}^M|f(x_m)-g(x_m)|^2\big)^{\frac{1}{2}}.$$
An $\varepsilon$-cover of $\mathcal{Q}$ is a set of functions $f_1,f_2,\cdots,f_L:\mathbb{R}^{d+u}\rightarrow \mathbb{R}$ such that
$$\min_{l=1,2,\cdots,L} d(f,f_l)\leq \varepsilon \,\,\,for \,\,f\in\mathcal{Q}.$$
Let $\mathcal{N}_2(\varepsilon,\mathcal{Q},x^M)$ denote the size $L$ of the smallest $\varepsilon$-cover of $\mathcal{Q}$ w.r.t. the distance $d$ and set by convention $\mathcal{N}_2(\varepsilon ,\mathcal{Q}, x^M)=+\infty$ if there does not exist any $\varepsilon$-cover of $\mathcal{Q}$ of finite size.

Take $\delta_M:=\gamma^4K_M\frac{logM}{M}$, let $\delta > \delta_M$, and denote
$$\Omega_g:=\Big\{f - g : f \in \mathcal{Q}_M,\frac{1}{M}\sum^M_{m=1}|f(x_m) - g(x_m)|^2\leq \frac{\delta_M}{\gamma_M^2}\Big\}.$$
By Appendix \ref{appendix 2}
\begin{equation} 
\begin{aligned}   
        &\int_{c_2\delta/\gamma^2_M}^{\sqrt{\delta}}\mbox{log}(\mathcal{N}_2(\frac{u}{4\gamma_M},\Omega_g,x_1^M))^{\frac{1}{2}}du\leq \int_{c_2\delta/\gamma^2_M}^{\sqrt{\delta}}\mbox{log}(\mathcal{N}_2(\frac{u}{4\gamma_M},\mathcal{Q}_M,x_1^M))^{\frac{1}{2}}du\\
        &\leq \int_{c_2\delta/\gamma^2_M}^{\sqrt{\delta}} \big((4(d+u)+9)K_M+1\big)^\frac{1}{2}\Big[\mbox{log}\big(\frac{48e\gamma_M^2(K_M+1)}{u}\big)\Big]^\frac{1}{2}du\\
        & \leq \int_{c_2\delta/\gamma^2_M}^{\sqrt{\delta}} \big((4(d+u)+9)K_M+1\big)^\frac{1}{2}\Big[\mbox{log}\big(48e(K_M+1)\frac{\gamma_M^4}{\delta}\big)\Big]^\frac{1}{2}du\\
        & \leq \sqrt{\delta}\big((4(d+u)+9)K_M+1\big)^\frac{1}{2}\Big[\mbox{log}\big(48e(K_M+1)\gamma_M^4\big)\Big]^\frac{1}{2}\\
        & \leq c_5\sqrt{\delta}\sqrt{K_M}\sqrt{\mbox{log}(M)}\,.
        \end{aligned}
\end{equation}
Since $\delta_M:=\gamma^4_MK_M\frac{\mbox{log}{(M)}}{M}<\delta$, we have $\sqrt{\delta}\sqrt{K_M}\sqrt{\mbox{log}}\leq  \frac{\sqrt{M}\delta}{\gamma_M^2}$. Then apply the Appendix \ref{appendix 1}, the following holds
\begin{equation}
    \begin{aligned}
        \mathbb{E}_M&[|\tilde{Q}_{n+1}(X_{n+1},A(X_{n+1}))-\bar{Q}_{n+1}(X_{n+1},A(X_{n+1}))|^2]\\&=\mathcal{O}_{\mathbb{P}}(\gamma^4_M\frac{K_M \mbox{log}(M)}{M}+\inf_{\Phi\in\mathcal{Q}}\mathbb{E}_M[|\Phi(X_{n+1},A(X_{n+1}))-\bar{Q}_{n+1}(X_{n+1},A(X_{n+1}))|^2])\,.
    \end{aligned}
\end{equation}
Note that $\mathbb{E}_M[|\hat{Q}(X_{n+1},A(X_{n+1}))-\bar{Q}(X_{n+1},A(X_{n+1}))|^2]\leq\mathbb{E}_M[|\tilde{Q}(X_{n+1},A(X_{n+1}))-\bar{Q}(X_{n+1},A(X_{n+1}))|^2]$ always holds, it turns out 
\begin{equation}
    \begin{aligned}
        \mathbb{E}_M&\Bigg[\Bigg|\hat{Q}_{n+1}(X_{n+1},A(X_{n+1}))-\bar{Q}_{n+1}(X_{n+1},A(X_{n+1}))\Bigg|^2\Bigg]\\&=\mathcal{O}_{\mathbb{P}}\Bigg(\gamma^4_M\frac{K_M \mbox{log}(M)}{M}+\inf_{\Phi\in\mathcal{Q}}\mathbb{E}\Big[\Big|\Phi(X_{n+1},A(X_{n+1}))-\bar{Q}_{n+1}(X_{n+1},A(X_{n+1}))\big|^2\Big]\bigg)\,.
    \end{aligned}
\end{equation}

\textbf{Step 3}. \quad 
For $\Phi\in\mathcal{Q}$, 
\begin{equation}
    \begin{aligned}
        \inf_{\Phi\in\mathcal{Q}}\Big\|&\Phi(X_{n+1},A(X_{n+1}))-\bar{Q}_{n+1}(X_{n+1},A(X_{n+1}))\Big\|_{M,2}\\
&\leq \inf_{\Phi\in\mathcal{Q}}\Big\|\Phi(X_{n+1},A(X_{n+1}))+g(X_N^{n+1,A})-Q_{n+1}(X_{n+1},a^{opt}(X_{n+1}))\Big\|_{M,2}\\
&\,\,\,\,+\Big\|Q_{n+1}(X_{n+1},a^{opt}(X_{n+1}))-\bar{Q}_{n+1}(X_{n+1},A(X_{n+1}))-g(X_N^{n+1,A})\Big\|_{M,2}\,, \label{split two}
    \end{aligned}
\end{equation}
where $\|\cdot\|_{M,p}=(\mathbb{E}_M[|\cdot|^p])^{\frac{1}{p}}$. 

In this step, we intend to get the bound of $$\inf_{A\in\mathcal{A}_M}\big\|Q_{n+1}(X_{n+1},a^{opt}(X_{n+1}))-\bar{Q}_{n+1}(X_{n+1},A(X_{n+1}))-g(X_N^{n+1,A})\big\|_{M,2}.$$

Taking the infimum over $A\in\mathcal{A}_M$ to the last term of (\ref{split two}), and split it as
\begin{equation}
    \begin{aligned}
        &\inf_{A\in\mathcal{A}_M}\big\|Q_{n+1}(X_{n+1},a^{opt}(X_{n+1}))-\bar{Q}_{n+1}(X_{n+1},A(X_{n+1}))-g(X_N^{n+1,A})\big\|_{M,2}\\&\leq
\big\|Q_{n+1}(X_{n+1},a^{opt}(X_{n+1}))-\inf_{a\in\mathbb{A}^\mathcal{X}}\big\{\big[-f(X_n,a(X_n))\Delta t+\hat{Q}_n(X_n,a(X_n))+g(X_N^{n,a})\big]\big\}\big\|_{M,2}\\
&\qquad+\inf_{A\in\mathcal{A}_M}\big\|\inf_{a\in\mathbb{A}^\mathcal{X}}\{ \big[-f(X_n,a(X_n))\Delta t+\hat{Q}_n(X_n,a(X_n))+g(X_N^{n,a})\big]\}-\bar{Q}_{n+1}(X_{n+1},A(X_{n+1}))-g(X_N^{n+1,A})\big\|_{M,2}\\
& \leq \big\|Q_{n+1}(X_{n+1},a^{opt}(X_{n+1}))-\inf_{a\in\mathbb{A}^\mathcal{X}}\big\{\big[-f(X_n,a(X_n))\Delta t+\hat{Q}_n(X_n,a(X_n))+g(X_N^{n,a})\big]\big\}\big\|_{M,2}\\
&\qquad+\big\|\inf_{a\in\mathbb{A}^\mathcal{X}}\big\{ \big[-f(X_n,a(X_n))\Delta t+\hat{Q}_n(X_n,a(X_n))+g(X_N^{n,a})\big]\big\}-\bar{Q}_{n+1}(X_{n+1},\hat{a}_n(X_{n+1}))-g(X_N^{n+1,\hat{a}_n})\big\|_{M,2}\,. \label{Q-barQ}
    \end{aligned}
\end{equation}
For the first part on the right of the above inequality
\begin{equation}
    \begin{aligned}
        &\mathbb{E}_M\big[\big|Q_{n+1}(X_{n+1},a^{opt}(X_{n+1}))-\inf_{a\in\mathbb{A}^\mathcal{X}}\big\{\big[-f(X_n,a(X_n))\Delta t+\hat{Q}_n(X_n,a(X_n))+g(X_N^{n,a})\big]\big\}\big|\big]\\
        &=\mathbb{E}_M\big[\big|-f(X_n,a^{opt}(X_n))\Delta t+Q_n(X_n,a^{opt}(X_n))-\inf_{a\in\mathbb{A}^\mathcal{X}}\big\{\big[-f(X_n,a(X_n))\Delta t+\hat{Q}_n(X_n,a(X_n))+g(X_N^{n,a})\big]\big\}\big|\big]\\
        & \leq \inf_{A\in\mathbb{A}^X}\Big\{[f]_L\Delta t\mathbb{E}_M\Big[\Big|A(X_n)-a^{opt}(X_n)\Big|\Big]+\mathbb{E}_M\Big[\Big|\hat{Q}_n(X_n,A(X_n))+g(X_N^{n,A})-Q_n(X_n,a^{opt}(X_n))\Big|\Big]\Big\}\\
&\leq \inf_{A\in\mathcal{A}_M}\Big\{[f]_L\Delta t\mathbb{E}_M\Big[\Big|A(X_n)-a^{opt}(X_n)\Big|\Big]+\mathbb{E}_M\Big[\Big|\hat{Q}_n(X_n,A(X_n))+g(X_N^{n,A})-Q_n(X_n,a^{opt}(X_n))\Big|\Big]\Big\} \,,\label{need combine 2} 
    \end{aligned}
\end{equation}
which can be justified by the same arguments used to prove the \textit{Step 2} of Lemma \ref{lemma 2}.

Notice that
\begin{equation*}
\begin{aligned}
    |Q_{n+1}(X_{n+1},a^{opt}(X_{n+1}))-\hat{Q}_{n+1}(X_{n+1},A(X_{n+1}))-g(X_N^{n+1,A})|\leq 2((N-n)\|f\|_{\infty}\Delta t+\|g\|_{\infty})\,.
\end{aligned}
\end{equation*}

Plug (\ref{need explain}), (\ref{need combine 2}) and above inequality to (\ref{Q-barQ}), we have
\begin{equation}
    \begin{aligned}
        &\inf_{A\in\mathcal{A}_M}\big\|Q_{n+1}(X_{n+1},a^{opt}(X_{n+1}))-\bar{Q}_{n+1}(X_{n+1},A(X_{n+1}))-g(X_N^{n+1,A})\big\|_{M,2}\leq\\
        &\sqrt{\inf_{A\in\mathcal{A}_M}\Big\{[f]_L\Delta t\mathbb{E}_M\Big[\Big|A(X_n)-a^{opt}(X_n)\Big|\Big]+\mathbb{E}_M\Big[\Big|\hat{Q}_n(X_n,A(X_n))+g(X_N^{n,A})-Q_n(X_n,a^{opt}(X_n))\Big|\Big]\Big\}H}\\
&+\sqrt{\big(2\varepsilon_{n+1}^{esti}+\varepsilon_{n+1}^{approx}+Y(X_n^{(m)})\big)H}\,. \label{need combine 3}
    \end{aligned}
\end{equation}
where $H=2((N-n)\|f\|_{\infty}\Delta t+\|g\|_{\infty})$.

\textbf{Step 4}. \quad 
Combine (\ref{need explain}), (\ref{split two}), (\ref{need combine 1}), (\ref{need combine 2}) and (\ref{need combine 3}), the following holds
\begin{equation}
    \begin{aligned}
        &\big\|\hat{Q}_{n+1}(X_{n+1},A)+g(X_N^{n+1,A})-Q_{n+1}(X_{n+1},a^{opt}(X_{n+1}))\big\|_{M,2}\\&\leq \big\|\hat{Q}_{n+1}(X_{n+1},A)-\bar{Q}_{n+1}(X_{n+1},A)\big\|_{M,2}+\big\|\bar{Q}_{n+1}(X_{n+1},A(X_{n+1}))+g(X_N^{n+1,A})-Q_{n+1}(X_{n+1},a^{opt}(X_{n+1}))\big\|_{M,2}\\
        &\leq \mathcal{O}_\mathbb{P}\bigg(\gamma_M^4\frac{K_Mlog(M)}{M}+\inf_{\Phi\in\mathcal{Q}}\mathbb{E}\Big[\Big|\Phi(X_{n+1},A(X_{n+1}))+g(X_N^{n+1,A})-{Q}_{n+1}(X_{n+1},a^{opt}(X_{n+1}))\Big|\Big]+\\
        & \sqrt{\inf_{A\in\mathcal{A}_M}\Big\{[f]_L\Delta t\mathbb{E}_M\Big[\Big|A(X_n)-a^{opt}(X_n)\Big|\Big]+\mathbb{E}_M\Big[\Big|\hat{Q}_n(X_n,A(X_n))+g(X_N^{n,A})-Q_n(X_n,a^{opt}(X_n))\Big|\Big]\Big\}H}\\        &+\sqrt{\big(2\varepsilon_{n+1}^{esti}+\varepsilon_{n+1}^{approx}+Y(X_n^{(m)})\big)H}\bigg).
    \end{aligned}
\end{equation}
By recursion, the following equation is obtained
\begin{equation}
    \begin{aligned}
        \inf_{A\in\mathcal{A}_M} &\mathbb{E}_M\Big[\Big|\hat{Q}_n(X_n,A(X_n))+g(X_N^{n,A})-Q_n(X_n,a^{opt})\Big|\Big]\\
&= \mathcal{O}_{\mathbb{P}}\Bigg(\Big(\gamma^4_M\frac{K_M \mbox{log}(M)}{M}\Big)^{\frac{1}{2n}}+\Big(\frac{\gamma_M\sqrt{\log{(M)}}}{\sqrt{M}}\big(\eta_M\gamma_M+[g]_L\frac{\rho_M-\rho_M^{N-n+1}}{1-\rho_M}\big)\Big)^{\frac{1}{2n}}\\
&+\sup_{1\leq k\leq n}\inf_{A\in\mathcal{A}_M}\inf_{\Phi\in\mathcal{Q}}\Big(\mathbb{E}_M\big[\big|\Phi(X_k,A(X_k))+g(X_N^{n,A})-Q_k(X_k,a^{opt}(X_k))\big|\big]\Big)^{\frac{1}{2n}}\\
&+\sup_{0\leq k\leq n-1}\inf_{A\in\mathcal{A}_M}\Big(\mathbb{E}_M\big[\big|A(X_k)-a^{opt}(X_k)\big|\big]\Big)^{\frac{1}{2n}}+\Big(\big|\hat{Q}_0(X_0,\hat{a}_0(X_0))-Q_0(X_0,a^{opt}(X_0)\big|\Big)^{\frac{1}{2n}}.\label{conclu}
    \end{aligned}
\end{equation}
\end{proof}

\section{Experiments}\label{experiments}

\setlength{\parindent}{2em} In this section, we apply TP-DDPG to compute the most probable transition paths for three examples, one-dimensional linear potential stochastic system, two-dimensional Maier-Stein system and three-dimensional lactose operon model. For the one-dimensional system and the other two nonlinear systems, our neural networks has one hidden layer has 15 neurons and 30 neurons respectively. In the actor network, we employ the Tanh activation function for the output layer and the ReLU activation function for the hidden layer. For the critic network, we utilize the Arctan activation function for the hidden layer and no activation function for the output layer. The learning rate of Adam optimizer is $10^{-3}$. All the experiments are run on A6000. Moreover, we employ penalty factors $\lambda$ in the predicted terminal cost $\mathcal{L}_{pred}$ to force the agent to reach the target terminal point at terminal time $T$.

\subsection{A Stochastic System with Linear Potential}

\setlength{\parindent}{2em} Consider the following stochastic differential equation with a linear potential,
\begin{equation}
    \begin{cases}
        dX_t=-X_tdt+dB_t\,,\\
        X_0=0, X_T=x_1\in\mathbb{R}\,.\label{linaer potential}
    \end{cases}
\end{equation}
The Onsager–Machlup action functional is by (\ref{OM functional})
$$S^{OM}_T(x,\dot{x})=\frac{1}{2}\int_0^T (|u_t|^2-1)dt.$$
By Euler–Lagrange equation, the most probable transition pathway satisfied the following two-point boundary value problem:
\begin{equation}
    \begin{cases}
        \ddot{x}=x\,,\\
        x(0)=x_0,x(T)=x_1.\label{two-point boundary value problem}
    \end{cases}
\end{equation}
By the method of constant variation, we have
\begin{equation}
    \begin{aligned}
        x(t)=-\frac{x_1-x_0e^{-T}}{e^{T}-e^{-T}}e^t+\frac{x_0e^{T}-x_1}{e^{T}-e^{-T}}e^{-t}\,.
    \end{aligned}
\end{equation}

For system (\ref{linaer potential}), $x=0$ is the unique stable point. We derive the most probable transition path from stable point $x_0=0$ to $x_1=2$ and from $x_0=0$ to $x_1=6$. Comparing the transition pathway obtained by our algorithm TP-DDPG with the solution of the two-point boundary value problem, the result is shown in figure \ref{departure from 0}. The blue line in figure \ref{0 to 2} and \ref{0 to 6} show the most probable transition pathway, and the purple dashed line stand for the true solution of the two-point boundary value problem (\ref{two-point boundary value problem}) by Euler–Lagrange equation. The pathway we plot is the average of the pathways corresponding to the 100th-300th episodes for figure \ref{0 to 2} and 300th-400th episodes for figure \ref{0 to 6}. It shows that our method works well in catching the dynamic system's most probable transition pathway.

\begin{figure*}[htb]
\centering
\subfigure[The most probable transition pathway from 0 to 2] {
 \label{0 to 2}
\includegraphics[width=0.22\columnwidth]{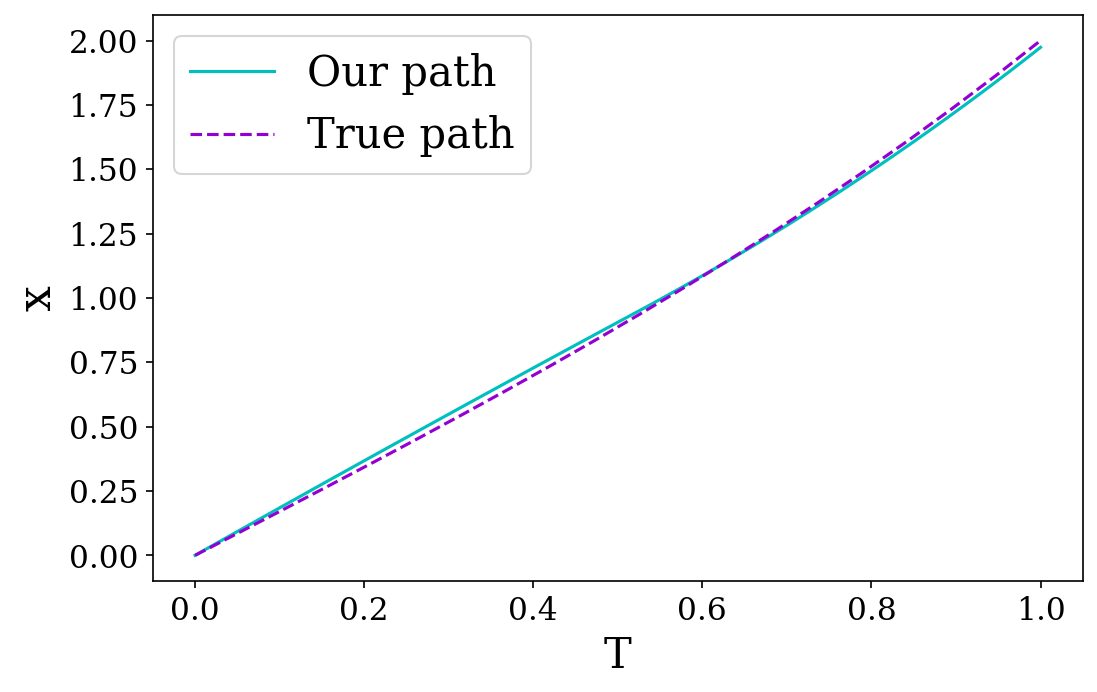}
}
\subfigure[Accumulative running cost] {
\label{0 to 2 actor}
\includegraphics[width=0.22\columnwidth]{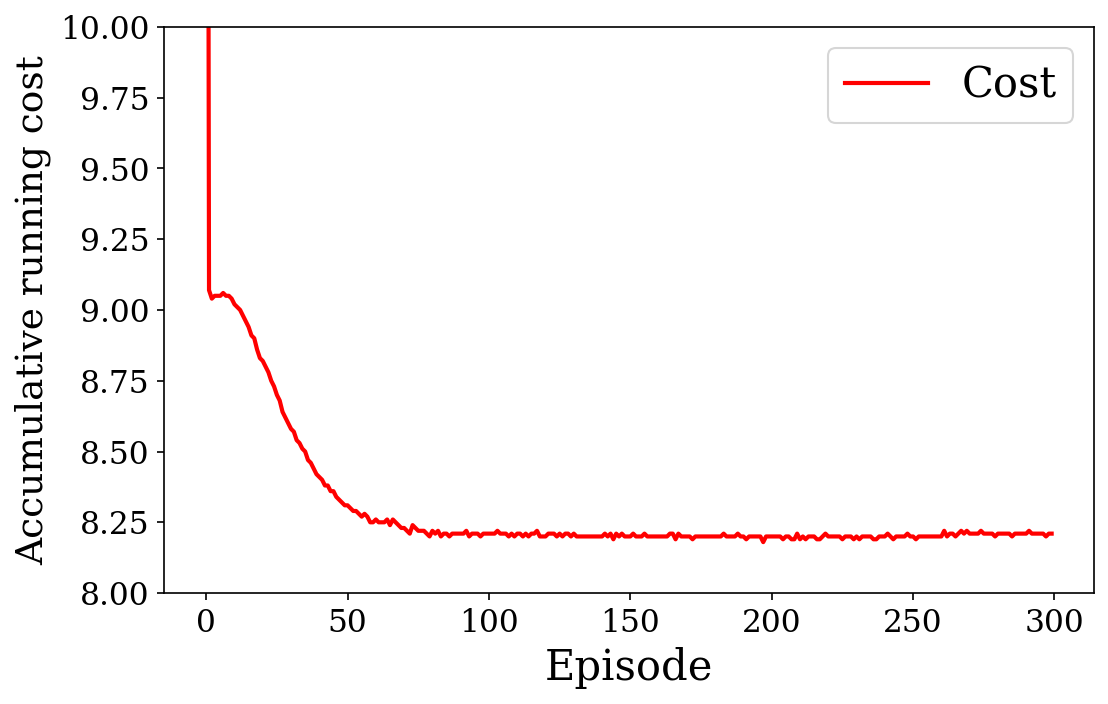}
}
\subfigure[Critic loss] {
\label{0 to 2 critic}
\includegraphics[width=0.22\columnwidth]{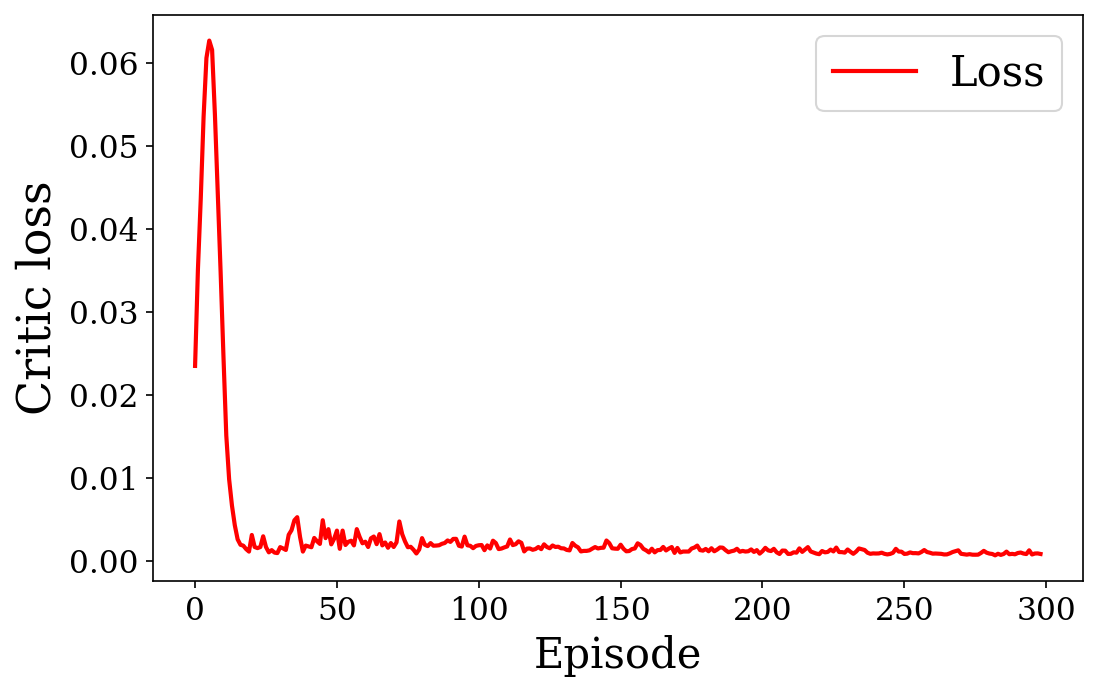}
}
\subfigure[Terminal loss] {
 \label{0 to 2 terminal}
\includegraphics[width=0.22\columnwidth]{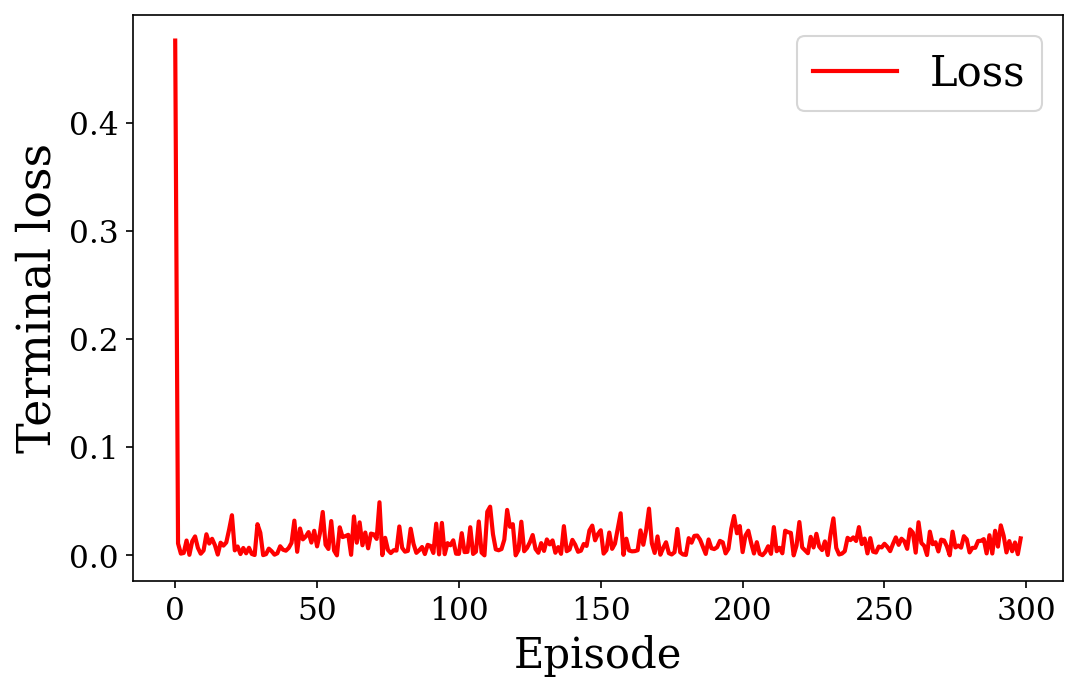}
}

\subfigure[The most probable transition pathway from 0 to 6] {
\label{0 to 6}
\includegraphics[width=0.22\columnwidth]{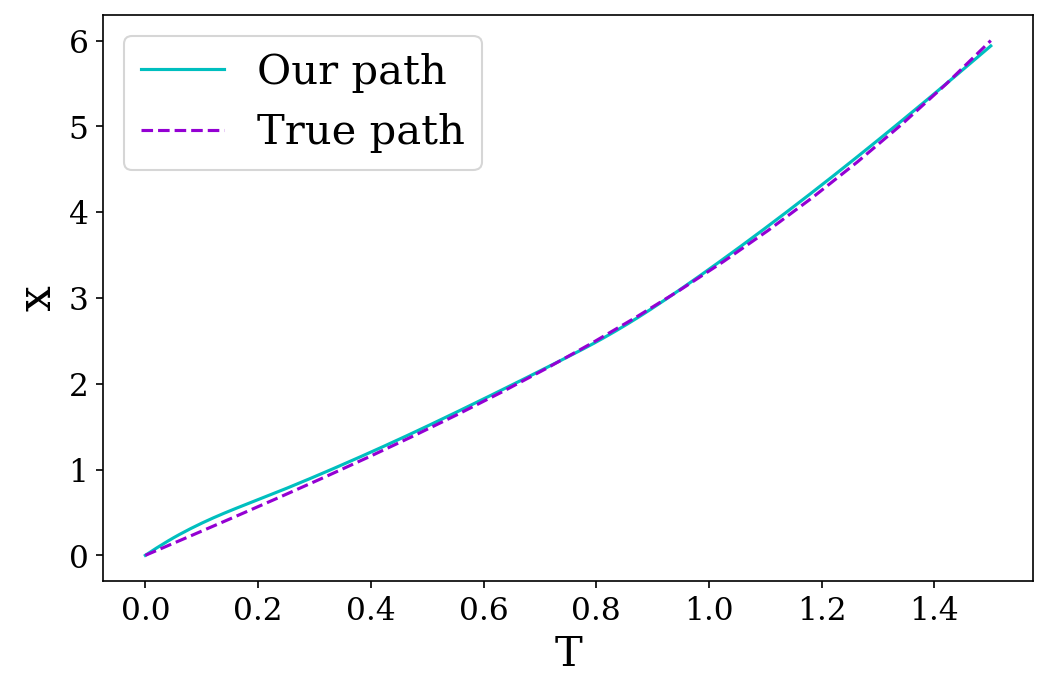}
}
\subfigure[Accumulative running cost] {
\label{0 to 6 actor}
\includegraphics[width=0.22\columnwidth]{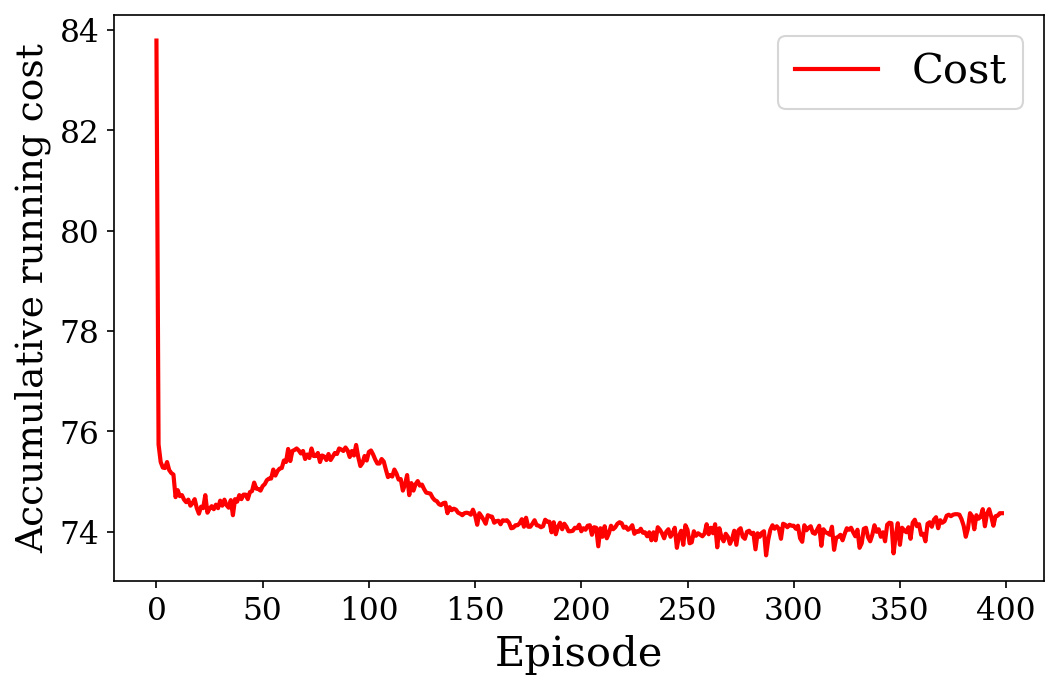}
}
\subfigure[Critic loss] {
\label{0 to 6 critic}
\includegraphics[width=0.22\columnwidth]{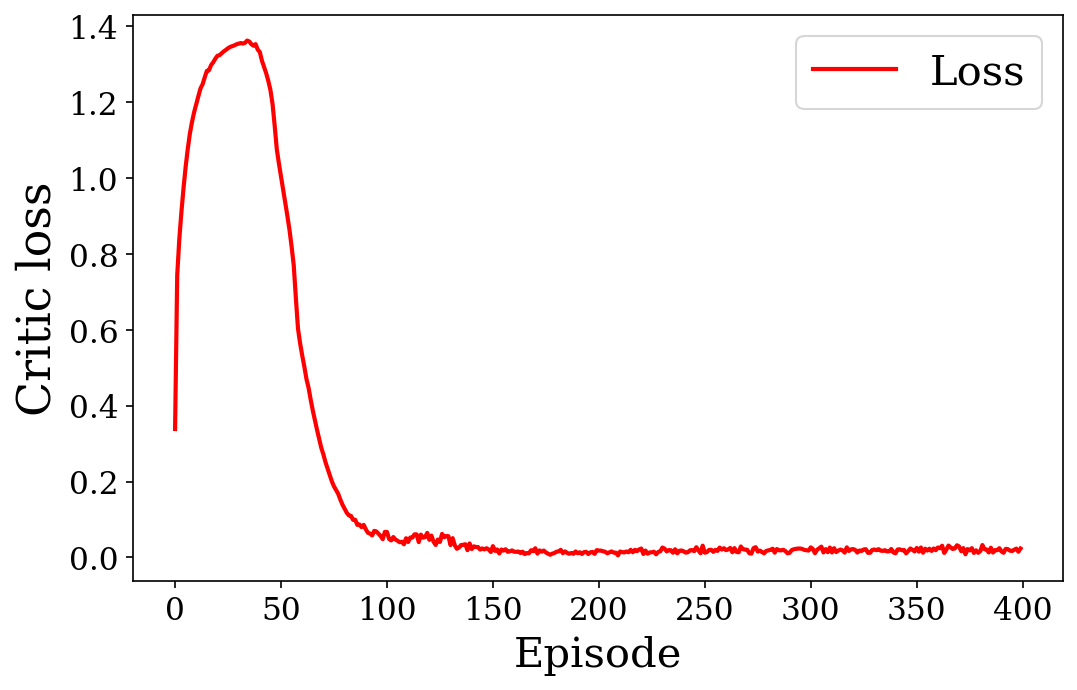}
}
\subfigure[Terminal loss] {
\label{0 to 6 terminal}
\includegraphics[width=0.22\columnwidth]{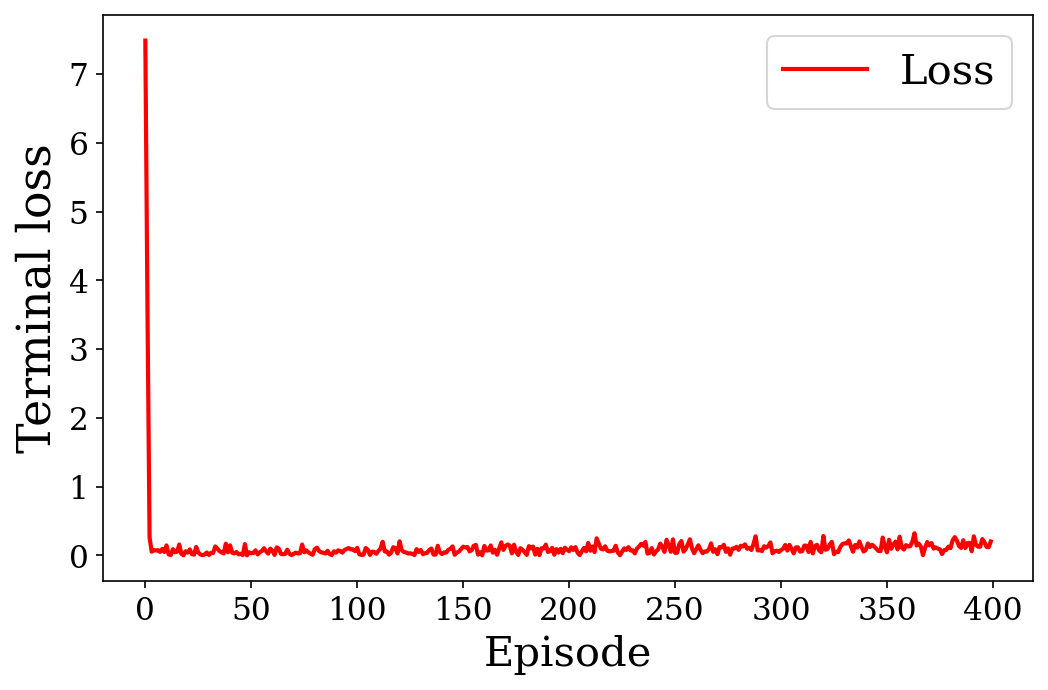}
}
\caption{The most probable transition pathway for the linear potential stochastic system: (a) departure from point $x_0=0$ to $x_1=2$, time $T=1.0$; (e) departure from point $x_0=0$ to $x_1=6$, time $T=1.5$. The blue line stands for the transition path learned by TP-DDPG, and the purple line represents the true solution.}
\label{departure from 0}
\end{figure*}

The accumulative running cost $\sum_{n=0}^{N-1}r_n=\sum_{n=0}^{N-1}f(X_n,a_n)\Delta t$ represents the Onsager–Machlup action functional, Figure \ref{0 to 2 actor}, \ref{0 to 6 actor} show that the running cost converges in 90-th and 160-th episodes for two cases.

Figure \ref{0 to 2 critic}, \ref{0 to 6 critic} shows the critic loss with respects to episodes, it can be seen that the critic loss reaches convergence at 130-th episodes and 170-th episodes for the two cases respectively. We find that the critical loss increases at first, possibly because the terminal loss dominates the actor network's update during this time. This leads the actor network to update toward minimizing terminal loss rather than decreasing the Q value estimated by the critic network.

The terminal loss in Figure \ref{0 to 6 terminal}, Figure \ref{0 to 2 terminal} show the $\mathcal{L}_{pred}=g(\hat{s}_N^{0,A})$ over episodes, which is the predicted terminal cost in timestep $n=0$ for every episode. That shows our algorithm can force the agent to reach the terminal point quickly in several episodes.

\begin{figure*}[htb]
\centering
\subfigure[The most probable transition pathway] {
 \label{beta=1 path}
\includegraphics[width=0.49\columnwidth]{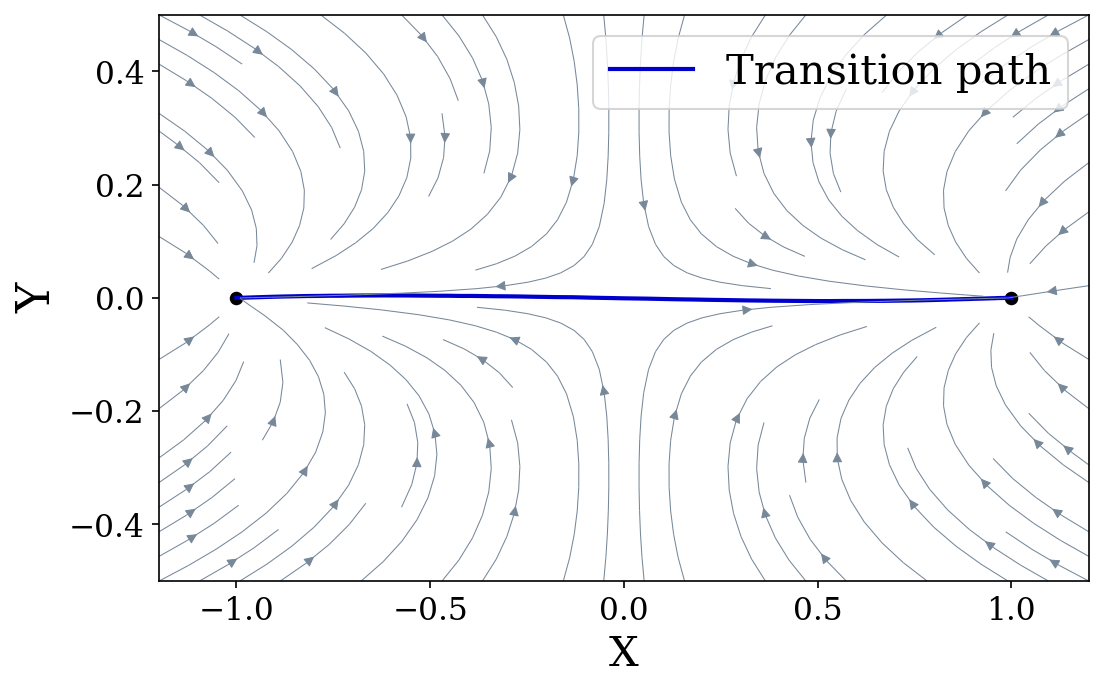}
}
\subfigure[Accumulative running cost] {
\label{beta=1 actor}
\includegraphics[width=0.48\columnwidth]{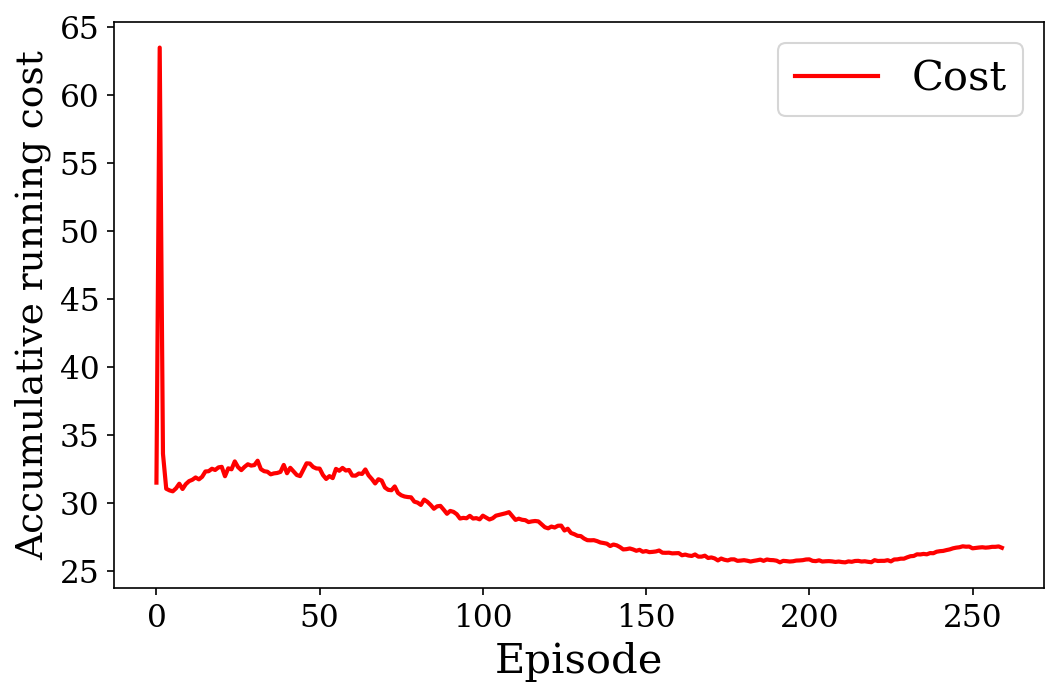}
}

\subfigure[Critic loss] {
\label{beta=1 critic}
\includegraphics[width=0.48\columnwidth]{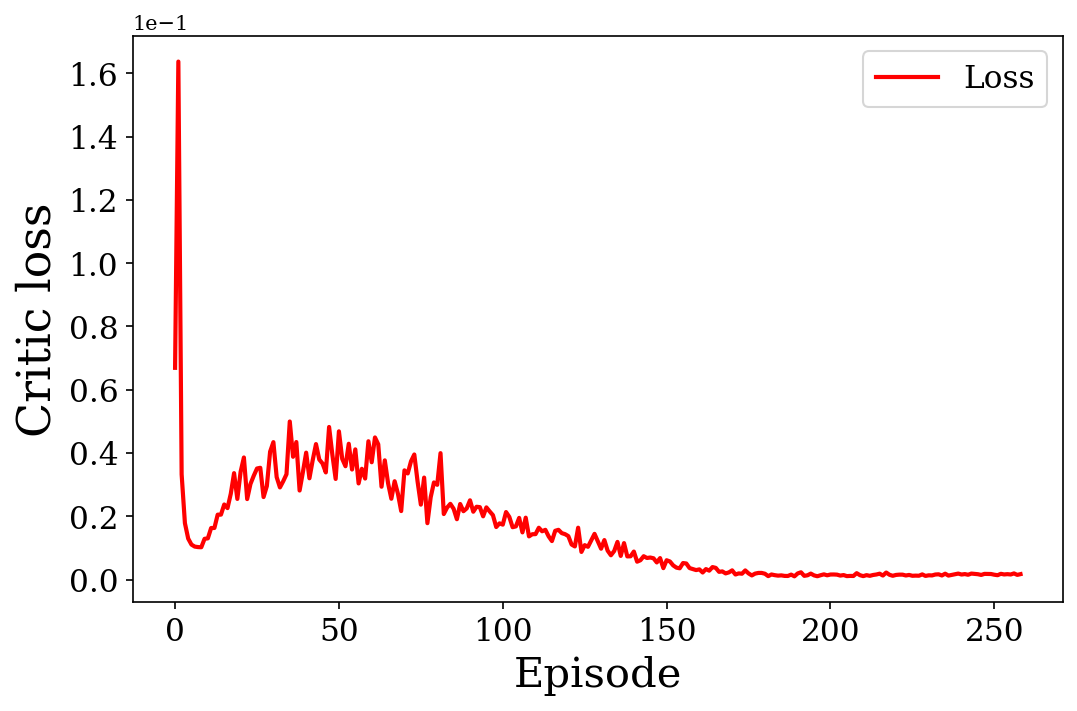}
}
\subfigure[Terminal loss] {
 \label{beta=1 terminal}
\includegraphics[width=0.48\columnwidth]{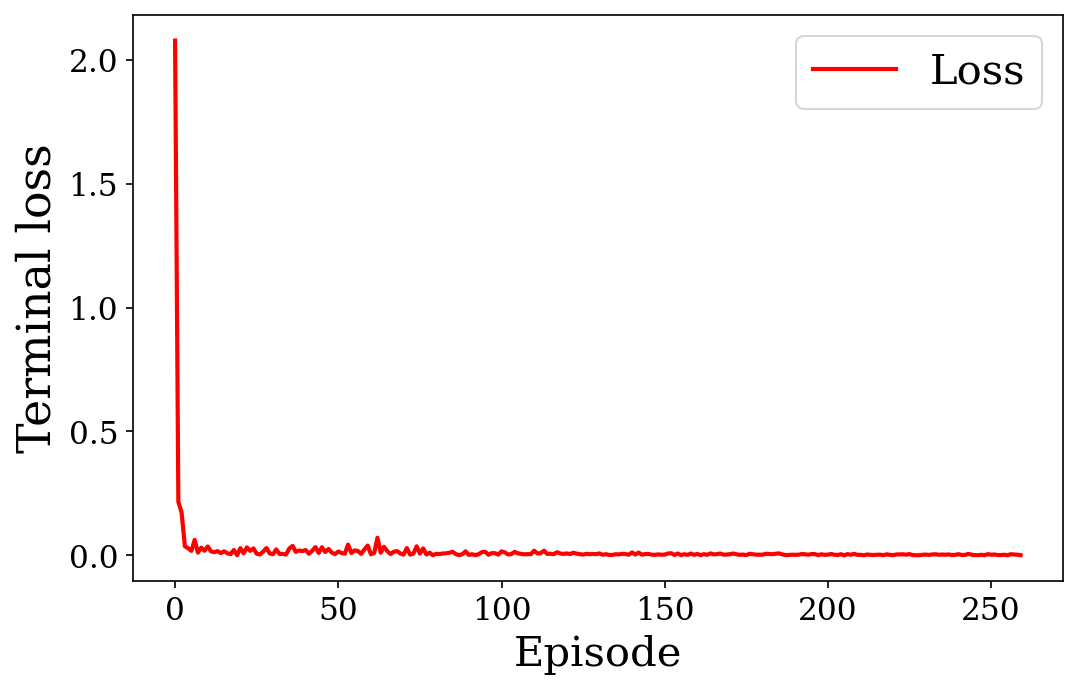}
}
\caption{ Maier-Stein system with $\beta=1$, noise intensity $\varepsilon=0.15$ and $T=5$.}
\label{Maier-Stein system beta=1}
\end{figure*}

\subsection{The Maier–Stein System}
\setlength{\parindent}{2em} Consider the Maier–Stein system \cite{1993The} in $\mathbb{R}^2$ that
    \begin{equation}
        \begin{cases}
            dX_t=b(X_t)dt+\varepsilon dB_t\,,\\
            X_0=(-1,0),X_T=(1,0)\in\mathbb{R}^2\,,
        \end{cases}    
    \end{equation}
where 
$$
\begin{aligned}
	b(X_t)=\left(\begin{array}
	{c}
        x - x^3-\beta xy^2 \\
	-(1+x^2)y
	\end{array}\right)\,,
	  \end{aligned}
$$
and the $\varepsilon$ is a positive constant representing the noise
intensity. Convert the problem into an optimal control problem (\ref{optimal control problem}), with the Onsager–Machlup action functional by (\ref{OM functional})
$$S^{OM}(X,\dot{X})=\frac{1}{2}\int_0^T [|u_t|^2+\nabla\cdot b(X_t)]dt\,.$$
This system has two metastable points $(\pm 1, 0)$ and one saddle point $(0, 0)$.

For $\beta=1$, we choose $T=5$, noise intensity $\varepsilon=0.15$ and the number of time interval $N=100$ to obtain the most probable transition pathway by TP-DDPG. We present the computed most probable transition pathway corresponding to the average of 200th-260th episodes in Figure \ref{beta=1 path}, which is a straight line through the saddle point $(0,0)$, connecting two metastable points $(-1,0)$ and $(1,0)$. Figure \ref{beta=1 actor} stands for the accumulative running $\sum_{t=0}^Tr(s_t,a_t)=\sum_{t=0}^T(a_t^2+\nabla \cdot b(s_t))\Delta t$ with respect to the episodes. Obviously, the accumulative running cost decreases as the number of episodes increases, which is consistent with the goal of our algorithm. That means the Onsager–Machlup action functional is getting smaller until converging. Figure \ref{beta=1 critic} shows the critic loss with respect to episodes, it clearly meets convergence at the 170-th episodes. That means the critic is able to approximate the accumulative running cost exactly after 170 episodes. From Figure \ref{beta=1 terminal} we can see that the agent could reach the terminal point robustly after several episodes.

\begin{figure*}[htb]
\centering
\subfigure[The most probable transition pathway] {
 \label{Maier-Stein system path}
\includegraphics[width=0.48\columnwidth]{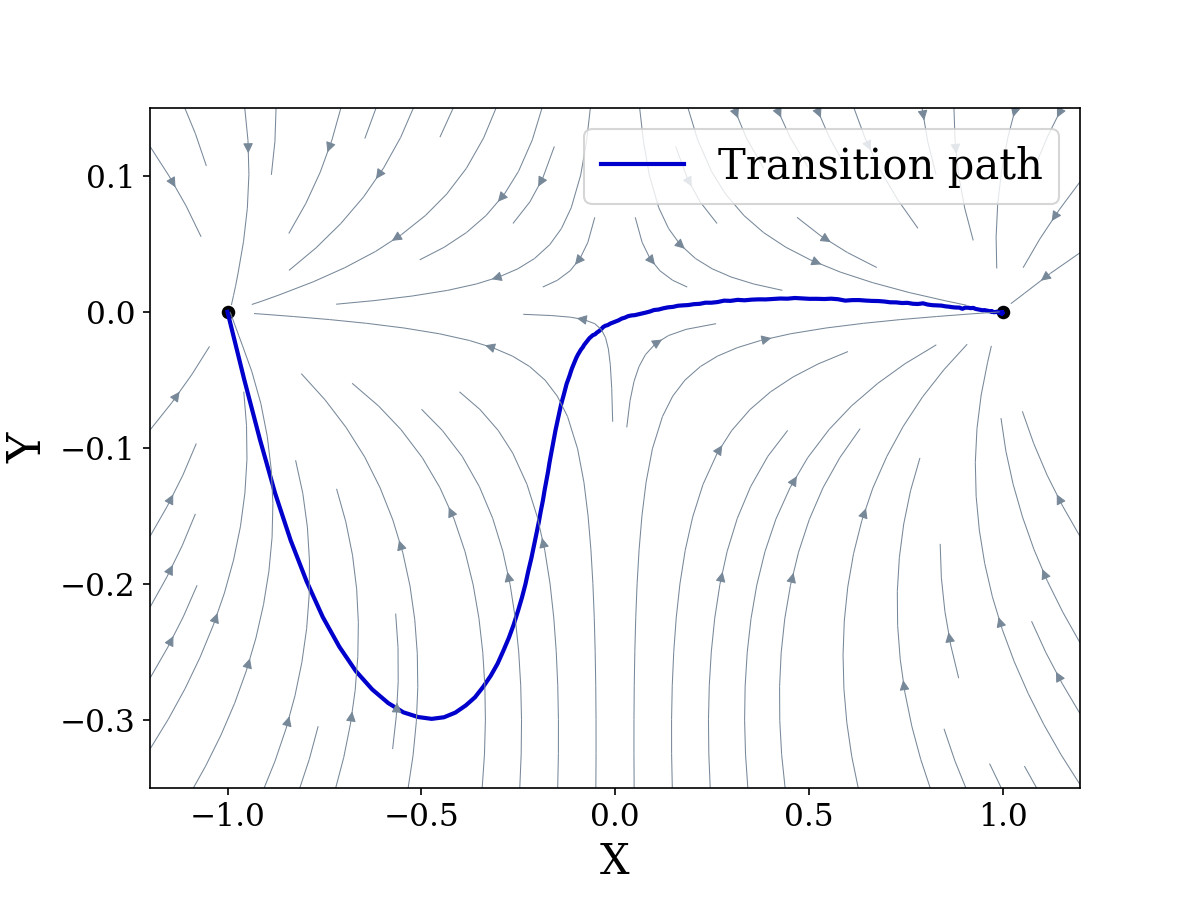}
}
\subfigure[Accumulative running cost] {
\label{Maier-Stein actor}
\includegraphics[width=0.48\columnwidth]{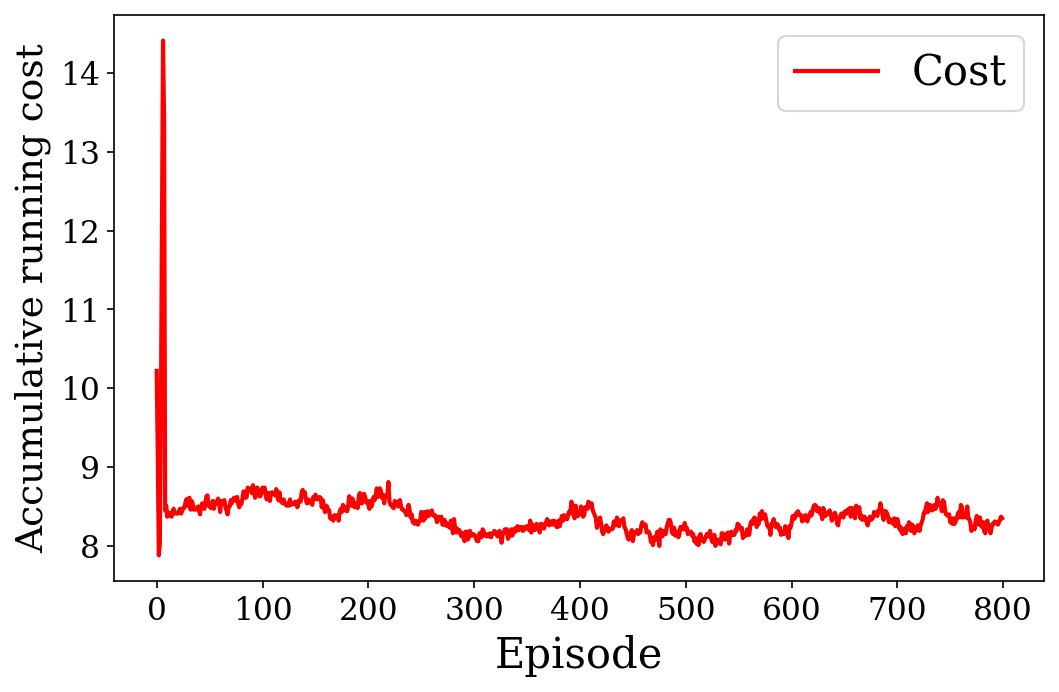}
}

\subfigure[Critic loss] {
\label{Maier-Stein critic}
\includegraphics[width=0.465\columnwidth]{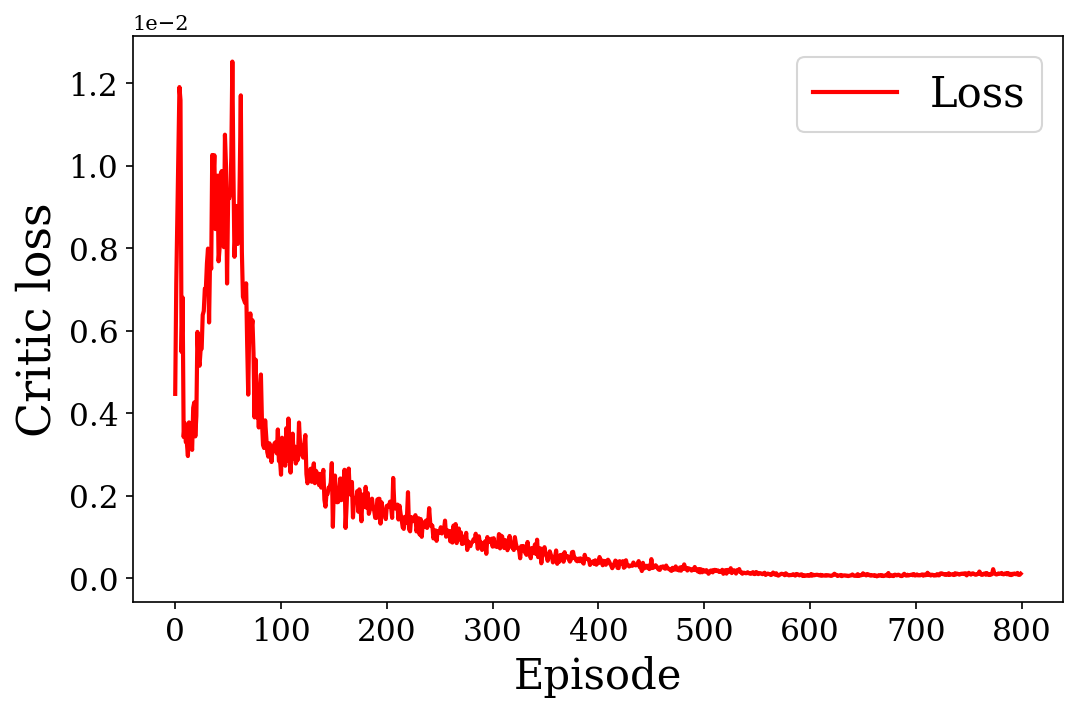}
}
\subfigure[Terminal loss] {
 \label{Maier-Stein terminal}
\includegraphics[width=0.48\columnwidth]{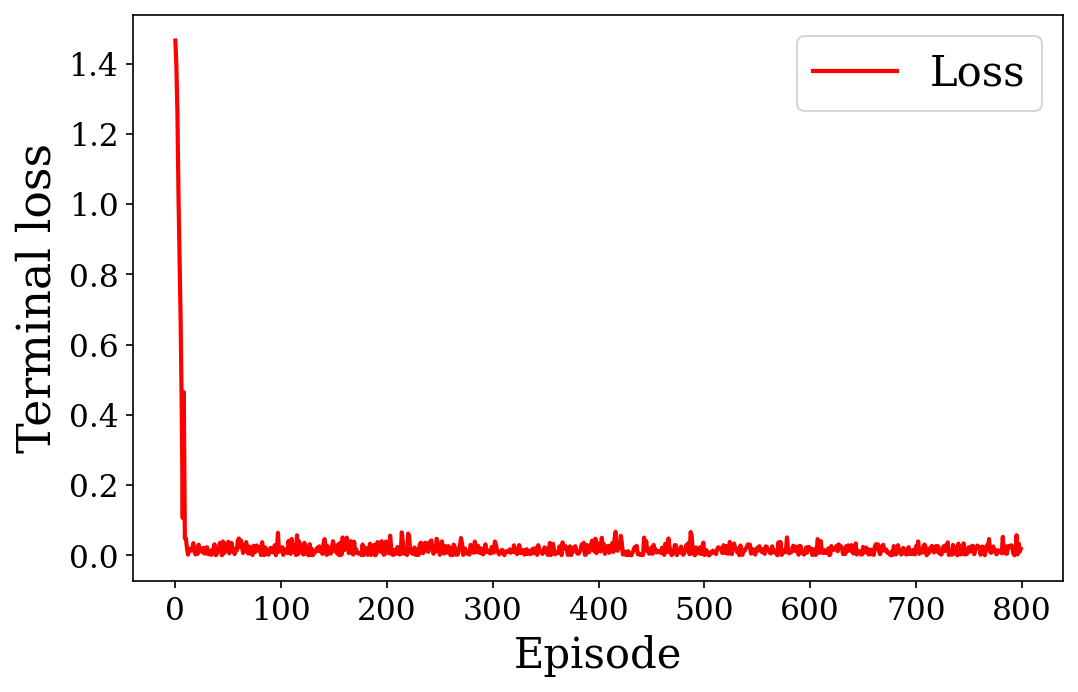}
}
\caption{ Maier-Stein system with $\beta=10$, noise intensity $\varepsilon=0.2$, and $T=10$.}
\label{Maier-Stein system}
\end{figure*}

For $\beta=10$, we choose $T=10$, noise intensity $\varepsilon=0.2$, and the number of time interval $N=200$ to obtain the most probable transition pathway by TP-DDPG. The results are shown in figure \ref{Maier-Stein system}. We present the pathway corresponding to the average of 600th-800th episodes as the critic loss becomes steadier after 500 episodes. In figure \ref{Maier-Stein system path}, the transition path departs from metastable point $(-1,0)$, passes through saddle point $(0,0)$ and then reaches metastable point $(1,0)$. Figure \ref{Maier-Stein terminal} shows that the transition pathway can reach the metastable point $x_1=(1,0)$ in 10 episodes with the confine of the terminal loss.

In the following, we consider the effects of the number of time interval $N$ on the value of terminal loss, under fixed $T$ and noise intensity $\varepsilon$, where $\beta=1$. The green dashed line, the pink dashed line, the blue dashed line and the red dashed line in the Figure \ref{terminal compare} correspond to the variation of the terminal prediction loss with the interactions of training at $N=20, 40, 80$ and $160$ respectively. Obviously, the terminal loss converges slower as $N$ increases. For $N=20$, the corresponding terminal loss converges earlier than the other three. Terminal loss of $N=160$ has the slowest convergence among these four.

Moreover, for $N=20,40,80,160$ and $\beta=1$, we demonstrate three statistical indicators of terminal loss in 20th-100th episodes in Table \ref{mean std},  considering the average terminal loss of $n=0,1,\cdots, 10$. Here all four scenarios have been converged after the 20-th episode. This suggests that the terminal loss will converge to a smaller value for the same timestep $n$ as $N$ decreases. That is consistent with the convergence analysis related to the terminal prediction loss $\Big(\frac{\gamma_M\sqrt{\log{(M)}}}{\sqrt{M}}[g]_L\frac{\rho_M-\rho_M^{N-n+1}}{1-\rho_M}\Big)^{\frac{1}{2n}}$ in Theorem \ref{theorem 1}, which implies that when the number of time intervals $N$ increases, the terminal prediction loss also increases under a given fixed value of $n$.

\begin{figure}[htb]
    \centering
    \includegraphics[width=0.8\columnwidth]{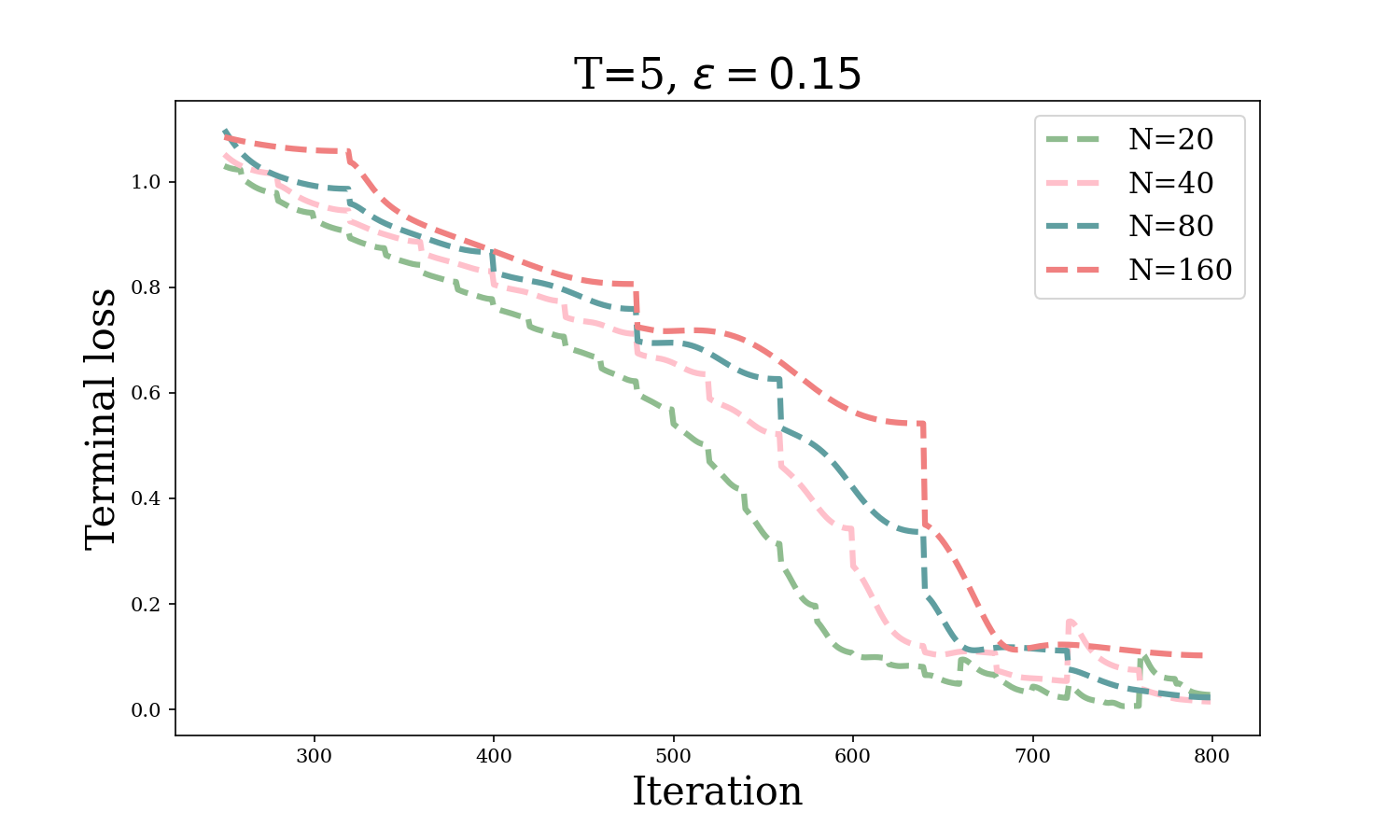}
    \caption{Terminal loss with respect to the number of predictions for different $N$ under the same $T$ and noise intensity.}
    \label{terminal compare}
\end{figure}

\begin{table*}[t]
    \centering
    \caption{The mean, standard deviation and maximum value of terminal loss for different $N$.}
    \label{mean std}
    \setlength{\tabcolsep}{6mm}
    \begin{tabular}{ccccccc}
\toprule
\textbf{\textbf{T}}   & $\bm{\varepsilon}$            & \textbf{N}  & \textbf{Mean}  & \textbf{Std} & \textbf{Max} \\ \midrule
\specialrule{0em}{1pt}{1pt}
\multirow{4}{*}{5} & \multirow{4}{*}{0.15} & 20 & $2.147\times 10^{-3}$ & $1.348\times 10^{-3}$  & $8.20\times 10^{-3}$  \\
\specialrule{0em}{1pt}{1pt}
                    &  & 40 & $2.637\times 10^{-3}$    & $2.536\times 10^{-3}$  & $1.38\times 10^{-2}$  \\
                    \specialrule{0em}{1pt}{1pt}
                    &  & 80 & $2.835\times 10^{-3}$    & $2.128\times 10^{-3}$  & $1.045\times 10^{-2}$  \\ 
                    \specialrule{0em}{1pt}{1pt}
                    &  & 160 & $3.424\times 10^{-3}$   &  $2.455\times 10^{-3}$ & $1.20\times 10^{-2}$  \\
                    \specialrule{0em}{1pt}{1pt}
\hline
\specialrule{0em}{1pt}{1pt}
\multirow{4}{*}{10} & \multirow{4}{*}{0.15} & 20 & $9.445\times 10^{-4}$ & $6.075\times 10^{-4}$  & $3.35\times 10^{-3}$  \\
\specialrule{0em}{1pt}{1pt}
                    &  & 40 & $1.812\times 10^{-3}$    & $1.184\times 10^{-3}$  & $7.65\times 10^{-3}$  \\
                    \specialrule{0em}{1pt}{1pt}
                    &  & 80 & $2.410\times 10^{-3}$    & $1.602\times 10^{-3}$  & $8.4\times 10^{-3}$  \\ 
                    \specialrule{0em}{1pt}{1pt}
                    &  & 160 & $2.984\times 10^{-3}$   &  $2.134\times 10^{-3}$ & $1.13\times 10^{-2}$  \\
\bottomrule
\end{tabular}
\end{table*}

\begin{figure*}[htb]
\centering
\subfigure[The most probable transition pathway] {
 \label{3 dim transition}
\includegraphics[height=0.38\columnwidth]{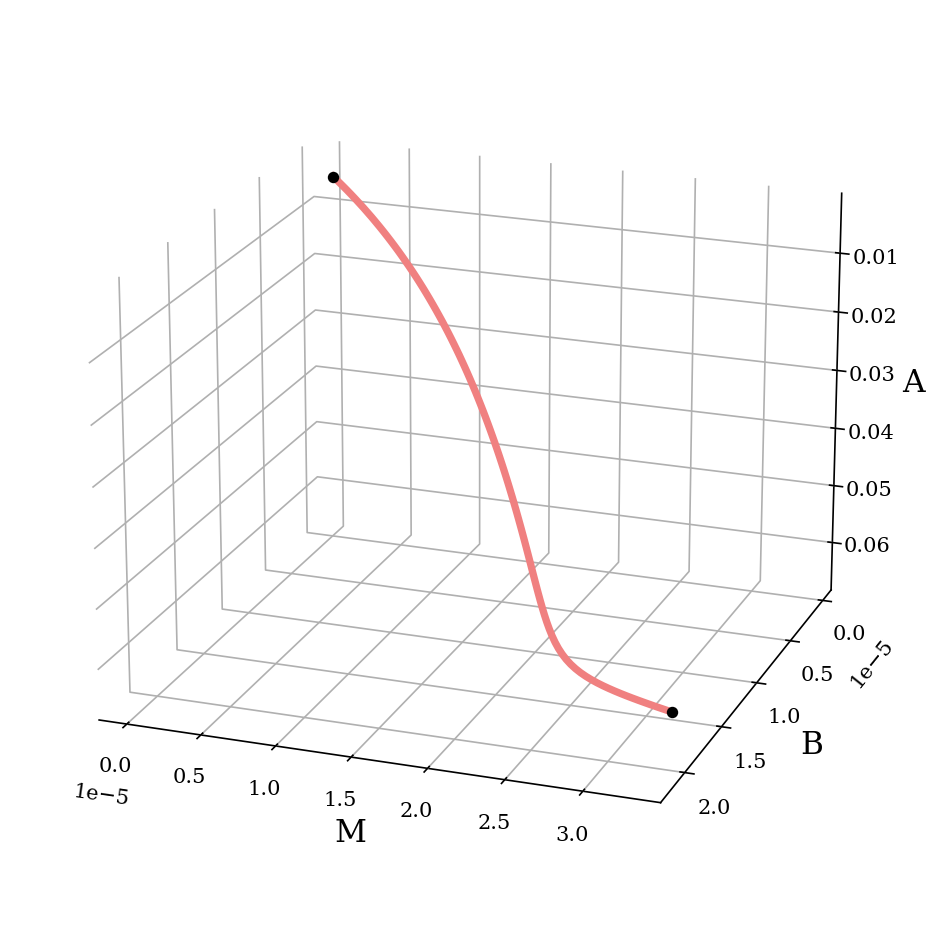}
}
\subfigure[Accumulative running cost] {
\label{3 dim total reward}
\includegraphics[width=0.48\columnwidth]{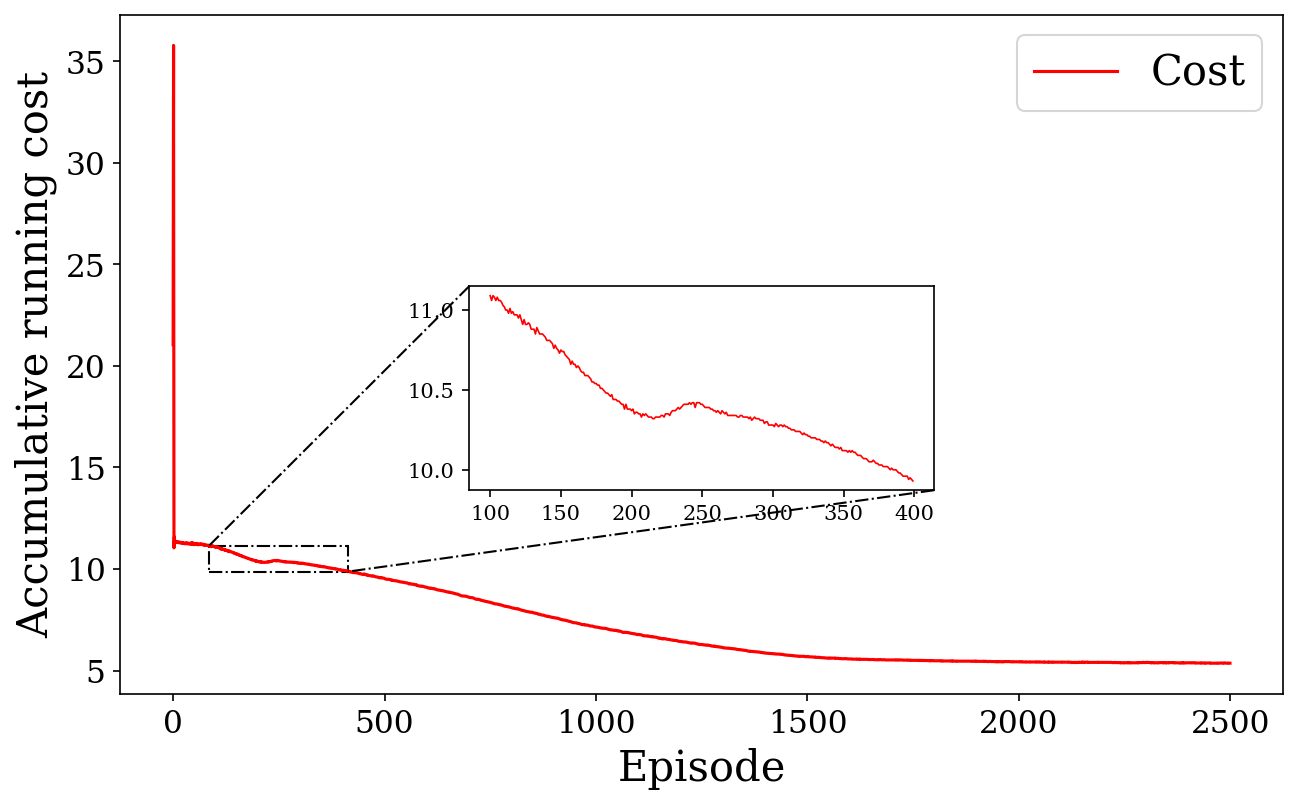}
}

\subfigure[Critic loss] {
\label{3 dim critic}
\includegraphics[width=0.48\columnwidth]{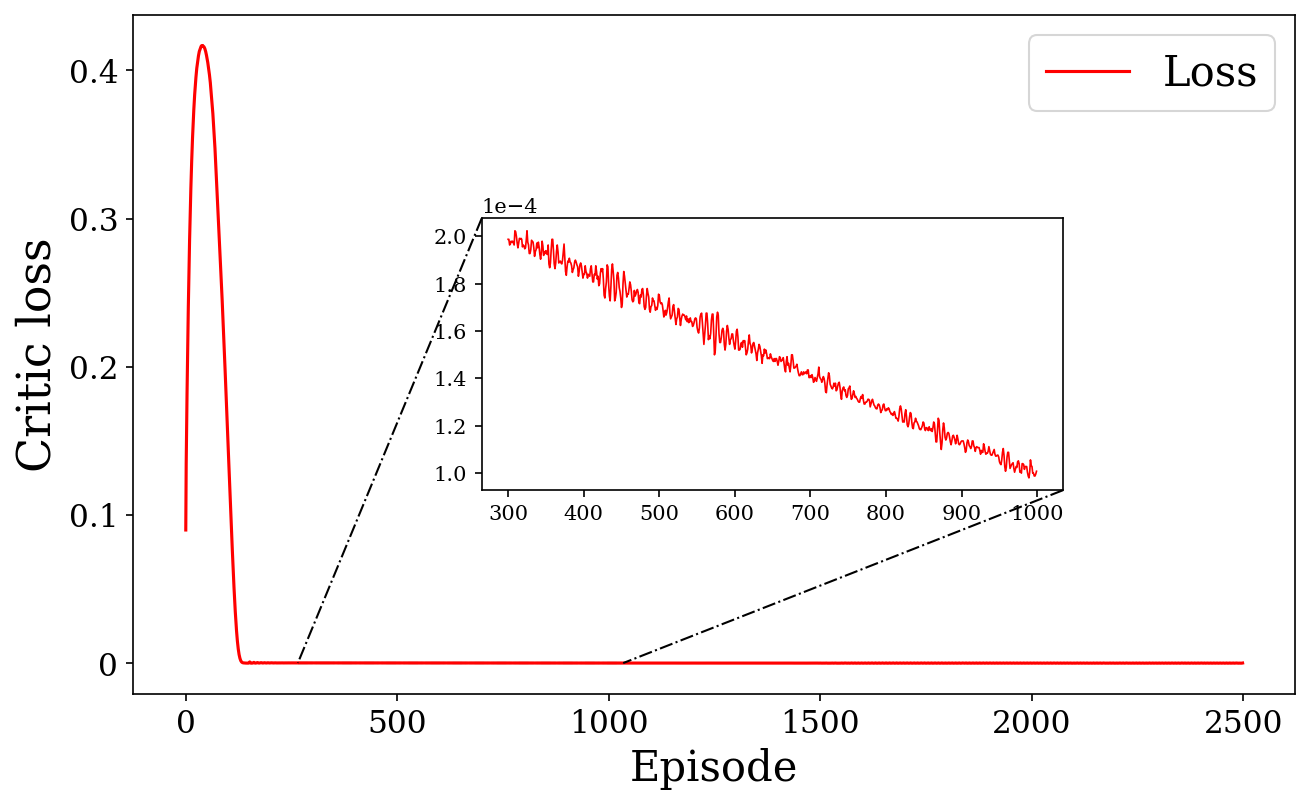}
}
\subfigure[Terminal loss] {
 \label{3 dim terminal}
\includegraphics[width=0.48\columnwidth]{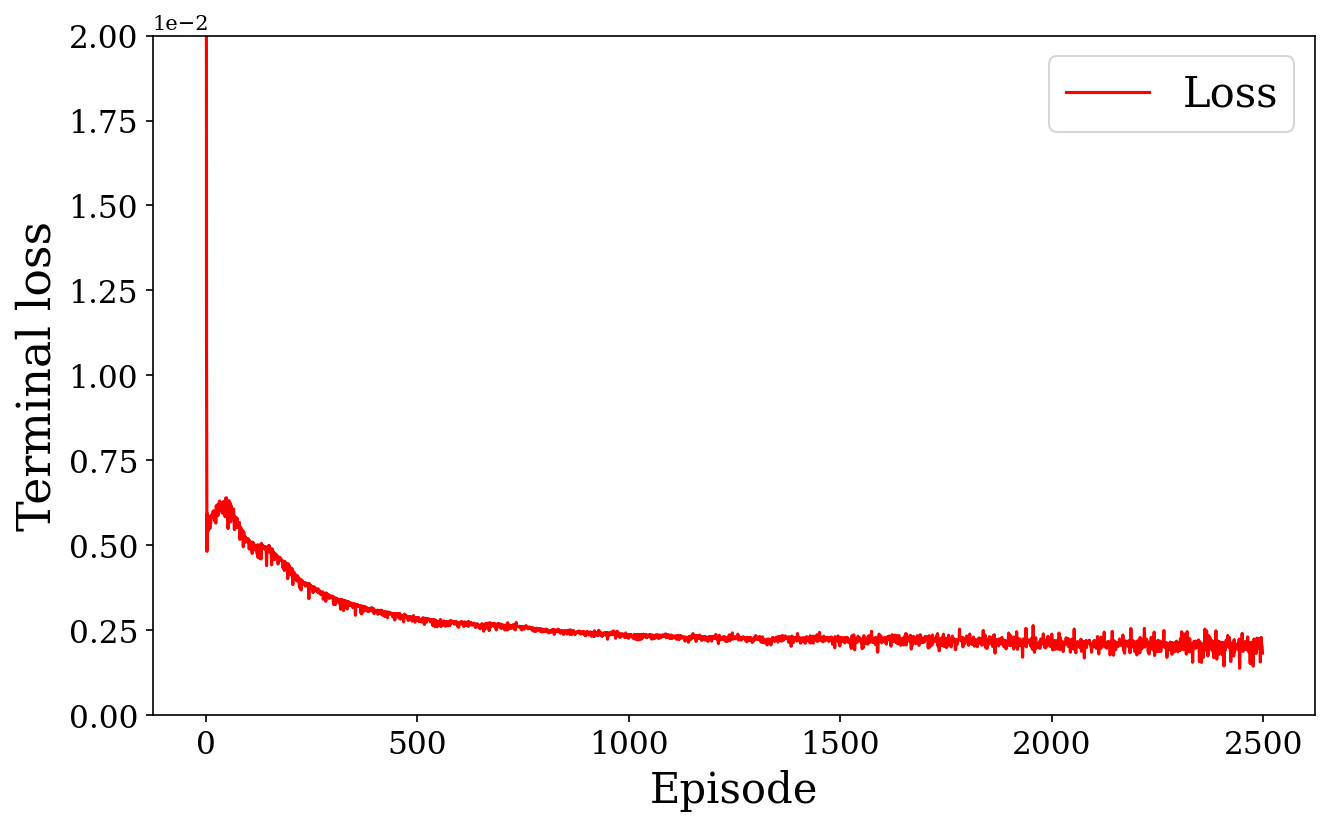}
}
\caption{Lactose operon model with noise intensity $\varepsilon=0.01$ and $T=3$.}
\label{3 dim}
\end{figure*}

\begin{figure}[htb]
    \centering
    \includegraphics[width=1.0\columnwidth]{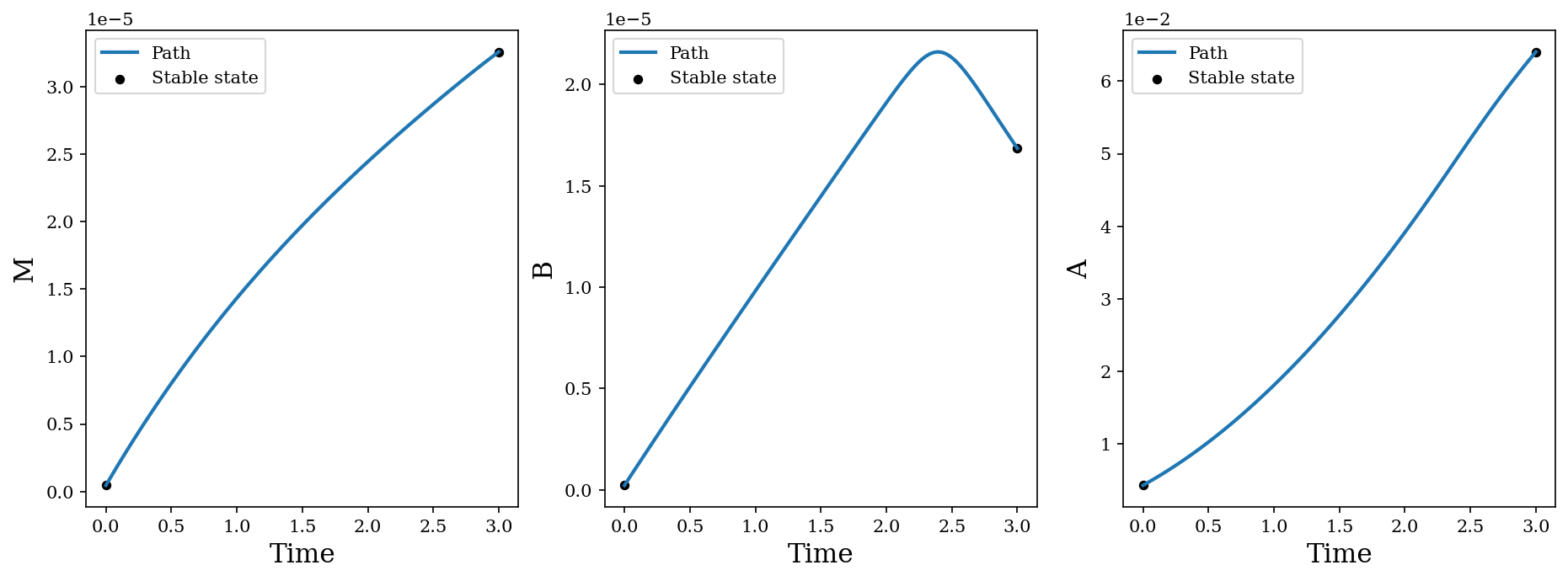}
    \caption{The three components of transition pathway over time.}
    \label{3 dim in 3}
\end{figure}

\begin{figure}[htb]
    \centering
    \includegraphics[width=0.58\columnwidth]{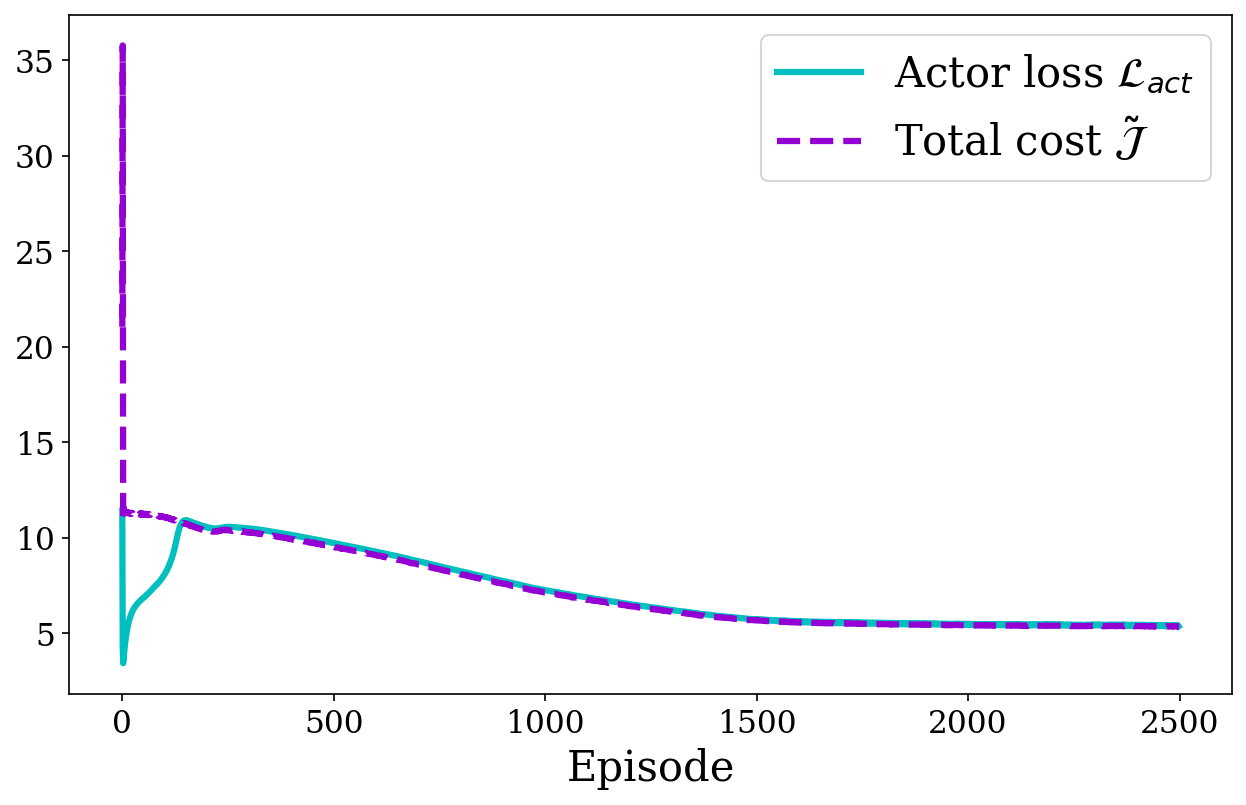}
    \caption{The actor loss $\mathcal{L}_{act}$ and total cost $\mathcal{\tilde{J}}$ with respect to episodes, where the actor loss is defined in (\ref{define actorl loss}), the total cost is defined in (\ref{define total cost}).}
    \label{compare Q and J}
\end{figure}

\subsection{Lactose Operon Model}
\setlength{\parindent}{2em} Yildirim and Mackey \cite{yildirim2003feedback} discover the bistability in the lactose operon dynamics. Yildirim and Necmettin \cite{yildirim2004dynamics} simplify the model by investigating only the function of the role of $\beta$-galactosidase in the operon regulation and ignoring that of lactose permease. Denote $M$ as mRNA concentration, $\beta$ as the galactosidase concentration, and A represents the concentration of allolactose. Consider the following three-dimensional dynamical system,
\begin{large}
\begin{equation}
    \begin{cases}
        \displaystyle\frac{dM}{dt}=\alpha_M\displaystyle\frac{1+K_1(e^{-\mu \tau_M}A_{\tau_M})^n}{K+K_1(e^{-\mu \tau_M}A_{\tau_M})^n}-\tilde{\gamma}_MM+\varepsilon dB_t\,,\\
        \\
        \displaystyle\frac{dB}{dt}=\alpha_B\displaystyle e^{-\mu \tau_B}M_{\tau_B}-\tilde{\gamma}_BB+\varepsilon dB_t\,,\\
        \\
        \displaystyle\frac{dA}{dt}=\alpha_AB\displaystyle\frac{L}{K_L+L}-\beta_AB\displaystyle\frac{A}{K_A+A}-\tilde{\gamma}_AA+\varepsilon dB_t\,,
    \end{cases}\label{lactose operon model}
\end{equation}
\end{large}
where $\tilde{\gamma}_M=\gamma_M+\mu$, $\tilde{\gamma}_B=\gamma_B+\mu$, $\tilde{\gamma}_A=\gamma_A+\mu$, and the $\varepsilon$ is a positive constant representing the noise intensity. The parameters are shown in Table \ref{parameter}.

The lactose operon dynamics is a bistable system, which is a significant complicated feature resulting from the nonlinearity of the associated processes. Investigating the transition pathways of the bistability (\ref{lactose operon model}) can help us to understand the dynamic behavior of this molecular regulatory system. 

 By formulating the problem as an optimal control problem (\ref{optimal control problem}), we can derive the Onsager-Machlup action functional
 $$S^{OM}(X,\dot{X})=\frac{1}{2}\int_0^T [|u_t|^2+\nabla\cdot b(X_t)]dt$$
 through the application of (\ref{OM functional}). The lactose operon model (\ref{lactose operon model}) has two stable states $(M_1^*,B_1^*,A_1^*)=(4.57\times 10^{-7}, 2.29\times 10^{-7}, 4.27\times 10^{-3})$ and $(M_2^*,B_2^*,A_2^*)=(3.28\times 10^{-5}, 1.65\times 10^{-5}, 6.47\times 10^{-2})$.
 
 Figure \ref{3 dim transition} shows the most probable transition pathway which represents the average path from 2000-th to 2500-th episodes from $(M_1^*,B_1^*,A_1^*)$ to $(M_2^*,B_2^*,A_2^*)$ determined by our method in three-dimensional space. The accumulative running cost, critic loss and terminal loss with respect to episodes are shown in Figure \ref{3 dim total reward}, \ref{3 dim critic} and \ref{3 dim terminal}. It shows that our accumulative running cost gets convergence after 1600 episodes, which stands for that we find a transition pathway corresponding to a minimization of the Onsager–Machlup Action Functional. The critic loss and terminal loss converge after 200-th and 1500-th episodes respectively. Moreover, we illustrate the three components of the transition pathway over time in Figure \ref{3 dim in 3}. The left one depicts the evolution of the mRNA concentration over time, the middle one and the right one display the changes in the concentration of $\beta $ galactosidase and allolactose as time goes by. In Figure \ref{compare Q and J}, we illustrate the actor loss defined in (\ref{define actorl loss}) as $\mathcal{L}_{act}=\hat{Q}(s_0,a_0)+\mathcal{L}_{pred}$ and the total cost defined in (\ref{define total cost}) as $\tilde{\mathcal{J}}=\sum_{i=0}^{N-1}f(X_i,u_i)\Delta t+g(X_N)$ with respect to episodes. It obvious that our networks could estimate the total cost $\tilde{\mathcal{J}}$ exactly after 500 episodes. This demonstrates that the TP-DDPG algorithm can converge to the minimum of this finite control issue (\ref{optimal control problem}).

\begin{table*}[htb]
\centering
\caption{The parameters for reduced lactose operon model.}
\label{parameter}
\begin{tabular}{cc}
   \toprule
   Parameter & Value  \\
   \midrule
   $\mu_{max}$ & $3.47\times 10^{-2}\,\,min^{-1}$ \\
   $\mu$ & $3.03\times 10^{-2}\,\,min^{-1}$ \\
   $\alpha_M$ & $997\,\, nM-min^{-1}$ \\
   $\alpha_B$ & $1.66\times 10^{-2}\,\,min^{-1}$\\
   $\alpha_A$ & $1.76 \times 10^{4}\,\,min^{-1}$\\
   $\gamma_M$ & $0.411\,\,min^{-1}$\\
   $\gamma_B$ & $8.33 \times 10^{-4}\,\,min^{-1}$ \\
   $\gamma_A$ & $1.35\times 10^{-2}\,\,min^{-1}$ \\
   $n$ & $2$ \\
   $K$ & $7200$ \\
   $K_1$ & $2.52\times 10^{-2}\,(\mu M)^{-2}$ \\
   $K_L$ & $0.97 \,mM$\\
   $K_A$ & $1.95 \,mM$\\
   $\beta_A$ & $2.15\times 10^4 \,\,min^{-1}$\\
   $\tau_M$ & $0.10\,min$\\
   $\tau_B$ & $2.00\,min$\\
   \bottomrule
\end{tabular}
\end{table*}

\section{Conclusion}

\setlength{\parindent}{2em} We propose an algorithm based on deep reinforcement learning that can be utilized to solve the finite-horizon control problem in a forward way, and this further leads to a new method to compute the most probable transition paths in stochastic dynamical systems. In this paper, deep reinforcement learning is used in conjunction with the terminal prediction method to deal with the most probable transition pathway problem in stochastic dynamical systems. The Onsager–Machlup functional measures the probability of rare events and we convert the transition path issue into an optimal control problem with a finite horizon. In addition, the convergence analysis of the proposed algorithm is performed. Moreover, we conduct three experiments to illustrate the effectiveness of our algorithm. In the one-dimensional linear potential stochastic system, we compare the transition path obtained by our algorithm and the true solution of two-point boundary value problem. The result illustrates that our method works well in catching the stochastic dynamic system’s most probable transition pathway. In the three-dimensional lactose operon model, it can be seen that our method can be utilized to deal with high dimensional transition pathway problems. Compared to the method in \cite{wei2022optimal}, our approach requires fewer computational costs and works more efficiently. This method in \cite{wei2022optimal} requires a grid for the full space including time intervals with the number of residual points $N_T = 10000$, while our method solely requires the partitioning of time intervals much less than $N_T$ (Table \ref{mean std}). For different stochastic dynamical systems, the method proposed in \cite{chen2023detecting} requires deriving different Pontryagin’s Maximum Principle associated with the corresponding systems, while we could handle these by directly substituting the environment function of TP-DDPG.

\section*{ACKNOWLEDGMENTS}
We would like to thank Huifang Huang, Pengbo Li and Wei Wei for their helpful discussions. This work was supported by the National Key Research and Development Program of China (No. 2021ZD0201300), the National Natural Science Foundation of China (No. 12141107), the Fundamental Research Funds for the Central Universities (5003011053).

\section{CRediT authorship contribution statemen}
\textbf{Jin Guo:} Conceptualization, Methodology, Software, Writing – original draft, Formal analysis. \textbf{Ting Gao:} Methodology, Software, Formal analysis, Writing - review \& editing, Resources, Funding Support. \textbf{Peng Zhang:} Software, Visualization, Methodology, Formal analysis. \textbf{Jiequn Han:} Methodology, Formal analysis, Writing - review \& editing. \textbf{Jinqiao Duan:} Project administration, Writing - review \& editing, Supervision, Funding Support.

\section{Declaration of competing interest}
The authors declare that they have no known competing financial interests or personal relationships that could have appeared to influence the work reported in this paper.


\bibliographystyle{unsrt}
\bibliography{main}

\begin{appendices}
\section{Appendix 1}\label{appendix 1}
\textbf{Lemma} \,\,\,Let $(X,Y)$ be a random variable. Assume $|Y|\leq L$ a.s. and let $m(x)=\mathbb{E}[Y|X=x]$. Assume $Y-m(X)$ is sub-Gaussian in the sense that $\mathop{\mbox{max}}\limits_{m=1,2,\cdots,M}c^2\mathbb{E}[e^{(Y-m(X))^2/c^2}-1|X]\leq\sigma^2$ a.s. for some $c,\sigma >0$. Let $\gamma_M, L\geq 1$ and assume that the regression function is bounded by $L$ and that $\gamma_M\xrightarrow{M\rightarrow +\infty} +\infty$. Set 
\begin{equation*}
    \hat{m}_M:=\mathop{\mbox{argmin}}_{\Phi\in \mathcal{Q}_M}\frac{1}{M}\sum_{m=1}^M|\Phi(x_i)-Y_m|^2
\end{equation*}
for some $\mathcal{Q}_M$ of functions $\Phi:\mathbb{R}\rightarrow [-\gamma_M,\gamma_M]$ and some random variable $Y_1,\cdots,Y_M$ which are bounded by $L$, and denote
$$\Omega_g:=\Big\{f - g : f \in \mathcal{Q}_M,\frac{1}{M}\sum^M_{m=1}|f(x_m) - g(x_m)|^2\leq \frac{\delta_M}{\gamma_M^2}\Big\}.$$
Then there exist constants $c_1,c_2>0$ which depend only on $\sigma$ and $c$ such that for any $\delta_M>0$ with
\begin{equation*}
\begin{aligned}
    \delta_M \xrightarrow{M\rightarrow +\infty} 0, \frac{M\delta_M}{\gamma_M} \xrightarrow{M\rightarrow +\infty} +\infty \,, \\
    c_1\frac{\sqrt{M}\delta}{\gamma_M^2}\geq 
    \int_{c_2\delta/\gamma^2_M}^{\sqrt{\delta}}log\Big(\mathcal{N}_2\Big(\frac{u}{4\gamma_M},\Omega_g,x_1^M\Big)\Big)^{\frac{1}{2}}du\,,
    \end{aligned}
\end{equation*}
for all $\delta\geq \delta_M$ and all $g\in \mathcal{Q}_M\cup \{m\}$ we have as $M\rightarrow +\infty$:
\begin{equation*}
    \mathbb{E}[|\hat{m}_M(X)-m(X)|^2]=\mathcal{O}_\mathbb{P}\Big(\delta_M+\inf_{\Phi\in\mathcal{Q}_M}\mathbb{E}[|\Phi(X)-m(X)|^2]\Big)\,.
\end{equation*}
\begin{proof}
    The proof is in \cite{kohler2006nonparametric} (Corollary 1).
\end{proof}

In the proof of \textbf{Theorem} \ref{theorem 1}, let $m(X)=\bar{Q}_{n+1}$,  $Y=-f(X_n,\hat{a}_n(X_n)+\xi_n)+\hat{Q}_n(X_n,\hat{a}_n(X_n)+\xi_n)$, and $\hat{m}(X)=\hat{Q}_{n+1}$. Then the following holds:
\begin{equation*}
    \begin{aligned}
        |Y-m(X)|&=|-f(X_n,\hat{a}_n(X_n)+\xi_n)+\hat{Q}_n(X_n,\hat{a}_n(X_n)+\xi_n)-\bar{Q}_{n+1}|\\
        &\leq |[f]_L\xi_n+\eta_M\gamma_M\xi_n|\,.
    \end{aligned}
\end{equation*}
Thus we have
\begin{equation*}
    \begin{aligned}
        c^2\mathbb{E}[e^{(Y-m(X))^2/c^2}-1|X]\leq  c^2\mathbb{E}[e^{([f]_L\xi_n+\eta_M\gamma_M\xi_n)^2/c^2}-1|X]
    \end{aligned}
\end{equation*}
Let $c=2\sigma_{act}([f]_L+\eta_M\gamma_M)$, since $\xi_n$ follows the distribution $N(0,\sigma^2_{act})$, we can obtain that
\begin{equation*}
    \begin{aligned}
         c^2\mathbb{E}[e^{([f]_L\xi_n+\eta_M\gamma_M\xi_n)^2/c^2}-1|X]&= 4\sigma^2_{act}([f]_L+\eta_M\gamma_M)^2\bigg[\int_{-\infty}^\infty \frac{1}{\sqrt{2\pi}\sigma}e^{\frac{x^2}{4\sigma^2_{act}}}\cdot e^{-\frac{x^2}{2\sigma^2_{act}}}dx-1\bigg]\\
        & =  4\sigma^2_{act}([f]_L+\eta_M\gamma_M)^2\bigg[\sqrt{2}\int_{-\infty}^\infty \frac{1}{\sqrt{2\pi}\sigma_{act}}e^{-\frac{x^2}{2\sigma^2_{act}}}dx-1\bigg]\\
        &=4\sigma^2_{act}([f]_L+\eta_M\gamma_M)^2(\sqrt{2}-1)\,.
    \end{aligned}
\end{equation*}
It is evident from the above analysis that $-f(X_n,\hat{a}_n(X_n)+\xi_n)+\hat{Q}_n(X_n,\hat{a}_n(X_n)+\xi_n)-\bar{Q}_{n+1}$ satisfies the sub-Gaussian condition.

\section{Appendix 2}\label{appendix 2}
\textbf{Lemma}\,\,\,For every $\varepsilon>0$, we have
\begin{equation*}
    \mathcal{N}_2\big(\varepsilon,\mathcal{Q}_M,(X_n^{(m)})_{1\leq m \leq M}\big)\leq\bigg(\frac{12e\gamma_M(K_M+1)}{\varepsilon}\bigg)^{(4(d+u)+9)K_M+1}\,,
\end{equation*}
where the class of neural networks $\mathcal{Q}_M$ is defined in section \ref{Convergence analysis}.
\begin{proof}
    We refer to Section 5 of \cite{kohler2010pricing} for proof. Replace $\mathcal{G}_2$ with $\mathcal{G}_2=\{\sigma_A(a^Tx +b): a \in\mathbb{R}^{d+u}, b \in \mathbb{R}\}$in the proof of Lemma 5.1 in \cite{kohler2010pricing}, where $\sigma_A$ is Arctan activate function, then we can obtain the above result.
\end{proof}

\section{Appendix 3}\label{appendix 3}
As the same procedure in Lemma \ref{lemma 1}, we introduce i.i.d. Rademacher random variables $(r_m)_{1\leq m\leq M}$. For $\varepsilon>0$, the following holds,
\begin{equation*}
\begin{aligned}
        &\mathbb{P}\bigg[\sup_{A\in\mathcal{A}_M}\Big|\frac{1}{M}\sum_{m=1}^M\big[f(X_n^{(m)},A(X_n^{(m)})-f(X_n^{(m)},\hat{a}_n(X_n^{(m)}))\big]\Big|>\varepsilon\bigg]\\
&\leq \mathbb{P}\bigg[\sup_{A\in\mathcal{A}_M}\Big|\frac{1}{M}\sum_{m=1}^M r_m\big[f(X_n^{(m)},A(X_n^{(m)})-f(X_n^{(m)},\hat{a}_n(X_n^{(m)}))\big]\Big|>\varepsilon\bigg]\,.
\end{aligned}
\end{equation*}
Consider a process subject to control 0 at and after timestep $n$, by (\ref{bound of A}) in \textit{Step 4} in the proof of Lemma \ref{lemma 1}, we have
\begin{equation*}
    \begin{aligned}        &\mathbb{E}\bigg[\sup_{A\in\mathcal{A}_M}\Big|\frac{1}{M}\sum_{m=1}^Mr_m\big[f(X_n^{(m)},A(X_n^{(m)})-f(X_n^{(m)},\hat{a}_n(X_n^{(m)}))\big]\Big|\bigg]\\
&\leq \mathbb{E}\bigg[\sup_{A\in\mathcal{A}_M}\Big|\frac{1}{M}\sum_{m=1}^Mr_m\big[f(X_n^{(m)},A(X_n^{(m)})-f(X_n^{(m)},0)\big]\Big|\bigg]+\mathbb{E}\bigg[\Big|\frac{1}{M}\sum_{m=1}^Mr_m\big[f(X_n^{(m)},\hat{a}_n(X_n^{(m)})-f(X_n^{(m)},0)\big]\Big|\bigg] \\
&\leq 2\mathbb{E}\bigg[\sup_{A\in\mathcal{A}_M}\Big|\frac{1}{M}\sum_{m=1}^Mr_m\big[f(X_n^{(m)},A(X_n^{(m)})-f(X_n^{(m)},0)\big]\Big|\bigg]\leq 2[f]_L\frac{1}{M}\mathbb{E}\Big[\sup_{A\in\mathcal{A}_M}\Big|\sum_{m=1}^Mr_m A(X_n^{(m)})\Big|\Big]\leq [f]_L\frac{2\gamma_M}{\sqrt{M}}\,.
    \end{aligned}
\end{equation*}

\section{Appendix 4} \label{proof of lemma 1}
The \textbf{proof of Lemma \ref{lemma 1}} is demonstrated in four steps.
 
\textbf{Step 1. Symmetrization by a ghost sample}. Define $(X^{'(m)}_k)_{1\leq m\leq M, 0\leq k\leq n}$ as a copy of $(X^{(m)}_k)_{1\leq m\leq M, 0\leq k\leq n}$, that means $(X^{'(m)}_k)_{1\leq m\leq M, 0\leq k\leq n}$ is generated from an independent copy of the exogenous noises $(\varepsilon^{'(m)}_k)_{1\leq m\leq M, 0\leq k\leq n}$. Moreover, $(X^{'(m)}_k)_{m=1}^M$ follows the same control $\hat{a}_k$ at time $k=0,1,\cdots,n-1$, and control $A$ at time n such that $X_{k+1}^{'(m)}=F(X_k^{'(m)},\hat{a}_k,\varepsilon^{'(m)}_k)$.

Take $\varepsilon >0$ and let $A^*\in\mathcal{A}_M$ satisfy 
\begin{equation}
    \begin{aligned}
        A^*=\mathop{argmax}\limits_{A\in \mathcal{A}_M}\Bigg|&\frac{1}{M}\sum_{m=1}^M\bigg[-f\big(X_n^{(m)},A(X_n^{(m)})\big)\Delta t+\hat{Q}_n\big(X_n^{(m)},A(X_n^{(m)}\big)+g\big(X_N^{(m),n,A}\big)]\\
        &-\mathbb{E}_M[-f(X_n,A(X_n))\Delta t+\hat{Q}_n(X_n,A(X_n)+g(X_N^{n,A})]\Bigg|>\varepsilon\,.
    \end{aligned}
\end{equation}

Applying Chebyshev’s inequality, we infer that
\begin{equation}
    \begin{aligned}
        \mathbb{P}_M\Bigg[\bigg|&\mathbb{E}_M\big[-f(X'_n,A^*(X'_n))\Delta t+\hat{Q}_n(X'_n,A^*(X'_n))+g(X_N^{'n,A^*})\big]\\
        &-\frac{1}{M}\sum_{m=1}^M\big[-f(X_n^{'(m)},A^*(X_n^{'(m)}))\Delta t+\hat{Q}_n(X_n^{'(m)},A^*(X_n^{'(m)}))+g(X_N^{'(m),n,A^*})\big]\bigg|>\frac{\varepsilon}{2}\Bigg]\\
&\leq \frac{1}{M(\varepsilon/2)^2}Var_M[-f(X_n^{'(m)},A^*(X_n^{'(m)}))\Delta t+\hat{Q}_n(X_n^{'(m)},A^*(X_n^{'(m)}))+g(X_N^{'(m),n,A^*})]\,.\label{Chebyshev’s inequality}
    \end{aligned}
\end{equation}

Since  $\Big|-f(X_n^{'(m)},A^*(X_n^{'(m)}))\Delta t+\hat{Q}_n(X_n^{'(m)},A^*(X_n^{'(m)}))+g(X_N^{'(m),n,A^*})\Big|\leq (N-n+1)\|f\|_{\infty}\Delta t+\|g\|_{\infty}$, we have
\begin{equation}
    \begin{aligned}
        &\mbox{Var}_M\bigg[-f(X_n^{'(m)},A^*(X_n^{'(m)}))\Delta t+\hat{Q}_n(X_n^{'(m)},A^*(X_n^{'(m)}))+g(X_N^{'(m),n,A^*})\bigg]\\
&=\mbox{Var}_M\Bigg[-f(X_n^{'(m)},A^*(X_n^{'(m)}))\Delta t+\hat{Q}_n(X_n^{'(m)},A^*(X_n^{'(m)}))+g(X_N^{'(m),n,A^*})-\frac{(N-n+1)\|f\|_{\infty}+\|g\|_{\infty}}{2}\Bigg]\,.
\end{aligned}
\end{equation}
When $-f(X_n^{'(m)},A^*(X_n^{'(m)}))\Delta t+\hat{Q}_n(X_n^{'(m)},A^*(X_n^{'(m)}))+g(X_N^{'(m),n,A^*})>0$, we have
\begin{equation}
\begin{aligned}
&Var_M\Bigg[-f(X_n^{'(m)},A^*(X_n^{'(m)}))\Delta t+\hat{Q}_n(X_n^{'(m)},A^*(X_n^{'(m)}))+g(X_N^{'(m),n,A^*})-\frac{(N-n+1)\|f\|_{\infty}+\|g\|_{\infty}}{2}\Bigg]\\
    &\leq \mathbb{E}\Bigg[\bigg(-f(X_n^{'(m)},A^*(X_n^{'(m)}))\Delta t+\hat{Q}_n(X_n^{'(m)},A^*(X_n^{'(m)}))+g(X_N^{'(m),n,A^*})-\frac{(N-n+1)\|f\|_{\infty}\Delta t+\|g\|_{\infty}}{2}\bigg)^2\Bigg]\\
    & (\mbox{Var}(X)=E(X^2)-(EX)^2\leq EX^2)\\
&\leq \Big(\frac{(N-n+1)\|f\|_{\infty}\Delta t+\|g\|_{\infty}}{2}\Big)^2\,.
    \end{aligned}
\end{equation}
If $-f(X_n^{'(m)},A^*(X_n^{'(m)}))\Delta t+\hat{Q}_n(X_n^{'(m)},A^*(X_n^{'(m)}))+g(X_N^{'(m),n,A^*})<0$, replace $-\frac{(N-n+1)\|f\|_{\infty}\Delta t+\|g\|_{\infty}}{2}$ by $+\frac{(N-n+1)\|f\|_{\infty}\Delta t+\|g\|_{\infty}}{2}$.

For $M>2\frac{((N-n+1)\|f\|_{\infty}+\|g\|_{\infty})^2}{\varepsilon^2}$, equation (\ref{Chebyshev’s inequality}) becomes
\begin{equation}
    \begin{aligned}
        \mathbb{P}_M\bigg[\Big|&\mathbb{E}_M[-f(X'_n,A^*(X'_n))\Delta t+\hat{Q}_n(X'_n,A^*(X'_n))+g(X_N^{'n,A^*})]\\
        &-\frac{1}{M}\sum_{m=1}^M[-f(X_n^{'(m)},A^*(X_n^{'(m)}))\Delta t+\hat{Q}_n(X_n^{'(m)},A^*(X_n^{'(m)}))+g(X_N^{'(m),n,A^*})]\Big|>\frac{\varepsilon}{2}\bigg]\\
&\leq \frac{1}{2}\,.
    \end{aligned}
\end{equation}
That is
\begin{equation}
    \begin{aligned}
        \mathbb{P}\bigg[\bigg|&\mathbb{E}_M[-f(X'_n,A^*(X'_n))\Delta t+\hat{Q}_n(X'_n,A^*(X'_n))+g(X_N^{'n,A^*})]\\
        &-\frac{1}{M}\sum_{m=1}^M[-f(X_n^{'(m)},A^*(X_n^{'(m)}))\Delta t+\hat{Q}_n(X_n^{'(m)},A^*(X_n^{'(m)}))+g(X_N^{'(m),n,A^*})]\bigg|\leq \frac{\varepsilon}{2}\bigg]\\
&\geq \frac{1}{2}\,.
    \end{aligned}
\end{equation}

Hence, the following holds,
\begin{equation}
    \begin{aligned}
        \mathbb{P}\Bigg[&\sup_{A\in \mathcal{A}_M}\bigg|\frac{1}{M}\sum_{m=1}^M\bigg[-f(X_n^{(m)},A(X_n^{(m)}))\Delta t+\hat{Q}_n(X_n^{(m)},A(X_n^{(m)})+g(X_N^{(m),n,A})\\
        &-(-f(X_n^{'(m)},A(X_n^{'(m)}))\Delta t+\hat{Q}_n(X_n^{'(m)},A(X_n^{'(m)})+g(X_N^{'(m),n,A}))\bigg]\bigg|>\frac{\varepsilon}{2}\Bigg]\\
\geq \mathbb{P}\Bigg[&\bigg|\frac{1}{M}\sum_{m=1}^M\big[-f(X_n^{(m)},A^*(X_n^{(m)}))\Delta t+\hat{Q}_n(X_n^{(m)},A^*(X_n^{(m)})+g(X_N^{(m),n,A^*})\\
&-(-f(X_n^{'(m)},A^*(X_n^{'(m)}))\Delta t+\hat{Q}_n(X_n^{'(m)},A^*(X_n^{'(m)})+g(X_N^{'(m),n,A^*}))\big]\bigg|>\frac{\varepsilon}{2}\Bigg]\\
\geq \mathbb{P}\Bigg[&\bigg|\frac{1}{M}\sum_{m=1}^M\big[-f(X_n^{(m)},A^*(X_n^{(m)}))\Delta t+\hat{Q}_n(X_n^{(m)},A^*(X_n^{(m)})+g(X_N^{(m),n,A^*})\big]\\
&-\mathbb{E}_M\big[-f(X_n,A^*(X_n))\Delta t+\hat{Q}_n(X_n,A^*(X_n)+g(X_N^{n,A^*})\big]\bigg|>\varepsilon,\\
&\bigg|\frac{1}{M}\sum_{m=1}^M\big[-f(X_n^{'(m)},A^*(X_n^{'(m)}))\Delta t+\hat{Q}_n(X_n^{'(m)},A^*(X_n^{'(m)})+g(X_N^{'(m),n,A^*}\big]\\
&-\mathbb{E}_M\big[-f(X_n,A^*(X_n))\Delta t+\hat{Q}_n(X_n,A^*(X_n)+g(X_N^{n,A^*})\big]\bigg|\leq\frac{\varepsilon}{2}\Bigg]\\
\geq \frac{1}{2}\mathbb{P}\Bigg[&\bigg|\frac{1}{M}\sum_{m=1}^M\big[-f(X_n^{(m)},A^*(X_n^{(m)}))\Delta t+\hat{Q}_n(X_n^{(m)},A^*(X_n^{(m)})+g(X_N^{(m),n,A^*})\big]\\
&-\mathbb{E}_M\big[-f(X_n,A^*(X_n))\Delta t+\hat{Q}_n(X_n,A^*(X_n)+g(X_N^{n,A^*})\big]\bigg|>\varepsilon\Bigg]\\
& (Due \,\,to\,\, independence \,\,of \,\,X_n^{(m)}\,\,and\,\, X_n^{'(m)})\,\,\\
\geq \frac{1}{2}\mathbb{P}\Bigg[&\sup_{A\in \mathcal{A}_M}\bigg|\frac{1}{M}\sum_{m=1}^M[-f(X_n^{(m)},A(X_n^{(m)}))\Delta t+\hat{Q}_n(X_n^{(m)},A(X_n^{(m)})+g(X_N^{(m),n,A})]\\
&-\mathbb{E}_M[-f(X_n,A(X_n))\Delta t+\hat{Q}_n(X_n,A(X_n)+g(X_N^{n,A})]\bigg|>\varepsilon\Bigg]\,.
    \end{aligned}
\end{equation}

That is
\begin{equation}
    \begin{aligned}
        \mathbb{P}\Bigg[&\sup_{A\in \mathcal{A}_M}\bigg|\frac{1}{M}\sum_{m=1}^M[-f(X_n^{(m)},A(X_n^{(m)}))\Delta t+\hat{Q}_n(X_n^{(m)},A(X_n^{(m)})+g(X_N^{(m),n,A})]\\
        &-\mathbb{E}_M[-f(X_n,A(X_n))\Delta t+\hat{Q}_n(X_n,A(X_n)+g(X_N^{n,A})]\bigg|>\varepsilon\Bigg]\\
\leq &2\mathbb{P}\Bigg[\sup_{A\in \mathcal{A}_M}\bigg|\frac{1}{M}\sum_{m=1}^M\big[-f(X_n^{(m)},A(X_n^{(m)}))\Delta t+\hat{Q}_n(X_n^{(m)},A(X_n^{(m)})+g(X_N^{(m),n,A})\\
&-(-f(X_n^{'(m)},A(X_n^{'(m)}))\Delta t+\hat{Q}_n(X_n^{'(m)},A(X_n^{'(m)})+g(X_N^{'(m),n,A}))\big]\bigg|>\frac{\varepsilon}{2}\Bigg]\,. \label{combine 1}
    \end{aligned}
\end{equation}

\textbf{Step 2}.    \quad Let $M'=\sqrt{2}\frac{(N-n+1)\|f\|_{\infty}\Delta t+\|g\|_{\infty}}{\sqrt{M}}$, then
\begin{equation}
    \begin{aligned}
        &\mathbb{E}\Bigg[\sup_{A\in \mathcal{A}_M}\Bigg|\frac{1}{M}\sum_{m=1}^M\Big[-f(X_n^{(m)},A(X_n^{(m)}))\Delta t+\hat{Q}_n(X_n^{(m)},A(X_n^{(m)})+g(X_N^{(m),n,A})\Big]\\
        &\qquad\qquad -\mathbb{E}_M\Big[-f(X_n,A(X_n))\Delta t+\hat{Q}_n(X_n,A(X_n)+g(X_N^{n,A})\Big]\Bigg|\Bigg]\\
&=\int_0^{\infty}\mathbb{P}\Bigg[\sup_{A\in \mathcal{A}_M}\bigg|\frac{1}{M}\sum_{m=1}^M\Big[-f(X_n^{(m)},A(X_n^{(m)}))\Delta t+\hat{Q}_n(X_n^{(m)},A(X_n^{(m)})+g(X_N^{(m),n,A})\Big]\\
 &\qquad\qquad-\mathbb{E}_M\Big[-f(X_n,A(X_n))\Delta t+\hat{Q}_n(X_n,A(X_n)+g(X_N^{n,A})\Big]\bigg|>\varepsilon\Bigg]d\varepsilon\\
&=\int_0^{M'}\mathbb{P}\Bigg[\sup_{A\in \mathcal{A}_M}\bigg|\frac{1}{M}\sum_{m=1}^M\big[-f(X_n^{(m)},A(X_n^{(m)}))\Delta t+\hat{Q}_n(X_n^{(m)},A(X_n^{(m)})+g(X_N^{(m),n,A})\big]\\
 &\qquad\qquad -\mathbb{E}_M\big[-f(X_n,A(X_n))\Delta t+\hat{Q}_n(X_n,A(X_n)+g(X_N^{n,A})\big]\bigg|>\varepsilon\Bigg]d\varepsilon\\
&\qquad+\int_{M'}^{\infty}\mathbb{P}\Bigg[\sup_{A\in \mathcal{A}_M}\bigg|\frac{1}{M}\sum_{m=1}^M\big[-f(X_n^{(m)},A(X_n^{(m)}))\Delta t+\hat{Q}_n(X_n^{(m)},A(X_n^{(m)})+g(X_N^{(m),n,A})\big]\\
 &\qquad\qquad-\mathbb{E}_M\big[-f(X_n,A(X_n))\Delta t+\hat{Q}_n(X_n,A(X_n)+g(X_N^{n,A})\big]\bigg|>\varepsilon\Bigg]d\varepsilon\\ 
&\leq \sqrt{2}\frac{(N-n+1)\|f\|_{\infty}\Delta t+\|g\|_{\infty}}{\sqrt{M}}
+4\int_{0}^{\infty}\mathbb{P}\Bigg[\sup_{A\in \mathcal{A}_M}\bigg|\frac{1}{M}\sum_{m=1}^M\big[-f(X_n^{(m)},A(X_n^{(m)}))\Delta t+\hat{Q}_n(X_n^{(m)},A(X_n^{(m)})\\
 &\qquad\qquad+g(X_N^{(m),n,A})-\big(-f(X_n^{'(m)},A(X_n^{'(m)}))\Delta t+\hat{Q}_n(X_n^{'(m)},A(X_n^{'(m)})+g(X_N^{'(m),n,A})\big)\big]\bigg|>\varepsilon\Bigg] \,.\\
  &\qquad(from\,\,(\ref{combine 1})\,\, and \,\,by \,\,variable \,\,substitution)\label{combine 2}
    \end{aligned}
\end{equation}

\textbf{Step 3. Introduce Rademacher random variables} \quad
Let $(r_m)_{1\leq m\leq M}$ be i.i.d. Rademacher random variables, which are independent uniform random variables taking values in $\{-1,\,\, +1\}$. For $\varepsilon>0$ the following holds
\begin{equation}
    \begin{aligned}
        &\mathbb{P}\Bigg[\sup_{A\in \mathcal{A}_M}\bigg|\frac{1}{M}\sum_{m=1}^M\bigg[-f(X_n^{(m)},A(X_n^{(m)}))\Delta t+\hat{Q}_n(X_n^{(m)},A(X_n^{(m)})+g(X_N^{(m),n,A})\\
        &\qquad\qquad-(-f(X_n^{'(m)},A(X_n^{'(m)}))\Delta t+\hat{Q}_n(X_n^{'(m)},A(X_n^{'(m)})+g(X_N^{'(m),n,A}))\bigg]\bigg|>\varepsilon\Bigg]\\
&\leq \mathbb{P}\Bigg[\sup_{A\in \mathcal{A}_M}\bigg|\frac{1}{M}\sum_{m=1}^Mr_m\bigg[-f(X_n^{(m)},A(X_n^{(m)}))\Delta t+\hat{Q}_n(X_n^{(m)},A(X_n^{(m)})+g(X_N^{(m),n,A})\\
&\qquad\qquad -(-f(X_n^{'(m)},A(X_n^{'(m)}))\Delta t+\hat{Q}_n(X_n^{'(m)},A(X_n^{'(m)})+g(X_N^{'(m),n,A}))\bigg]\bigg|>\varepsilon\Bigg]\\
&\leq \mathbb{P}\Bigg[\sup_{A\in \mathcal{A}_M}\bigg|\frac{1}{M}\sum_{m=1}^Mr_m\bigg[-f(X_n^{(m)},A(X_n^{(m)}))\Delta t+\hat{Q}_n(X_n^{(m)},A(X_n^{(m)})+g(X_N^{(m),n,A})\bigg]\bigg|>\frac{\varepsilon}{2}\Bigg]\\
&\quad +\mathbb{P}\Bigg[\sup_{A\in \mathcal{A}_M}\bigg|\frac{1}{M}\sum_{m=1}^Mr_m\bigg[-f(X_n^{'(m)},A(X_n^{'(m)}))\Delta t+\hat{Q}_n(X_n^{'(m)},A(X_n^{'(m)})+g(X_N^{'(m),n,A})\bigg]\bigg|>\frac{\varepsilon}{2}\Bigg]\\
&\leq 2\mathbb{P}\Bigg[\sup_{A\in \mathcal{A}_M}\bigg|\frac{1}{M}\sum_{m=1}^Mr_m\bigg[-f(X_n^{(m)},A(X_n^{(m)}))\Delta t+\hat{Q}_n(X_n^{(m)},A(X_n^{(m)})+g(X_N^{(m),n,A})\bigg]\bigg|>\frac{\varepsilon}{2}\Bigg]\,.
    \end{aligned}
\end{equation}

Through variable substitution, we have
\begin{equation}
    \begin{aligned}
        &\mathbb{E}\Bigg[\sup_{A\in \mathcal{A}_M}\bigg|\frac{1}{M}\sum_{m=1}^M\bigg[-f(X_n^{(m)},A(X_n^{(m)}))\Delta t+\hat{Q}_n(X_n^{(m)},A(X_n^{(m)})+g(X_N^{(m),n,A})\\
        &\qquad\qquad\qquad\qquad -(-f(X_n^{'(m)},A(X_n^{'(m)}))\Delta t+\hat{Q}_n(X_n^{'(m)},A(X_n^{'(m)})+g(X_N^{'(m),n,A}))\bigg]\bigg|\Bigg]\\
&\leq 4\mathbb{E}\Bigg[\sup_{A\in \mathcal{A}_M}\bigg|\frac{1}{M}\sum_{m=1}^Mr_m\bigg[-f(X_n^{(m)},A(X_n^{(m)}))\Delta t+\hat{Q}_n(X_n^{(m)},A(X_n^{(m)})+g(X_N^{(m),n,A})\bigg]\bigg|\Bigg]\,.  \label{combine 3}
    \end{aligned}
\end{equation}

\textbf{Step 4}.    \quad
Consider a process subject to control 0 at and after timestep $n$.
\begin{equation}
    \begin{aligned}
        &\mathbb{E}\Bigg[\sup_{A\in \mathcal{A}_M}\bigg|\frac{1}{M}\sum_{m=1}^Mr_m\big[-f(X_n^{(m)},A(X_n^{(m)}))\Delta t+\hat{Q}_n(X_n^{(m)},A(X_n^{(m)})+g(X_N^{(m),n,A})\big]\bigg|\Bigg]\\
&\leq \mathbb{E}\Bigg[\bigg|\frac{1}{M}\sum_{m=1}^Mr_m\big[-f(X_n^{(m)},0))\Delta t+\hat{Q}_n(X_n^{(m)},0)+g(X_N^{(m),n,0})\big]\bigg|\Bigg]\\
& +\mathbb{E}\Bigg[\sup_{A\in \mathcal{A}_M}\bigg|\frac{1}{M}\sum_{m=1}^Mr_m\big[-f(X_n^{(m)},A(X_n^{(m)})) \Delta t+\hat{Q}_n(X_n^{(m)},A(X_n^{(m)})+g(X_N^{(m),n,A})
 \\
 &\qquad\qquad-(-f(X_n^{(m)},0))\Delta t+\hat{Q}_n(X_n^{(m)},0)+g(X_N^{(m),n,0}))\big]\bigg|\Bigg]\,.\label{divide two parts}
    \end{aligned}
\end{equation}

For the first term of the right of (\ref{divide two parts}), by the Cauchy-Schwarz inequality, we have
\begin{equation}
    \begin{aligned}
        &\mathbb{E}\Bigg[\bigg|\frac{1}{M}\sum_{m=1}^Mr_m\big[-f(X_n^{(m)},0))\Delta t+\hat{Q}_n(X_n^{(m)},0)+g(X_N^{(m),n,0})\big]\bigg|\Bigg]\\
&\leq \frac{1}{M}\mathbb{E}\sqrt{\sum_{m=1}^Mr_m^2\cdot\sum_{m=1}^M\big[-f(X_n^{(m)},0))\Delta t+\hat{Q}_n(X_n^{(m)},0)+g(X_N^{(m),n,0})\big]^2}\\
&\leq \frac{1}{\sqrt{M}}\mathbb{E}\frac{\sum_{m=1}^M\big|-f(X_n^{(m)},0))\Delta t+\hat{Q}_n(X_n^{(m)},0)+g(X_N^{(m),n,0})\big|}{M}\\
&\leq \frac{1}{\sqrt{M}}\mathbb{E}\frac{M\mathop{max}\limits_{1\leq m\leq M}\big|-f(X_n^{(m)},0))\Delta t+\hat{Q}_n(X_n^{(m)},0)+g(X_N^{(m),n,0})\big|}{M}\\
&\leq \frac{(N-n+1)\|f\|_{\infty}\Delta t+\|g\|_{\infty}}{\sqrt{M}}\,. \label{the 1-th}
    \end{aligned}
\end{equation}

For the second term, by the Lipschitz continuity of $f$, we have
\begin{equation}
    \begin{aligned}
        &\mathbb{E}\Bigg[\sup_{A\in \mathcal{A}_M}\bigg|\sum_{m=1}^Mr_m\Big[-f(X_n^{(m)},A(X_n^{(m)}))\Delta t+\hat{Q}_n(X_n^{(m)},A(X_n^{(m)})+g(X_N^{(m),n,A})\\
        &\qquad\qquad -\big(-f(X_n^{(m)},0)\Delta t+\hat{Q}_n(X_n^{(m)},0)+g(X_N^{(m),n,0})\big)\Big]\bigg|\Bigg]\\
&\leq [f]_L\Delta t\mathbb{E}\Bigg[\sup_{A\in \mathcal{A}_M}\Big|\sum_{m=1}^Mr_mA(X_n^{(m)})\Big|\Bigg]+\mathbb{E}\Bigg[\sup_{A\in \mathcal{A}_M}\Big|\sum_{m=1}^Mr_m\Big(\hat{Q}_n(X_n^{(m)},A(X_n^{(m)})-\hat{Q}_n(X_n^{(m)},0)\Big)\Big|\Bigg]\\
&\qquad\qquad +\mathbb{E}\Bigg[\sup_{A\in\mathcal{A}_M}\bigg|\sum_{m=1}^Mr_m(g(X_N^{(m),n,A})-g(X_N^{(m),n,0}))\bigg|\Bigg]\\
&\leq ([f]_L\Delta t+\eta_M\gamma_M)\mathbb{E}\Bigg[\sup_{A\in \mathcal{A}_M}\Big|\sum_{m=1}^Mr_mA(X_n^{(m)})\Big|\Bigg]
+\mathbb{E}\Bigg[\sup_{A\in\mathcal{A}_M}\Big|\sum_{m=1}^Mr_m\big(g(X_N^{(m),n,A})-g(X_N^{(m),n,0})\big)\Big|\Bigg]\,.\label{g-g}
    \end{aligned}
\end{equation}

Moreover,
\begin{equation}
    \begin{aligned}
        \mathbb{E}[\sup_{A\in \mathcal{A}_M}|\sum_{m=1}^Mr_mA(X_k^{(m)})|]&\leq\gamma_M\mathbb{E}[\sup_{\|v\|_2\leq \frac{1}{R}}|\sum_{m=1}^Mr_m(v^TX_k^{(m)})_+|]\\
&\leq \gamma_M\mathbb{E}[\sup_{\|v\|_2\leq \frac{1}{R}}|\sum_{m=1}^Mr_m(v^TX_k^{(m)})|]\,,
    \end{aligned}
\end{equation}
where $R > 0$ is a bound for the state space(more details can be seen in \cite{bach2017breaking}). By the Cauchy-Schwarz inequality, we have
\begin{equation}
    \begin{aligned}
        \mathbb{E}[\sup_{A\in \mathcal{A}_M}|\sum_{m=1}^Mr_mA(X_k^{(m)})|]\leq\gamma_M\mathbb{E}\bigg[\sup_{\|v\|_2\leq \frac{1}{R}}\Big|\sum_{m=1}^Mr_m(v^TX_k^{(m)})\Big|\bigg]\leq \frac{\gamma_M}{R}\sqrt{\mathbb{E}\Big[\sum_{m=1}^Mr_mX_k^{(m)}\Big]^2}\leq \gamma_M\sqrt{M}\,.\label{bound of A}
    \end{aligned}
\end{equation}

For the last term in the right of the inequation (\ref{g-g})
\begin{equation}
    \begin{aligned}
        &|g(X_N^{(m),n,A})-g(X_N^{(m),n,0})|\leq[g]_L|X_N^{(m),n,A}-X_N^{(m),n,0}|
\leq [g]_L\rho_M\Big(|X_{N-1}^{(m),n,A}-X_{N-1}^{(m),n,0}|+|A(X_{N-1}^{(m),n,A})|\Big)\\
&\leq [g]_L\rho_M(|\rho_M(|X_{N-2}^{(m),n,A}-X_{N-2}^{(m),n,0}|+|A(X_{N-2}^{(m),n,A}|)+|A(X_{N-1}^{(m),n,A})|)\\
&\leq [g]_L(\Pi_{k=n+2}^N\rho_M(\rho_M(|X_n-X_n|+|A(X_n)|)+\Pi_{k=n+2}^N\rho_M|A(X_k^{(m),n,A}|+\cdots+\rho_M|A(X_{N-1}^{(m),n,A})|)\\
&=[g]_L(\rho_M^{N-n}|A(X_n^{(m)})|+\rho_M^{N-n-1}|A(X_k^{(m),n,A})|+\cdots+\rho_M|A(X_{N-1}^{(m),n,A})|)\,.
    \end{aligned}
\end{equation}

Then,
\begin{equation}
    \begin{aligned}
        &\mathbb{E}\bigg[\sup_{A\in\mathcal{A}_M}\Big|\sum_{m=1}^Mr_m(g(X_N^{(m),n,A})-g(X_N^{(m),n,0}))\Big|\bigg]\\
        & \leq [g]_L(\rho_M^{N-n}+\rho_M^{N-2}+\cdots+\rho_M)\gamma_M\sqrt{M}\\
&=[g]_L\frac{\rho_M-\rho_M^{N-n+1}}{1-\rho_M}\gamma_M\sqrt{M}\,.
    \end{aligned}
\end{equation}

Thus the following holds
\begin{equation}
    \begin{aligned}
        &\mathbb{E}\Bigg[\sup_{A\in \mathcal{A}_M}\bigg|\sum_{m=1}^Mr_m\Big[-f(X_n^{(m)},A(X_n^{(m)}))\Delta t+\hat{Q}_n(X_n^{(m)},A(X_n^{(m)})+g(X_N^{(m),n,A})\\
        &\qquad\qquad-\Big(-f(X_n^{(m)},0))\Delta t+\hat{Q}_n(X_n^{(m)},0)+g(X_N^{(m),n,0})\Big)\Big]\bigg|\Bigg]\\
&\leq ([f]_L\Delta t+\eta_M\gamma_M+[g]_L\frac{\rho_M-\rho_M^{N-n+1}}{1-\rho_M})\gamma_M\sqrt{M}\,.\label{the 2-th}
    \end{aligned}
\end{equation}

Plug (\ref{the 1-th}) and (\ref{the 2-th}) to (\ref{divide two parts}), we have
\begin{equation}
    \begin{aligned}
        \mathbb{E}\Bigg[\sup_{A\in \mathcal{A}_M}\bigg|\frac{1}{M}\sum_{m=1}^Mr_m\Big[-f(X_n^{(m)},A(X_n^{(m)}))\Delta t+\hat{Q}_n(X_n^{(m)},A(X_n^{(m)})+g(X_N^{(m),n,A})\Big]\bigg|\Bigg]\\
\leq \frac{(N-n+1)\|f\|_{\infty}\Delta t+\|g\|_{\infty}}{\sqrt{M}}+\bigg([f]_L\Delta t+\eta_M\gamma_M+[g]_L\frac{\rho_M-\rho_M^{N-n+1}}{1-\rho_M}\bigg)\frac{\gamma_M}{\sqrt{M}}\,. \label{combine 4}
    \end{aligned}
\end{equation}

\textbf{Step 5}.    \quad
Combine (\ref{combine 1}), (\ref{combine 2}),(\ref{combine 3}) and (\ref{combine 4}), the following holds
\begin{equation}
    \begin{aligned}
        \mathbb{E}\varepsilon^{esti}_{n+1}\leq (\sqrt{2}+16)\frac{(N-n+1)\|f\|_{\infty}\Delta t+\|g\|_{\infty}}{\sqrt{M}}+16\Big([f]_L\Delta t+\eta_M\gamma_M+[g]_L\frac{\rho_M-\rho_M^{N-n+1}}{1-\rho_M}\Big)\frac{\gamma_M}{\sqrt{M}}\,.
    \end{aligned}
\end{equation}

This completes the proof of Lemma 1.

\section{Appendix 5} \label{proof of lemma 2}
The \textbf{proof of Lemma \ref{lemma 2}}.

    We divide the error into two parts
    \begin{equation}
        \begin{aligned}
            \varepsilon_{n+1}^{approx}=&\inf_{A\in \mathcal{A}_M}\big[\mathbb{E}_M[-f(X_n,A(X_n))\Delta t+\hat{Q}_n(X_n,A(X_n))+g(X_N^{n,A})]\\
            &\qquad-\mathbb{E}_M[-f(X_n,a^{opt}(X_n))\Delta t+Q_{n}(X_n,a^{opt}(X_n))]\\
&+\mathbb{E}_M[-f(X_n,a^{opt}(X_n))\Delta t+Q_{n}(X_n,a^{opt}(X_n))]\\
&\qquad -\inf_{A\in\mathbb{A}^\mathcal{X}}\mathbb{E}_M[-f(X_n,A(X_n))\Delta t+\hat{Q}_n(X_n,A(X_n))+g(X_N^{n,A})]\,.\label{approx two parts}
        \end{aligned}
    \end{equation}

\textbf{Step 1}. \quad
For the first term, take $A\in \mathcal{A}_M$,
\begin{equation}
    \begin{aligned}
        &\mathbb{E}_M\bigg[-f(X_n,A(X_n))\Delta t+\hat{Q}_n(X_n,A(X_n))+g(X_N^{n,A})\bigg]-\mathbb{E}_M\Big[-f(X_n,a^{opt}(X_n))\Delta t+Q_{n}(X_n,a^{opt}(X_n))\Big]\\
&\leq [f]_L\Delta t\mathbb{E}_M|A(X_n)-a^{opt}(X_n)|+\mathbb{E}_M\Big[\Big|\hat{Q}_n(X_n,A(X_n))+g(X_N^{n,A})-Q_n(X_n,a^{opt}(X_n)\Big|\Big]\,.
\label{1-th part}
    \end{aligned}
\end{equation}

\textbf{Step 2}. \quad
For the second term,
\begin{equation}
    \begin{aligned}
        & \mathbb{E}_M\big[-f(X_n,a^{opt}(X_n))\Delta t+Q_{n}(X_n,a^{opt}(X_n))\big]-\inf_{A\in\mathbb{A}^\mathcal{X}}\mathbb{E}_M\big[-f(X_n,A(X_n))\Delta t+\hat{Q}_n(X_n,A(X_n))+g(X_N^{n,A})\big]\\
&\leq  \inf_{A\in\mathbb{A}^\mathcal{X}}\Big\{[f]_L\Delta t\mathbb{E}_M|A(X_n)-a^{opt}(X_n)|+\mathbb{E}_M\Big[\Big|\hat{Q}_n(X_n,A(X_n))+g(X_N^{n,A})-Q_n(X_n,a^{opt}(X_n)\Big|\Big]\Big\}\\
&\leq  \inf_{A\in\mathcal{A}_M}\Big\{[f]_L\Delta t\mathbb{E}_M|A(X_n)-a^{opt}(X_n)|+\mathbb{E}_M\Big[\Big|\hat{Q}_n(X_n,A(X_n))+g(X_N^{n,A})-Q_n(X_n,a^{opt}(X_n)\Big|\Big]\Big\}\,.\\
&\quad(the \,\,class\,\, of\,\, \mathcal{A}_M\,\, is\,\, not\,\, dense\,\, in\,\, the\,\, set\,\, \mathbb{A}^X\,\, of\,\, all\,\, Borelian\,\, functions)
\label{2-th part}
    \end{aligned}
\end{equation}

\textbf{Step 3}. \quad
Plugging (\ref{1-th part}) and (\ref{2-th part}) into (\ref{approx two parts}), we have
\begin{equation}
    \begin{aligned}
        \varepsilon_{n+1}^{approx}\leq &\inf_{A\in\mathcal{A}_M}\Big\{2[f]_L\Delta t\mathbb{E}_M\Big[\Big|A(X_n)-a^{opt}(X_n)\Big|\Big]+2\mathbb{E}_M\Big[\Big|\hat{Q}_n(X_n,A(X_n))+g(X_N^{n,A})-Q_n(X_n,a^{opt}(X_n))\Big|\Big]\Big\}\,.
    \end{aligned}
\end{equation}

\end{appendices}

\end{document}